\documentclass[a4paper,10pt]{article}

\usepackage[english]{babel}
\usepackage{amsmath,amsfonts,amssymb,amsthm}
\usepackage{mathrsfs}
\usepackage{indentfirst}
\usepackage{graphicx}
\usepackage{relsize}
\usepackage[all]{xy}
\usepackage{pgf,tikz}

\newtheorem{theorem}{Theorem}[section]

\newtheorem{lemma}{Lemma}[section]
\newtheorem{proposition}{Proposition}[section]

\theoremstyle{definition}
\newtheorem{remark}{Remark}[section]

\numberwithin{equation}{section}
\newcommand\blfootnote[1]{\begingroup\renewcommand\thefootnote{}\footnote{#1}\addtocounter{footnote}{-1}\endgroup}

\begin{document}

\title{
{\bf\Large
Positive solutions for super-sublinear indefinite problems:
high multiplicity results via coincidence degree}}

\author{
\vspace{1mm}
\\
{\bf\large Alberto Boscaggin}
\vspace{1mm}\\
{\it\small Department of Mathematics, University of Torino}\\
{\it\small via Carlo Alberto 10}, {\it\small 10123 Torino, Italy}\\
{\it\small e-mail: alberto.boscaggin@unito.it}\vspace{1mm}\\
\vspace{1mm}\\
{\bf\large Guglielmo Feltrin}
\vspace{1mm}\\
{\it\small SISSA - International School for Advanced Studies}\\
{\it\small via Bonomea 265}, {\it\small 34136 Trieste, Italy}\\
{\it\small e-mail: guglielmo.feltrin@sissa.it}\vspace{1mm}\\
\vspace{1mm}\\
{\bf\large Fabio Zanolin}
\vspace{1mm}\\
{\it\small Department of Mathematics and Computer Science, University of Udine}\\
{\it\small via delle Scienze 206},
{\it\small 33100 Udine, Italy}\\
{\it\small e-mail: fabio.zanolin@uniud.it}\vspace{1mm}}

\date{}

\maketitle

\vspace{-2mm}

\begin{abstract}
\noindent
We study the periodic boundary value problem associated with the second order nonlinear equation
\begin{equation*}
u'' + \bigr{(} \lambda a^{+}(t) - \mu a^{-}(t) \bigr{)} g(u) = 0,
\end{equation*}
where $g(u)$ has superlinear growth at zero and sublinear growth at infinity.
For $\lambda, \mu$ positive and large, we prove the existence of $3^{m}-1$ positive
$T$-periodic solutions when the weight function $a(t)$ has $m$ positive humps
separated by $m$ negative ones (in a $T$-periodicity interval). As a byproduct of our approach we also provide
abundance of positive subharmonic solutions and symbolic dynamics.
The proof is based on coincidence degree theory for locally compact operators
on open unbounded sets and also applies to Neumann and Dirichlet boundary conditions.
Finally, we deal with radially symmetric positive solutions for the Neumann and the Dirichlet problems
associated with elliptic PDEs.
\end{abstract}

\section{Introduction and statement of the main result}\label{section-1}

In this paper, we present some multiplicity results for \textit{positive} solutions to boundary value problems associated
with nonlinear differential equations of the type
\begin{equation}\label{eq-intro}
u'' + q(t) g(u) = 0,
\end{equation}
where $q(t)$ is a sign-changing weight function and $g(s)$ is a function with superlinear growth at zero, sublinear growth at infinity
and positive on $\mathopen{]}0,+\infty\mathclose{[}$.
Due to these assumptions, we refer to \eqref{eq-intro} as a \textit{super-sublinear indefinite} problem.
The terminology ``indefinite'', meaning that $q(t)$ is of non-constant sign, was probably introduced in \cite{AtEvOn-74} dealing with a linear eigenvalue problem and,
starting with \cite{HeKa-80}, it has become very popular also in nonlinear differential problems
(especially when $g(s)$ is a superlinear function, for instance as $g(s) \sim s^{p}$ with $p>1$,
so that \eqref{eq-intro} is said to be superlinear indefinite, see \cite{BaBoVe-15, BeCaDoNi-94, CaDaPa-02, GoReLoGo-00}).

We now describe our setting in more detail.
Denoting by $\mathbb{R}^{+} := \mathopen{[}0,+\infty\mathclose{[}$ the set of non-negative real numbers,
we assume that $g \colon \mathbb{R}^{+} \to \mathbb{R}^{+}$ is a continuous function satisfying the sign hypothesis
\begin{equation*}
g(0) = 0, \qquad g(s) > 0 \quad \text{for } \; s > 0,
\leqno{(g_{*})}
\end{equation*}
as well as the conditions of superlinear growth at zero
\begin{equation*}
\lim_{s\to 0^{+}} \dfrac{g(s)}{s} = 0
\leqno{(g_{0})}
\end{equation*}
and sublinear growth at infinity
\begin{equation*}
\lim_{s\to +\infty} \dfrac{g(s)}{s} = 0.
\leqno{(g_{\infty})}
\end{equation*}
Concerning the weight function $q(t)$, we find convenient to write it as
\begin{equation*}
q(t) = a_{\lambda,\mu}(t) := \lambda a^{+}(t) - \mu a^{-}(t),
\end{equation*}
where $a \in L^{1}(\mathopen{[}0,T\mathclose{]})$ is a sign-changing function, that is
\begin{equation*}
\int_{0}^{T} a^{+}(t)~\!dt \neq 0 \neq \int_{0}^{T} a^{-}(t)~\!dt,
\end{equation*}
and $\lambda,\mu > 0$ are real parameters. Summing up, we deal with the equation
\begin{equation}\label{eq-main}
u'' + \bigl{(} \lambda a^{+}(t) - \mu a^{-}(t) \bigr{)} g(u) = 0
\end{equation}
and we investigate multiplicity of positive solutions (in the Carath\'{e}odory sense, see \cite{Ha-80}) to \eqref{eq-main} in dependence of the parameters $\lambda,\mu > 0$.

Results in this direction have already appeared in the literature. When \eqref{eq-main} is considered together with Dirichlet
boundary conditions $u(0) = u(T) = 0$, for instance, it is well known that two positive solutions exist if $\lambda > 0$ is large
enough and for any value $\mu > 0$. This is a classical result, on a line of research initiated by
Rabinowitz in \cite{Ra-7374} (dealing with the Dirichlet problem associated with a super-sublinear elliptic PDE on a bounded
domain, see also \cite{Am-72} for previous related results) and later developed by many authors. Actually, typical versions of this theorem do not take into account an
indefinite weight function (that is, they are stated for $a^{-}\equiv 0$ in \eqref{eq-main}), but nowadays standard tools (such as
critical point theory, fixed point theorems for operators on cones, dynamical systems techniques) permit to successfully handle
also this more general situation. We refer to \cite{BoZa-13} for the precise statement in the indefinite setting as well as to the
introductions in \cite{BoFeZa-15,BoZa-12} for a more complete presentation and bibliography on the subject.

As far as Neumann boundary conditions $u'(0) = u'(T) = 0$ or $T$-periodic boundary conditions
$u(T) - u(0) = u'(T)-u'(0) = 0$ are taken into account, the problem becomes slightly more subtle.
Indeed, on one hand, the indefinite character of the problems plays a crucial role, since no positive Neumann/periodic solutions
to \eqref{eq-main} can exist if $a^{-} \equiv 0$ or if $a^{+} \equiv 0$, as it is easily seen by integrating the equation on $\mathopen{[}0,T\mathclose{]}$.
On the other hand, some restrictions on the ranges of the parameters $\lambda,\mu > 0$ also appear.
Precisely, as already observed in previous papers \cite{BaPoTe-88,BoZa-12}, whenever $g'(s) > 0$ for any $s > 0$,
a necessary condition for the existence of positive Neumann/periodic solutions to \eqref{eq-main} turns out to be
\begin{equation*}
\int_{0}^{T} a_{\lambda,\mu}(t)~\!dt < 0,
\end{equation*}
which equivalently reads as
\begin{equation}\label{mudiesis}
\mu > \mu^{\#}(\lambda) := \lambda \, \dfrac{\int_{0}^{T} a^{+}(t)~\!dt}{\int_{0}^{T} a^{-}(t)~\!dt}.
\end{equation}
Hence, contrarily to the Dirichlet problem, the existence of positive solutions cannot be ensured for any $\mu > 0$.
However, under slightly more restrictive assumptions than $(g_{0})$ and $(g_{\infty})$
(like, for instance, $g(s) \sim s^{\alpha}$ with $\alpha > 1$ at zero and $g(s) \sim s^{\beta}$ with $0 < \beta < 1$ at infinity),
the existence of two positive Neumann/periodic solutions
to \eqref{eq-main} is still guaranteed for $\lambda > 0$ large enough and $\mu$ satisfying \eqref{mudiesis}.
This was shown in \cite{BoZa-12} using critical point theory and in \cite{BoFeZa-15} using a topological degree argument
(this last proof working for the damped equation $u'' + cu' + a_{\lambda,\mu}(t) g(u) = 0$, as well).
In both the approaches condition \eqref{mudiesis} plays the role of pushing the nonlinearity
below the principal eigenvalue $k_{0} = 0$ of the Neumann/periodic problem both at zero and at infinity
(notice that this is not needed if Dirichlet boundary conditions are taken into account, since the first eigenvalue is strictly positive).

The above recalled results seem to be optimal from the point of view of the multiplicity of solutions, in the sense that
no more than two positive solutions can be expected for a general weight.
In this regard, sharp existence results of exactly two solutions (at least for the Dirichlet problem
and with a positive constant weight) are described and surveyed in \cite{Li-82, OuSh-98, OuSh-99}
(more specifically, see \cite[Theorem~6.19]{OuSh-99}).

The aim of the present paper is to show that, on the other hand, many positive solutions for the Dirichlet/Neumann/periodic boundary value problems associated
with \eqref{eq-main} can be obtained by playing
with the nodal behavior of the weight function: roughly speaking, we will require it to have $m$ positive humps, together with a large negative part
(that is, $\mu \gg 0$).

\bigskip

We now focus on the $T$-periodic boundary value problem associated with \eqref{eq-main}
and we proceed to state our main result more precisely, as follows.

Let $a \colon \mathbb{R} \to \mathbb{R}$ be a $T$-periodic locally integrable function and suppose that
\begin{itemize}
\item [$(a_{*})$]
\textit{there exist $2m + 1$ points (with $m \geq 1$)
\begin{equation*}
\sigma_{1} < \tau_{1} < \ldots < \sigma_{i} < \tau_{i} < \ldots < \sigma_{m} < \tau_{m} < \sigma_{m+1}, \quad \text{with } \; \sigma_{m+1}-\sigma_{1} = T,
\end{equation*}
such that, for $i=1,\ldots,m$, $a(t) \succ 0$ on $\mathopen{[}\sigma_{i},\tau_{i}\mathclose{]}$ and $a(t) \prec 0$ on $\mathopen{[}\tau_{i},\sigma_{i+1}\mathclose{]}$,}
\end{itemize}
where, following a standard notation, $w(t) \succ 0$ on a given interval means that
$w(t)\geq 0$ almost everywhere with $w\not\equiv 0$ on that interval; moreover, $w(t) \prec 0$ stands for $-w(t) \succ 0$.
Without loss of generality, due to the $T$-periodicity of the function $a(t)$, in the sequel we assume that $\sigma_{1} = 0$ and $\sigma_{m+1} = T$.
We also set, for $i=1,\ldots,m$,
\begin{equation}\label{Ipm}
I^{+}_{i} := \mathopen{[}\sigma_{i},\tau_{i}\mathclose{]} \quad \text{ and } \quad I^{-}_{i} := \mathopen{[}\tau_{i},\sigma_{i+1}\mathclose{]}.
\end{equation}
We look for solutions $u(t)$ of \eqref{eq-main} (in the Carath\'{e}odory sense) which are globally defined on $\mathbb{R}$ with $u(t) = u(t+T) > 0$ for all $t\in {\mathbb{R}}$.
Such solutions will be referred to as \textit{positive $T$-periodic}.
Then, the following result holds true.

\begin{theorem}\label{main-theorem}
Let $g \colon \mathbb{R}^{+} \to \mathbb{R}^{+}$ be a continuous function satisfying $(g_{*})$, $(g_{0})$ and $(g_{\infty})$.
Let $a \colon \mathbb{R} \to \mathbb{R}$ be a $T$-periodic locally integrable function satisfying $(a_{*})$.
Then there exists $\lambda^{*} > 0$ such that for each $\lambda > \lambda^{*}$ there exists $\mu^{*}(\lambda) > 0$ such that
for each $\mu > \mu^{*}(\lambda)$ equation \eqref{eq-main} has at least $3^{m}-1$ positive $T$-periodic solutions.

More precisely, fixed an arbitrary constant $\rho > 0$ there exists $\lambda^{*} = \lambda^{*}(\rho) > 0$ such that
for each $\lambda > \lambda^{*}$ there exist two constants $r,R$ with $0 < r < \rho < R$ and $\mu^{*}(\lambda) =
\mu^{*}(\lambda,r,R)>0$ such that for any $\mu > \mu^{*}(\lambda)$ and
any finite string $\mathcal{S} = (\mathcal{S}_{1},\ldots,\mathcal{S}_{m}) \in \{0,1,2\}^{m}$, with $\mathcal{S} \neq (0,\ldots,0)$,
there exists a positive $T$-periodic solution $u(t)$ of \eqref{eq-main} such that
\begin{itemize}
\item $\max_{t \in I^{+}_{i}} u(t) < r$, if $\mathcal{S}_{i} = 0$;
\item $r < \max_{t \in I^{+}_{i}} u(t) < \rho$, if $\mathcal{S}_{i} = 1$;
\item $\rho < \max_{t \in I^{+}_{i}} u(t) < R$, if $\mathcal{S}_{i} = 2$.
\end{itemize}
\end{theorem}

\begin{remark}\label{rem-1.1}
As already anticipated, the same multiplicity result holds true for the Neumann as well as for the Dirichlet problems associated with
\eqref{eq-main} on the interval $\mathopen{[}0,T\mathclose{]}$. Dealing with these boundary value problems, the weight function
$a(t)$ is allowed to be negative on a right neighborhood of $0$ and/or positive on a left neighborhood of $T$.
Indeed, what is crucial to obtain $3^{m} - 1$ positive solutions
is the fact that there are $m$ positive humps of the weight function which are separated by negative ones.
Accordingly, if we study the Neumann or the Dirichlet problems on $\mathopen{[}0,T\mathclose{]}$ it will be sufficient to suppose that there are
$m-1$ intervals where $a(t)\prec 0$ separating $m$ intervals where $a(t)\succ 0$.
On the other hand, the nature of periodic boundary conditions requires that the positive humps of the
weight coefficient are separated by negative humps on $\mathopen{[}0,T\mathclose{]}/\{0,T\} \simeq \mathbb{R}/T\mathbb{Z}\simeq S^{1}$.
This is the reason for which condition $(a_{*})$ for the periodic problem is conventionally expressed
assuming that, in an interval of length $T$, the weight function starts positive and ends negative.
For a more detailed discussion, see Section~\ref{section-7.2}.
$\hfill\lhd$
\end{remark}

Let us now make some comments about Theorem~\ref{main-theorem}, trying at first to explain its meaning in an informal way.
The existence of $3^{m}-1$ positive solutions comes from the possibility of prescribing, for a positive $T$-periodic solution of \eqref{eq-main},
the behavior in each interval of positivity of the weight function $a(t)$ among three possible ones: either the solution
is ``very small'' on $I^{+}_{i}$ (if $\mathcal{S}_{i} = 0$), or it is ``small'' (if $\mathcal{S}_{i} = 1$)
or it is ``large'' (if $\mathcal{S}_{i} = 2$). This is related to the fact that, as discussed at the beginning of this introduction,
three non-negative solutions for the Dirichlet problem associated with $u'' + \lambda a^{+}(t)g(u) = 0$
on $I^{+}_{i}$ are always available, when $g(s)$ is super-sublinear, for $\lambda > 0$ large enough:
the trivial one, and two positive solutions given by Rabinowitz's theorem (cf.~\cite{Ra-7374}).
This point of view can be made completely rigorous by showing that
the solutions constructed in Theorem~\ref{main-theorem} converge, for $\mu \to +\infty$, to solutions of the Dirichlet problem associated with
$u'' + \lambda a^{+}(t)g(u) = 0$ on each $I^{+}_{i}$ and to zero on $\bigcup_{i}I^{-}_{i}$
(see the second part of Section~\ref{section-5} for a detailed discussion).
With this is mind, one can interpret Theorem~\ref{main-theorem} as a singular perturbation result
from the limit case $\mu = + \infty$. Indeed, by taking into account all the possibilities for the non-negative solutions of
the Dirichlet problem associated with $u'' + \lambda a^{+}(t)g(u) = 0$ on each $I^{+}_{i}$, one finds $3^{m}$ limit profiles for positive solutions
to \eqref{eq-main}. Among them, $3^{m}-1$ are non-trivial and give rise, for $\mu \gg 0$, to $3^{m}-1$ positive $T$-periodic solutions to \eqref{eq-main}, while the
trivial limit profile still persists as the trivial solution to \eqref{eq-main} for any $\mu > 0$.
Figure~\ref{fig-01} below illustrates an example of existence of eight positive solutions for the Dirichlet problem when the weight function
possesses two positive humps separated by a negative one.

\begin{figure}[h!]
\centering
\begin{tikzpicture} [scale=1]
\node at (0,5.1) {};
\node at (-3,3) {\includegraphics[width=0.41\textwidth]{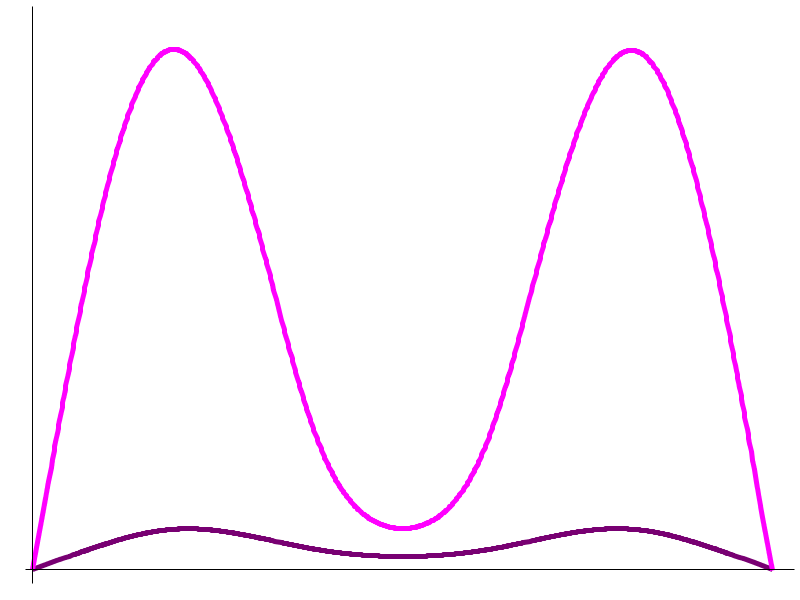}};
\node at (3,3) {\includegraphics[width=0.41\textwidth]{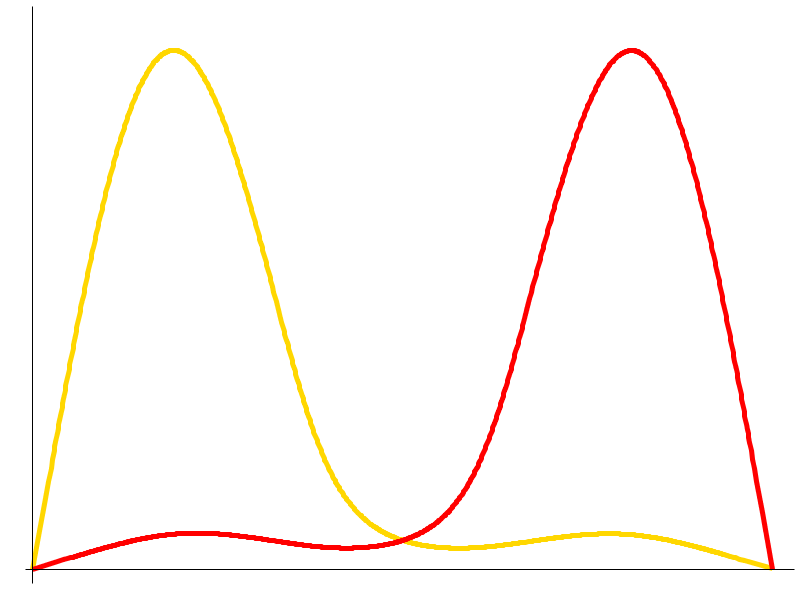}};
\node at (-3,-1.4) {\includegraphics[width=0.41\textwidth]{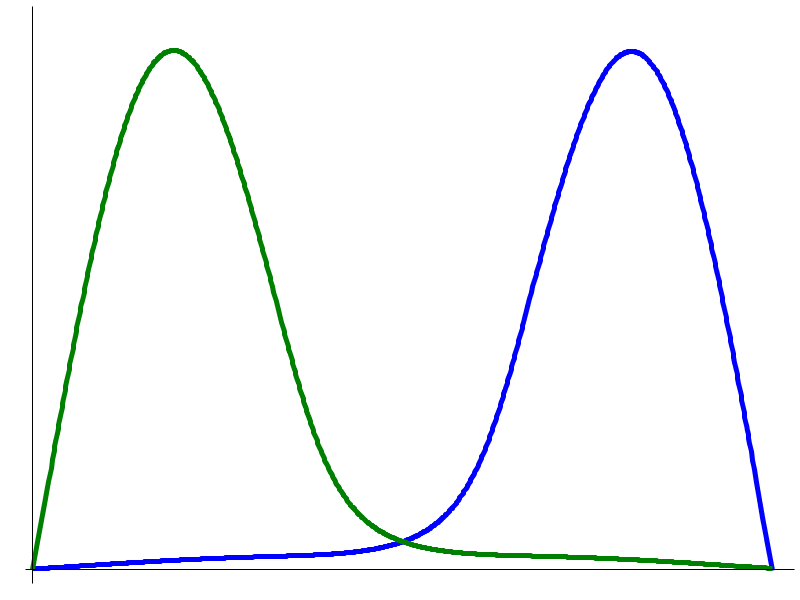}};
\node at (3,-1.4) {\includegraphics[width=0.41\textwidth]{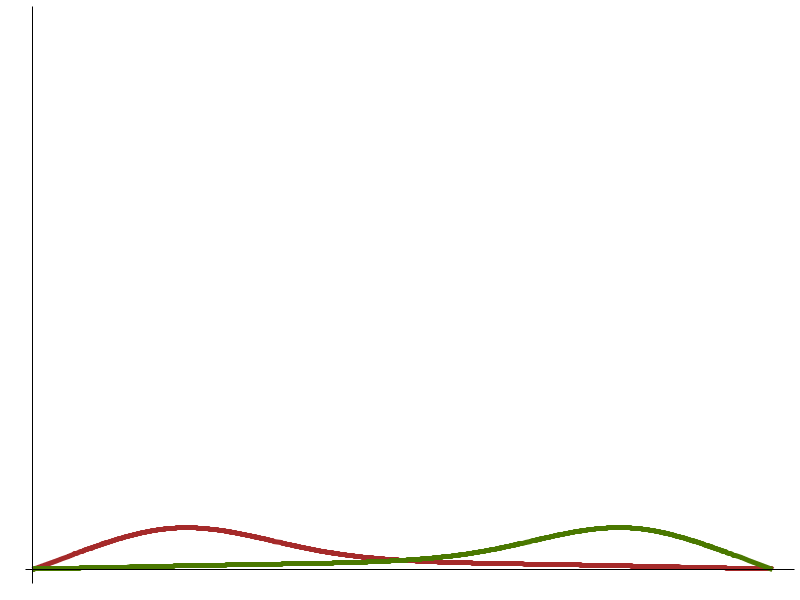}};
\end{tikzpicture}
\caption{\small{The figure shows an example of $8 = 3^{2}-1$ positive solutions to the Dirichlet problem
for the super-sublinear nonlinearity $g(s) = s^{2}/(1+s^{2})$.
For this simulation we have chosen the interval $\mathopen{[}0,T\mathclose{]}$ with $T=3\pi$ and the weight function
$a_{\lambda,\mu}(t) := \lambda \, \sin^{+}(t) - \mu \, \sin^{-}(t)$, so that
$m=2$ is the number of positive humps separated by a negative one. Evidence of multiple
positive solutions (agreeing with Theorem~\ref{th-7.1}) is obtained for $\lambda=3$ and $\mu=10$.
The subfigures (to be read in the natural order left-right and top-bottom)
show pairs of solutions according to the following codes:
$(2,2)$ and $(1,1)$, $(2,1)$ and $(1,2)$, $(2,0)$ and $(0,2)$, $(1,0)$ and $(0,1)$.}}
\label{fig-01}
\end{figure}

What may appear as a relevant aspect of our result is the fact that a minimal set of assumptions
on the nonlinearity $g(s)$ is required. Indeed, only positivity, continuity and the hypotheses
on the limits $g(s)/s$ for $s \to 0^{+}$ and $s \to +\infty$ are required. In particular,
no supplementary power-type growth conditions at zero or at infinity are needed. In the recent paper \cite{BoFeZa-15}
we obtained the existence of at least two positive $T$-periodic solutions under the sharp condition \eqref{mudiesis}
on the coefficient $\mu$; on the other hand, in the same paper some extra (although mild) assumptions on $g(s)$ were imposed.
It is interesting to observe that increasing the value of $\mu$ yields both abundance of solutions and no-extra assumptions on $g(s)$.

The possibility of finding multiple positive solutions of indefinite nonlinear problems
by playing with the nodal behavior of the weight function was at first suggested in a paper by G{\'o}mez-Re{\~n}asco and L{\'o}pez-G{\'o}mez
\cite{GoReLoGo-00}, in analogy with the celebrated papers by Dancer \cite{Da-88,Da-90} providing multiplicity of solutions to elliptic BVPs by playing with the shape of the domain.
In particular, it was then proved in \cite{GaHaZa-03, GaHaZa-04} that the Dirichlet boundary value problem associated with the superlinear indefinite equation
\begin{equation*}
u'' + \bigl{(} a^{+}(t) - \mu a^{-}(t) \bigr{)} u^{p} = 0, \quad \text{with } \; p > 1,
\end{equation*}
has $2^{m}-1$ positive solutions (with $m$ again being the number of negative humps of $a(t)$) when $\mu$ is large.
This result has later been extended in various directions, so as to cover also the case of an elliptic
PDE (cf.~\cite{BoGoHa-05, GiGo-09}), more general nonlinearities (cf.~\cite{FeZa-jde2015, GiGo-09Ed}) as well as Neumann/periodic boundary conditions
(cf.~\cite{BaBoVe-15,Bo-11, FeZa-ade2015, FeZa-pp2015}).
The fact that in the superlinear case less solutions, with respect to Theorem~\ref{main-theorem}, are available
is not surprising, since in general no more than one positive solution can be expected for the Dirichlet problem
associated with $u'' + a^{+}(t)u^{p} = 0$ on the interval $I^{+}_{i}$ (the uniqueness is simple to check at least for
$a^{+}\equiv 1$). On the other hand, the parameter $\lambda$ in front of the
positive part of the weight function is not necessary to ensure existence: indeed, the superlinear growth at infinity
plays here the same role as the largeness of $\lambda$ in the super-sublinear case.
Referring to Theorem~\ref{main-theorem}, we can thus say that the $2^{m}-1$ solutions associated
with strings $\mathcal{S}$ with $\mathcal{S}_{i} \in \{0,1\}$ correspond to the solutions already available for superlinear problems,
while all the other ones (that is, $\mathcal{S}_{i} = 2$ for at least an index $i$) are typical of super-sublinear nonlinearities.

An important feature of Theorem~\ref{main-theorem} is that all the constants appearing in the statement (precisely $\lambda^{*}$, $r$, $R$ and $\mu^{*}(\lambda)$)
can be explicitly estimated (depending on $g(s)$, $a(t)$, as well as on the arbitrary choice of $\rho$).
In particular, it turns out
that, whenever Theorem~\ref{main-theorem} is applied to an interval of the form
$\mathopen{[}0,kT\mathclose{]}$, with $k \geq 1$ an integer number, these constants can be chosen independently on $k$.
This implies that, for any fixed $\lambda>\lambda^{*}$ and for any $\mu>\mu^{*}(\lambda)$, equation \eqref{eq-main}
has positive $T$-periodic solutions as well as positive $kT$-periodic solutions for any $k\geq2$.
Such solutions can of course be coded similarly as the $T$-periodic ones, by prescribing their behavior on the intervals
\begin{equation*}
I^{+}_{i,\ell} := I^{+}_{i} + \ell \, T, \quad \text{for } \; i=1,\ldots,m \; \text{ and } \; \ell \in \mathbb{Z},
\end{equation*}
according to a non-null bi-infinite $km$-periodic string $\mathcal{S}$
in the alphabet $\mathscr{A}:= \{0,1,2\}$ of $3$ symbols (see Theorem~\ref{main-theorem-sub}).
This information can be used to prove that many of these positive $kT$-periodic solutions
have $kT$ as minimal period, namely they are \textit{subharmonic solutions} of order $k$
(see Theorem~\ref{main-theorem-sub2}, where a lower bound based on the combinatorial concept of \textit{Lyndon words} is given).
Finally, using an approximation argument of Krasnosel'ski\u{\i}-Mawhin type (cf.~\cite{Kr-68, Ma-96}) for $k\to\infty$,
it is possible to construct globally defined bounded
positive solutions to \eqref{eq-main}, whose behavior on each $I^{+}_{i,\ell}$ can be prescribed a priori with a nontrivial bi-infinite
string $\mathcal{S} \in \mathscr{A}^{\mathbb{Z}}$ and thus exhibiting \textit{chaotic-like dynamics} (see Theorem~\ref{th-6.4}).
In this way we can improve the main result in \cite{BoZa-12b}, where arguments from topological horseshoes theory were
used to construct a symbolic dynamics on two symbols ($1$ and $2$, according to the notation of the present paper).

For the proof of Theorem~\ref{main-theorem} and its variants,
we use a topological degree approach looking for solutions to an operator equation of the form
\begin{equation}\label{eq-1.6}
Lu = N_{\lambda,\mu} u, \quad u \in \text{\rm dom}\,L,
\end{equation}
where $L$ is the differential operator $u \mapsto - u''$ subject to the boundary conditions
and $N_{\lambda,\mu}$ is the Nemytskii operator induced by a suitable extension, defined for all $s\in \mathbb{R}$,
of $a_{\lambda,\mu}(t) g(s)$.
Once we have chosen an appropriate pair of spaces $X, Z$ such that $L \colon \text{\rm dom}\,L \,(\subseteq X) \to Z$
and $N_{\lambda,\mu} \colon X \to Z$, we transform equation \eqref{eq-1.6} into an
equivalent fixed point problem of the form
\begin{equation*}
u = \Phi_{\lambda,\mu} u, \quad u \in X,
\end{equation*}
with $\Phi_{\lambda,\mu}$ a completely continuous operator acting on $X$. In the case of the Dirichlet problem,
the linear operator $L$ is invertible and thus $\Phi_{\lambda,\mu} = L^{-1}N_{\lambda,\mu}$,
while for periodic and Neumann boundary conditions we follow the approach introduced by J.~Mawhin in \cite{Ma-69}
for the definition of the coincidence degree. The crucial steps in the proofs consist in defining some special open
and unbounded sets $\Lambda^{\mathcal{I},\mathcal{J}}\subseteq X$ and in computing $\text{\rm deg}_{LS}(Id - \Phi_{\lambda,\mu},\Lambda^{\mathcal{I},\mathcal{J}},0)$,
where ``$\text{\rm deg}_{LS}$'' denotes the
Leray-Schauder degree for locally compact operators (cf.~\cite{Gr-72, Nu-85, Nu-93}).
In the definition of these open sets,
$\mathcal{I}$ and $\mathcal{J}$ are prescribed disjoint sets of indices contained in $\{1,\ldots,m\}$ and
$u\in \Lambda^{\mathcal{I},\mathcal{J}}$ provided that $u(t)$ is ``very small'', ``small'' or ``large'' on the intervals $I^{+}_{i}$
when $i\notin \mathcal{I}\cup \mathcal{J}$, $i\in \mathcal{I}$, or $i\in \mathcal{J}$, respectively.
Moreover, by construction, $0\notin \Lambda^{\mathcal{I},\mathcal{J}}$ when $\mathcal{I}\cup \mathcal{J}\neq \emptyset$.
For $\lambda > \lambda^{*}$ and $\mu > \mu^{*}(\lambda)$, we prove that the degree is defined and
\begin{equation}\label{eq-1.7}
\text{\rm deg}_{LS}(Id - \Phi_{\lambda,\mu},\Lambda^{\mathcal{I},\mathcal{J}},0) \neq 0.
\end{equation}
Condition \eqref{eq-1.7} together with a maximum principle argument implies the existence of a \textit{non-negative} solution $u(t)$ of
\eqref{eq-main} satisfying the boundary conditions and, moreover, such that $u\in \Lambda^{\mathcal{I},\mathcal{J}}$.
This non-negative solution is either positive or the trivial one.
Considering all the possible choices of pairwise disjoint sets $\mathcal{I}, \mathcal{J} \subseteq \{1,\ldots m\}$
with $\mathcal{I}\cup \mathcal{J}\neq \emptyset$, we thus obtain the desired $3^{m} -1$ positive solutions.

\medskip

The plan of the paper is the following. In Section~\ref{section-2} we introduce the functional analytic setting
to deal with the operator equation \eqref{eq-1.6}. We shall focus our attention mainly to the case of the periodic
boundary value problem (so that the operator $L$ is not invertible) exploiting the
framework and the properties of Mawhin's coincidence degree. Although coincidence degree theory has been already
well developed in some classical textbooks (see \cite{GaMa-77, Ma-79, Ma-93}), we recall some main properties for the reader's convenience.
In particular, due to our choice of considering equation \eqref{eq-1.6}
on open and unbounded sets, we present the theory from the slightly more general point of view
of locally compact operators.
In Section~\ref{section-3} we define the open and unbounded sets $\Lambda^{\mathcal{I},\mathcal{J}}$ and describe the general strategy for the proof of
the degree formula \eqref{eq-1.7}. In more detail, we first introduce some auxiliary open and unbounded sets $\Omega^{\mathcal{I},\mathcal{J}}$
and we then present two lemmas (Lemma~\ref{lem-deg0} and Lemma~\ref{lem-deg1}) for the computation of
\begin{equation}\label{eq-1.8}
\text{\rm deg}_{LS}(Id - \Phi_{\lambda,\mu},\Omega^{\mathcal{I},\mathcal{J}},0).
\end{equation}
The obtention of \eqref{eq-1.7} from the evaluation of the degrees in \eqref{eq-1.8}
is justified in the appendix using a purely combinatorial argument.
Next, in Section~\ref{section-4} we actually show, by means of some careful estimates on the solutions
of some parameterized equations related to \eqref{eq-main}, that the above lemmas and the general strategy
can be applied for $\lambda$ and $\mu$ large, thus concluding the proof of Theorem~\ref{main-theorem}.
In Section~\ref{section-5} we present some general properties of (not necessarily periodic) positive
solutions of \eqref{eq-main} defined on the whole real line and we discuss the limit behavior for $\mu\to+\infty$.
Section~\ref{section-6} is devoted to the study of positive subharmonic solutions and of
positive solutions with a chaotic-like behavior. In a dynamical system perspective, we also prove the presence of
a Bernoulli shift as a factor within the set of positive bounded solutions of \eqref{eq-main}.
The paper ends with Section~\ref{section-7}, where we discuss variants and extensions of Theorem~\ref{main-theorem}
and we also present an application to radially symmetric solutions for some elliptic PDEs.

\section{Abstract setting}\label{section-2}

Dealing with boundary value problems, it is often convenient to choose spaces of functions
defined on compact domains. Therefore, for the $T$-periodic problem, as usual, we
shall restrict ourselves to functions $u(t)$ defined on $\mathopen{[}0,T\mathclose{]}$ and such that
\begin{equation}\label{pbc}
u(0) = u(T), \quad u'(0) = u'(T).
\end{equation}
In the sequel, solutions of a given second order differential equation satisfying
the boundary condition \eqref{pbc} will be referred to as \textit{$T$-periodic solutions}.

\medskip

Let $X:=\mathcal{C}(\mathopen{[}0,T\mathclose{]})$ be the Banach space of continuous functions $u \colon \mathopen{[}0,T\mathclose{]} \to \mathbb{R}$,
endowed with the norm
\begin{equation*}
\|u\|_{\infty} := \max_{t\in \mathopen{[}0,T\mathclose{]}} |u(t)|,
\end{equation*}
and let $Z:=L^{1}(\mathopen{[}0,T\mathclose{]})$ be the Banach space of integrable functions $v \colon \mathopen{[}0,T\mathclose{]} \to \mathbb{R}$,
endowed with the norm
\begin{equation*}
\|v\|_{L^{1}}:= \int_{0}^{T} |v(t)|~\!dt.
\end{equation*}
As well known, the differential operator
\begin{equation*}
L \colon u \mapsto - u'',
\end{equation*}
defined on
\begin{equation*}
\text{\rm dom}\,L := \bigl{\{}u\in W^{2,1}(\mathopen{[}0,T\mathclose{]}) \colon u(0) = u(T), \; u'(0) = u'(T) \bigr{\}} \subseteq X,
\end{equation*}
is a linear Fredholm map of index zero with range
\begin{equation*}
\text{\rm Im}\,L = \biggl{\{}v\in Z \colon \int_{0}^{T} v(t)~\!dt = 0 \biggr\}.
\end{equation*}
Moreover, we can define the projectors
\begin{equation*}
P \colon X \to \ker L \cong {\mathbb{R}}, \qquad Q \colon Z \to \text{\rm coker}\,L \cong Z/\text{\rm Im}\,L \cong \mathbb{R},
\end{equation*}
as the average operators
\begin{equation*}
Pu = Qu := \dfrac{1}{T}\int_{0}^{T} u(t)~\!dt.
\end{equation*}
Finally, let
\begin{equation*}
K_{P} \colon \text{\rm Im}\,L \to \text{\rm dom}\,L \cap \ker P
\end{equation*}
be the right inverse of $L$, that is the operator that to any function $v\in L^{1}(\mathopen{[}0,T\mathclose{]})$ with $\int_{0}^{T} v(t)~\!dt =0$
associates the unique $T$-periodic solution $u$ of
\begin{equation*}
u'' + v(t) =0, \quad \text{ with }\; \int_{0}^{T} u(t)~\!dt = 0.
\end{equation*}

Next, we introduce the $L^{1}$-Carath\'{e}odory function
\begin{equation*}
f_{\lambda,\mu}(t,s) :=
\begin{cases}
\, -s, & \text{if } s \leq 0;\\
\, \bigl{(}\lambda a^{+}(t) - \mu a^{-}(t)\bigr{)} g(s), & \text{if } s \geq 0;
\end{cases}
\end{equation*}
where $a \colon \mathbb{R} \to \mathbb{R}$ is a locally integrable $T$-periodic function, $g \colon {\mathbb{R}}^{+}\to {\mathbb{R}}^{+}$
is a continuous function with $g(0) = 0$ and $\lambda,\mu > 0$ are fixed parameters.
Let us denote by $N_{\lambda,\mu} \colon X \to Z$ the Nemytskii operator induced by the function $f_{\lambda,\mu}$, that is
\begin{equation*}
(N_{\lambda,\mu} u)(t):= f_{\lambda,\mu}(t,u(t)), \quad t\in \mathopen{[}0,T\mathclose{]}.
\end{equation*}
By coincidence degree theory, the operator equation
\begin{equation*}
Lu = N_{\lambda,\mu}u,\quad u\in \text{\rm dom}\,L,
\end{equation*}
is equivalent to the fixed point problem
\begin{equation*}
u = \Phi_{\lambda,\mu}u:= Pu + QN_{\lambda,\mu}u + K_{P}(Id-Q)N_{\lambda,\mu}u, \quad u \in X.
\end{equation*}
Notice that the term $QN_{\lambda,\mu}u$ in the above formula should be more correctly written as $JQN_{\lambda,\mu}u$, where $J$ is a
linear (orientation-preserving) isomorphism from $\text{\rm coker}\,L$ to $\ker L$. However, in our situation, both $\text{\rm coker}\,L$, as well as $\ker L$, can
be identified with $\mathbb{R}$, so that we can take as $J$ the identity on $\mathbb{R}$.
It is standard to verify that $\Phi_{\lambda,\mu} \colon X \to X$ is a completely continuous operator and thus we
say that $N_{\lambda,\mu}$ is \textit{$L$-completely continuous}.

If $\mathcal{O}\subseteq X$ is an open and \textit{bounded} set such that
\begin{equation*}
Lu \neq N_{\lambda,\mu}u, \quad \forall \, u\in \partial{\mathcal{O}}\cap \text{\rm dom}\,L,
\end{equation*}
the \textit{coincidence degree} $D_{L}(L-N_{\lambda,\mu},{\mathcal{O}})$ (\textit{of $L$ and $N_{\lambda,\mu}$ in ${\mathcal{O}}$})
is defined as
\begin{equation*}
D_{L}(L-N_{\lambda,\mu},\mathcal{O}):= \text{\rm deg}_{LS}(Id - \Phi_{\lambda,\mu},{\mathcal{O}},0).
\end{equation*}

In order to introduce the coincidence degree on open (possibly unbounded) sets,
we just follow the standard approach used to define the Leray-Schauder degree
for locally compact maps defined on open sets, which is classical in the theory of fixed point index
(cf.~\cite{Gr-72,Ma-99,Nu-85,Nu-93}). In more detail,
let $\Omega\subseteq X$ be an open set and suppose that the solution set
\begin{equation*}
\text{\rm Fix}\,(\Phi_{\lambda,\mu},\Omega):= \bigl{\{}u\in {\Omega} \colon u =
\Phi_{\lambda,\mu}u\bigr{\}} = \bigl{\{}u\in {\Omega}\cap \text{\rm dom}\,L \colon Lu = N_{\lambda,\mu}u\bigr{\}}
\end{equation*}
is compact.
Then, the Leray-Schauder degree $\text{\rm deg}_{LS}(Id - \Phi_{\lambda,\mu},\Omega,0)$ is defined as
\begin{equation*}
\text{\rm deg}_{LS}(Id - \Phi_{\lambda,\mu},\Omega,0):=\text{\rm deg}_{LS}(Id - \Phi_{\lambda,\mu},\mathcal{V},0),
\end{equation*}
where $\mathcal{V}$ is an open bounded set with
\begin{equation}\label{eq-2.3}
\text{\rm Fix}\,(\Phi_{\lambda,\mu},\Omega) \subseteq \mathcal{V} \subseteq \overline{\mathcal{V}} \subseteq \Omega.
\end{equation}
It is possible to check that the definition is independent of the choice of $\mathcal{V}$.
Accordingly, we define the \textit{coincidence degree} $D_{L}(L-N_{\lambda,\mu},\Omega)$
(\textit{of $L$ and $N_{\lambda,\mu}$ in $\Omega$}) as
\begin{equation*}
D_{L}(L-N_{\lambda,\mu},\Omega):= D_{L}(L-N_{\lambda,\mu},{\mathcal{V}}) =\text{\rm deg}_{LS}(Id - \Phi_{\lambda,\mu},\mathcal{V},0),
\end{equation*}
with $\mathcal{V}$ as above.
In the special case of an open and \textit{bounded} set $\Omega$ such that
\begin{equation}\label{eq-2.4}
Lu \neq N_{\lambda,\mu}u, \quad \forall \, u\in \partial\Omega\cap \text{\rm dom}\,L,
\end{equation}
it is easy to verify that the above definition reduces to the classical one. Indeed, if \eqref{eq-2.4} holds with $\Omega$ open and bounded, then,
by the excision property of the Leray-Schauder degree, we have
$\text{\rm deg}_{LS}(Id - \Phi_{\lambda,\mu},\mathcal{V},0) = \text{\rm deg}_{LS}(Id - \Phi_{\lambda,\mu},\Omega,0)$
for each open bounded set $\mathcal{V}$ satisfying \eqref{eq-2.3}.

Combining the properties of coincidence degree with the theory of fixed point index for locally compact operators,
it is possible to derive the following versions of the main properties of the degree.

\begin{itemize}
\item \textit{Additivity. }
Let $\Omega_{1}$, $\Omega_{2}$ be open and disjoint subsets of $\Omega$ such that $\text{\rm Fix}\,(\Phi_{\lambda,\mu},\Omega)\subseteq \Omega_{1}\cup\Omega_{2}$.
Then
\begin{equation*}
D_{L}(L-N_{\lambda,\mu},\Omega) = D_{L}(L-N_{\lambda,\mu},\Omega_{1})+ D_{L}(L-N_{\lambda,\mu},\Omega_{2}).
\end{equation*}

\item \textit{Excision. }
Let $\Omega_{0}$ be an open subset of $\Omega$ such that $\text{\rm Fix}\,(\Phi_{\lambda,\mu},\Omega)\subseteq \Omega_{0}$.
Then
\begin{equation*}
D_{L}(L-N_{\lambda,\mu},\Omega)=D_{L}(L-N_{\lambda,\mu},\Omega_{0}).
\end{equation*}

\item \textit{Existence theorem. }
If $D_{L}(L-N_{\lambda,\mu},\Omega)\neq0$, then $\text{\rm Fix}\,(\Phi_{\lambda,\mu},\Omega)\neq\emptyset$,
hence there exists $u\in {\Omega}\cap \text{\rm dom}\,L$ such that $Lu = N_{\lambda,\mu}u$.

\item \textit{Homotopic invariance. }
Let $H\colon\mathopen{[}0,1\mathclose{]}\times \Omega \to X$, $H_{\vartheta}(u) := H(\vartheta,u)$, be a continuous homotopy such that
\begin{equation*}
\mathcal{S}:=\bigcup_{\vartheta\in\mathopen{[}0,1\mathclose{]}} \bigl{\{}u\in \Omega\cap \text{\rm dom}\,L \colon Lu=H_{\vartheta}u\bigr{\}}
\end{equation*}
is a compact set and there exists an open neighborhood $\mathcal{W}$ of $\mathcal{S}$ such that $\overline{\mathcal{W}}\subseteq \Omega$ and
$(K_{P}(Id-Q)H)|_{\mathopen{[}0,1\mathclose{]}\times\overline{\mathcal{W}}}$ is a compact map.
Then the map $\vartheta\mapsto D_{L}(L-H_{\vartheta},\Omega)$ is constant on $\mathopen{[}0,1\mathclose{]}$.

\end{itemize}

For more details, proofs and applications, we refer to \cite{GaMa-77,Ma-79,Ma-93} and the references therein.

In the sequel we will apply this general setting in the following manner. We consider a $L$-completely continuous operator $\mathcal{N}$
and an open (not necessarily bounded) set $\mathcal{A}$ such that
the solution set $\{ u\in \overline{\mathcal{A}}\cap \text{\rm dom}\,L\colon Lu = \mathcal{N} u \}$
is compact and disjoint from $\partial\mathcal{A}$. Therefore $D_L(L-\mathcal{N},\mathcal{A})$ is well defined.
We will proceed analogously when dealing with homotopies.

We notice that, by the existence theorem, if $D_{L}(L-N_{\lambda,\mu},\Omega)\neq0$ for some open set $\Omega \subseteq X$,
then equation
\begin{equation}\label{eq-flm}
u'' + f_{\lambda,\mu}(t,u) = 0
\end{equation}
has at least one solution in $\Omega$ satisfying the boundary condition \eqref{pbc}. If we denote by $u(t)$ such a solution,
we have that $u(t)$ can be extended by $T$-periodicity to a $T$-periodic solution of \eqref{eq-flm} defined on the whole real line.
Moreover, a standard application of the maximum principle ensures that $u(t) \geq 0$ for all $t\in \mathbb{R}$.
Finally, if $g(s)/s$ is bounded in a right neighborhood of $s=0$ (a situation which always occurs if we assume $(g_{0})$),
then either $u\equiv 0$ or $u(t) > 0$ for all $t\in \mathbb{R}$.

\begin{remark}\label{rem-2.1}
As already observed in the introduction, our main attention is devoted to the study of the periodic problem for $L u = - u''$,
while, for Neumann and Dirichlet boundary conditions, as well as for other operators, we only underline which modifications are needed.

If we study the Neumann problem, we just modify the domain of $L$ as
\begin{equation*}
\text{\rm dom}\,L := \bigl{\{} u\in W^{2,1}(\mathopen{[}0,T\mathclose{]}) \colon u'(0) = u'(T) = 0 \bigr{\}} \subseteq X
\end{equation*}
and all the rest is basically the same with elementary modifications. Obviously,
the right inverse of $L$ now is the operator which associates to any function $v\in L^{1}(\mathopen{[}0,T\mathclose{]})$
satisfying $\int_{0}^{T} v(t)~\!dt = 0$ the unique solution of $u'' + v(t) = 0$ with
$u'(0) = u'(T) = 0$ and $\int_{0}^{T} u(t)~\!dt = 0$.

In the case of the Dirichlet problem, the domain of $L$ is
\begin{equation*}
\text{\rm dom}\,L := W^{2,1}_{0}(\mathopen{[}0,T\mathclose{]}) = \bigl{\{}u\in W^{2,1}(\mathopen{[}0,T\mathclose{]}) \colon u(0) = u(T) = 0 \bigr{\}} \subseteq X,
\end{equation*}
but now the differential operator $L$ is invertible (indeed it can be expressed by means of the Green's function), so that
$\Phi_{\lambda,\mu} = L^{-1} N_{\lambda,\mu}$. In this situation, coincidence degree theory reduces to the
classical Leray-Schauder one for locally compact operators.

Finally, we observe that the above framework remains substantially unchanged for other classes of
linear differential operators. In the periodic case, exactly the same considerations as above are valid if we take the operator
\begin{equation*}
L \colon u \mapsto - u'' - cu',
\end{equation*}
where $c \in \mathbb{R}$ is an arbitrary but fixed constant
(recall that the maximum principle is still valid in this setting, see \cite[\S~6]{FeZa-ade2015}). This, in principle, allows us
to insert a dissipation term in equation \eqref{eq-main} (see Section~\ref{section-7.1} for a more detailed discussion).

Concerning the Neumann and the Dirichlet problems, we can easily deal with
self-adjoint differential operators of the form
\begin{equation*}
L \colon u \mapsto -(p(t)u')',
\end{equation*}
with $p(t) > 0$ for all $t\in \mathopen{[}0,T\mathclose{]}$. We do not insist further on these aspects; however,
we will see later a special example of $p(t)$ which naturally
arises in the study of radially symmetric solutions of elliptic PDEs (see Section~\ref{section-7.3}).
$\hfill\lhd$
\end{remark}

\section{Proof of Theorem~\ref{main-theorem}: an outline}\label{section-3}

The proof of Theorem~\ref{main-theorem} and its variants is based on the abstract setting
described in the previous section but it also requires some careful estimates on the
solutions of \eqref{eq-main} and of some related equations. In this section we
first introduce some special open sets of the Banach space $X$ where the coincidence degree
will be computed and next we present the main steps which are required for these computations.
In this manner we can skip for a moment all the technical estimates (which are developed in Section~\ref{section-4})
and focus ourselves on the general strategy of the proof.

From now on, all the assumptions on $a(t)$ and $g(s)$ in Theorem~\ref{main-theorem} will be implicitly assumed.

\subsection{General strategy}\label{section-3.1}

Let us fix an arbitrary constant $\rho > 0$. Depending on $\rho$, we determine a value
$\lambda^{*} = \lambda^{*}(\rho) > 0$ such that, for $\lambda > \lambda^{*}$, any solution to
\begin{equation*}
u'' + \lambda a^{+}(t)g(u) = 0,
\end{equation*}
with $\max_{t\in I^{+}_{i}} u(t) = \rho$, must vanish on $I^{+}_{i}$ (whatever the index $i=1\ldots,m$). This
fact is expressed in a more formal way in Lemma~\ref{lem-rho} (where we also consider a more general equation).
From now on, both $\rho$ and $\lambda > \lambda^{*}$ are fixed.

Next, given any constants $r,R$ with $0 < r < \rho < R$ and for any pair of
subsets of indices $\mathcal{I},\mathcal{J} \subseteq \{1,\ldots,m\}$ (possibly empty) with
$\mathcal{I} \cap \mathcal{J} = \emptyset$, we define the open and unbounded set
\begin{equation}\label{eq-Omega}
\Omega^{\mathcal{I},\mathcal{J}}_{(r,\rho,R)} :=
\left\{ u \in X \colon
\begin{array}{l}
\max_{I^{+}_{i}}|u|<r, \; i\in\{1,\ldots,m\}\setminus(\mathcal{I}\cup\mathcal{J})
\\
\max_{I^{+}_{i}}|u|<\rho, \; i\in\mathcal{I}
\\
\max_{I^{+}_{i}}|u|<R, \; i\in\mathcal{J}
\end{array} \right\}.
\end{equation}
Then, in Section~\ref{section-4.2} we determine two specific constants $r$, $R$ with $0 < r < \rho < R$ such that, for any choice of $\mathcal{I},\mathcal{J}$
as above, the coincidence degree
\begin{equation*}
D_{L}\bigl{(}L-N_{\lambda,\mu},\Omega^{\mathcal{I},\mathcal{J}}_{(r,\rho,R)} \bigr{)}
\end{equation*}
is defined, provided that $\mu$ is sufficiently large (say $\mu > \mu^{*}(\lambda,r,R)$).
Along this process, in Section~\ref{section-4.3} and Section~\ref{section-4.4} we also prove Theorem~\ref{deg-Omega} below.

\begin{theorem}\label{deg-Omega}
In the above setting, it holds that
\begin{equation}\label{eq-3.1}
D_{L}\bigl{(}L-N_{\lambda,\mu},\Omega^{\mathcal{I},\mathcal{J}}_{(r,\rho,R)} \bigr{)} =
\begin{cases}
\, 0, & \text{if } \;\mathcal{I} \neq \emptyset; \\
\, 1, & \text{if } \;\mathcal{I} = \emptyset.
\end{cases}
\end{equation}
\end{theorem}

Then, having fixed $\rho$, $\lambda$, $r$, $R$, $\mu$ as above,
we further introduce the open and unbounded sets
\begin{equation}\label{eq-Lambda}
\Lambda^{\mathcal{I},\mathcal{J}}_{(r,\rho,R)} :=
\left\{ u \in X \colon
\begin{array}{l}
 \max_{I^{+}_{i}}|u|<r, \; i\in\{1,\ldots,m\}\setminus(\mathcal{I}\cup\mathcal{J})
\\
r<\max_{I^{+}_{i}}|u|<\rho, \; i\in\mathcal{I}
\\
\rho<\max_{I^{+}_{i}}|u|<R, \; i\in\mathcal{J}
\end{array} \right\}.
\end{equation}
From Theorem~\ref{deg-Omega} and a combinatorial argument (see Appendix~\ref{appendix-A}), we can prove the following.

\begin{theorem}\label{deg-Lambda}
In the above setting, it holds that
\begin{equation}\label{eq-3.2}
D_{L}\bigl{(}L-N_{\lambda,\mu},\Lambda^{\mathcal{I},\mathcal{J}}_{(r,\rho,R)}\bigr{)} =
(-1)^{\# \mathcal{I}}.
\end{equation}
\end{theorem}

As a consequence of the existence property for the coincidence degree, we thus obtain the existence of a
$T$-periodic solution of \eqref{eq-flm} in each of these $3^{m}$ sets $\Lambda^{\mathcal{I},\mathcal{J}}_{(r,\rho,R)}$
(taking into account all the possible cases for $\mathcal{I},\mathcal{J}$).
Notice that $\Lambda^{\emptyset,\emptyset}(r,\rho,R)$ contains the trivial solution.
In all the other $3^{m}-1$ sets, the solution must be nontrivial and hence, by the maximum principle argument recalled in the previous section,
a \textit{positive} solution of \eqref{eq-main}.
In this manner we can conclude that, for each choice of $\mathcal{I},\mathcal{J}$
with $\mathcal{I} \cup \mathcal{J} \neq \emptyset$, there exists at least one positive $T$-periodic solution $u(t)$ of
\eqref{eq-main} such that
\begin{itemize}
\item $0 < \max_{t\in I^{+}_{i}} u(t) < r$, for $i \notin \mathcal{I} \cup \mathcal{J}$;
\item $r < \max_{t\in I^{+}_{i}} u(t) < \rho$, for all $i \in \mathcal{I}$;
\item $\rho < \max_{t\in I^{+}_{i}} u(t) < R$, for all $i \in \mathcal{J}$.
\end{itemize}
Finally, in order to achieve the conclusion of Theorem~\ref{main-theorem}, we just observe that,
given any finite string $\mathcal{S} = (\mathcal{S}_{1},\ldots,\mathcal{S}_{m}) \in \{0,1,2\}^{m}$, with $\mathcal{S} \neq (0,\ldots,0)$,
we can associate to $\mathcal{S}$ the sets
\begin{equation*}
{\mathcal{I}}:= \bigl{\{} i\in \{1,\ldots,m\} \colon \mathcal{S}_{i} =1 \bigr{\}},\quad
{\mathcal{J}}:= \bigl{\{} i\in \{1,\ldots,m\} \colon \mathcal{S}_{i} =2 \bigr{\}},
\end{equation*}
so that $\mathcal{S}_{i} = 0$ when $i\notin {\mathcal{I}} \cup {\mathcal{J}}$.
This completes the proof of Theorem~\ref{main-theorem}.
\qed

\subsection{Degree lemmas}\label{section-3.2}

For the proof of Theorem~\ref{deg-Omega}, we need to compute the topological degrees
in formula \eqref{eq-3.1}. To this end, we will use the following results.

\begin{lemma}\label{lem-deg0}
Let $\mathcal{I} \neq \emptyset$.
Assume that there exists $v \in L^{1}(\mathopen{[}0,T\mathclose{]})$, with $v(t) \succ 0$ on $\mathopen{[}0,T\mathclose{]}$
and $v \equiv 0$ on $\bigcup_{i} I^{-}_{i}$, such that the following properties hold.
\begin{itemize}
\item[$(H_{1})$]
If $\alpha \geq 0$, then any $T$-periodic solution $u(t)$ of
\begin{equation}\label{eq-lem-deg0}
u'' + \bigl{(} \lambda a^{+}(t) - \mu a^{-}(t) \bigr{)} g(u) + \alpha v(t) = 0,
\end{equation}
with $0 \leq u(t) \leq R$ for all $t\in \mathopen{[}0,T\mathclose{]}$,
satisfies
\begin{itemize}
\item[$\bullet$]
$\max_{t\in I^{+}_{i}} u(t)\neq r$, if $i \notin \mathcal{I} \cup \mathcal{J}$;
\item[$\bullet$]
$\max_{t\in I^{+}_{i}} u(t) \neq \rho$, if $i \in \mathcal{I}$;
\item[$\bullet$]
$\max_{t\in I^{+}_{i}} u(t) \neq R$, if $i \in \mathcal{J}$.
\end{itemize}
\item[$(H_{2})$]
There exists $\alpha^{*} \geq 0$ such that equation \eqref{eq-lem-deg0}, with $\alpha=\alpha^{*}$, does not possess any
non-negative $T$-periodic solution $u(t)$ with
\begin{equation*}
u(t) \leq \rho, \quad \forall \, t \in \bigcup_{i \in \mathcal{I}} I^{+}_{i}.
\end{equation*}
\end{itemize}
Then it holds that
\begin{equation*}
D_{L}\bigl{(}L-N_{\lambda,\mu},\Omega^{\mathcal{I},\mathcal{J}}_{(r,\rho,R)}\bigr{)} = 0.
\end{equation*}
\end{lemma}

\begin{proof}
We adapt to our situation an argument from \cite[Lemma~2.1]{BoFeZa-15}.
We first write the equation
\begin{equation}\label{eq-lem3.1}
u'' + f_{\lambda,\mu}(t,u) + \alpha v(t) = 0
\end{equation}
as a coincidence equation in the space $X$
\begin{equation*}
Lu = N_{\lambda,\mu} u + \alpha v, \quad u \in \text{\rm dom}\,L,
\end{equation*}
and we check that the coincidence degree
$D_{L}\bigl{(}L-N_{\lambda,\mu}-\alpha v,\Omega^{\mathcal{I},\mathcal{J}}_{(r,\rho,R)}\bigr{)}$ is
well-defined for any $\alpha \geq 0$. To this end, for $\alpha \geq 0$, we consider the solution set
\begin{equation*}
\mathcal{R}_{\alpha}
:= \bigl{\{} u \in \text{\rm cl}\bigl{(}\Omega^{\mathcal{I},\mathcal{J}}_{(r,\rho,R)} \bigr{)} \cap \text{\rm dom}\,L \colon
Lu = N_{\lambda,\mu} u + \alpha v \bigr{\}}.
\end{equation*}
We have that $u \in \mathcal{R}_{\alpha}$ if and only if $u(t)$ is a $T$-periodic solution of
\eqref{eq-lem3.1} with $|u(t)| \leq r$ for all $t \in I^{+}_{i}$ if $i \notin \mathcal{I} \cup \mathcal{J}$,
$|u(t)| \leq \rho$ for all $t \in I^{+}_{i}$ if ${i} \in \mathcal{I}$, and $|u(t)| \leq R$ for all $t \in I^{+}_{i}$ if $i \in \mathcal{J}$.
By a maximum principle argument, we find $u(t) \geq 0$ for any $t$.
Moreover, taking into account that $v(t) \succ 0$ on $\mathopen{[}0,T\mathclose{]}$ and $v \equiv 0$ on $\bigcup_{i} I^{-}_{i}$,
we have that $u(t)$ is concave in each $I^{+}_{i}$ and convex in each $I^{-}_{i}$. As a consequence, $u(t) \leq R$ for any $t$.
Hence, $\mathcal{R}_{\alpha} \subseteq B[0,R]:=\{u\in X\colon\|u\|_{\infty}\leq R\}$ and the complete continuity of
$\Phi_{\lambda,\mu}$ implies that $\mathcal{R}_{\alpha}$ is compact.
Furthermore, condition $(H_{1})$ guarantees that
$\max_{I^{+}_{i}} u < r$ if $i \notin \mathcal{I} \cup \mathcal{J}$,
$\max_{I^{+}_{i}} u < \rho$ if $i \in \mathcal{I}$, and
$\max_{I^{+}_{i}} u < R$ if $i \in \mathcal{J}$.
Thus, $\mathcal{R}_{\alpha} \subseteq \Omega^{\mathcal{I},\mathcal{J}}_{(r,\rho,R)}$.
In this way we conclude that the coincidence degree $D_{L}\bigl{(}L-N_{\lambda,\mu}-\alpha v,\Omega^{\mathcal{I},\mathcal{J}}_{(r,\rho,R)}\bigr{)}$
is well-defined for any $\alpha \geq 0$.

Now, using $\alpha$ as homotopy parameter and using the homotopic invariance of the degree
(with the same argument as above, we can see that $\bigcup_{\alpha \in \mathopen{[}0,\alpha^{*}\mathclose{]}}\mathcal{R}_{\alpha}$ is
a compact subset of $\Omega^{\mathcal{I},\mathcal{J}}_{(r,\rho,R)}$), we have that
\begin{equation*}
D_{L}\bigl{(}L-N_{\lambda,\mu},\Omega^{\mathcal{I},\mathcal{J}}_{(r,\rho,R)}\bigr{)}
= D_{L}\bigl{(}L-N_{\lambda,\mu}-\alpha^{*} v,\Omega^{\mathcal{I},\mathcal{J}}_{(r,\rho,R)}\bigr{)}.
\end{equation*}
If, by contradiction, this degree is non-null, then there exists at least one $T$-periodic solution
$u \in \Omega^{\mathcal{I},\mathcal{J}}_{(r,\rho,R)}$ of \eqref{eq-lem3.1} with $\alpha = \alpha^{*}$.
Again by the maximum principle, we then have a non-negative $T$-periodic solution of \eqref{eq-lem-deg0}
with $\alpha = \alpha^{*}$ and, since $u \in \Omega^{\mathcal{I},\mathcal{J}}_{(r,\rho,R)}$
with $\mathcal{I} \neq \emptyset$, it holds that $\max_{I^{+}_{i}} u \leq \rho$ if $i \in \mathcal{I}$.
This contradicts assumption $(H_{2})$ and the proof is complete.
\end{proof}

The next result uses a duality theorem by Mawhin which relates the coincidence degree
with the (finite dimensional) Brouwer degree, denoted here as ``$\text{\rm deg}_{B}$''.
We recall also the definition of $\mu^{\#}(\lambda)$ given in \eqref{mudiesis}.

\begin{lemma}\label{lem-deg1}
Let $\mathcal{I} = \emptyset$ and assume that the following property holds.
\begin{itemize}
\item[$(H_{3})$]
If $\vartheta\in \mathopen{]}0,1\mathclose{]}$, then any $T$-periodic solution $u(t)$ of
\begin{equation}\label{eq-lem-deg1}
u'' + \vartheta \bigl{(} \lambda a^{+}(t) - \mu a^{-}(t) \bigr{)} g(u) = 0,
\end{equation}
with $0 \leq u(t) \leq R$ for all $t\in \mathopen{[}0,T\mathclose{]}$, satisfies
\begin{itemize}
\item[$\bullet$]
$\max_{t \in I^{+}_{i}} u(t) \neq r$, if $i \notin \mathcal{J}$;
\item[$\bullet$]
$\max_{t\in I^{+}_{i}} u(t) \neq R$, if $i \in \mathcal{J}$.
\end{itemize}
\end{itemize}
Then, for any $\lambda > 0$ and $\mu > \mu^{\#}(\lambda)$, it holds that
\begin{equation*}
D_{L}\bigl{(}L-N_{\lambda,\mu},\Omega^{\emptyset,\mathcal{J}}_{(r,\rho,R)} \bigr{)} = 1.
\end{equation*}
\end{lemma}

\begin{proof}
We argue similarly as in \cite[Lemma~2.2]{BoFeZa-15},
using a reduction property for the coincidence degree from Mawhin's continuation
theorem (see \cite[Theorem~2.4]{Ma-93} as well as \cite{Ma-69},
where the result was previously given in the context of the periodic problem for ODEs).
We consider the parameterized equation
\begin{equation*}
u = \Psi_{\vartheta}(u) := Pu + Q N_{\lambda,\mu} u + \vartheta K_{P} (Id-Q) N_{\lambda,\mu} u,
\quad u \in X, \; \vartheta \in \mathopen{[}0,1\mathclose{]}.
\end{equation*}
Let also
\begin{equation*}
\mathcal{S} := \bigcup_{\vartheta \in \mathopen{[}0,1\mathclose{]}}
\Bigl{\{} u \in \text{\rm cl}\,\bigl{(}\Omega^{\emptyset,\mathcal{J}}_{(r,\rho,R)}\bigr{)} \colon u = \Psi_{\vartheta}(u) \Bigr{\}}.
\end{equation*}

Suppose that $0 < \vartheta \leq 1$. In this situation, $u = \Psi_{\vartheta}(u)$ if and only if
\begin{equation*}
Lu = \vartheta N_{\lambda,\mu} u, \quad u \in \text{\rm dom}\,L,
\end{equation*}
or, equivalently, $u(t)$ is a $T$-periodic solution of
\begin{equation*}
u'' + \vartheta f_{\lambda,\mu}(t,u) = 0.
\end{equation*}
If $u \in \text{\rm cl}\,\bigl{(}\Omega^{\emptyset,\mathcal{J}}_{(r,\rho,R)}\bigr{)}$, we know that
$\max_{I^{+}_{i}}|u| \leq r$ if $i \notin \mathcal{J}$ and
$\max_{I^{+}_{i}} |u| \leq R$ if $i \in \mathcal{J}$.
Hence, by a maximum principle, $u(t)$ is a non-negative $T$-periodic solution of \eqref{eq-lem-deg1}
and, by a convexity argument, $u(t) \leq R$ for any $t$. Moreover, by $(H_{3})$, $\max_{I^{+}_{i}} u < r$
if $i \notin \mathcal{J}$ and $\max_{I^{+}_{i}} u < R$ if $i \in \mathcal{J}$.

On the other hand, if $\vartheta = 0$, $u$ is a solution of $u = \Psi_{0}(u)$
if and only if $u = Pu + Q N_{\lambda,\mu} u$, that is,
$u \in \ker L$ and $Q N_{\lambda,\mu} u = 0$.
Since $\ker L \cong \mathbb{R}$ and
\begin{equation*}
QN_{\lambda,\mu}u = \dfrac{1}{T}\int_{0}^{T} f_{\lambda,\mu}(t,s)~\!dt, \quad \text{ for } u \equiv \text{constant} = s \in \mathbb{R},
\end{equation*}
we conclude that $u \equiv s \in \mathbb{R}$ is a solution
of $u = \Psi_{0}(u)$ with $u \in \text{\rm cl}\,\bigl{(}\Omega^{\emptyset,\mathcal{J}}_{(r,\rho,R)}\bigr{)}$
if and only if $|s| \leq r$ if $\mathcal{J} \neq \{1,\ldots,m\}$ and $|s| \leq R$ if $\mathcal{J} = \{1,\ldots,m\}$
and, moreover, $f_{\lambda,\mu}^{\#}(s) = 0$, where we have set
\begin{equation*}
f_{\lambda,\mu}^{\#}(s) := \dfrac{1}{T}\int_{0}^{T} f_{\lambda,\mu}(t,s)~\!dt =
\begin{cases}
\, -s, & \text{if } s \leq 0; \\
\, \biggl{(}\dfrac{1}{T} \displaystyle \int_{0}^{T} a_{\lambda,\mu}(t)~\!dt \biggr{)} g(s), & \text{if } s \geq 0.
\end{cases}
\end{equation*}
If $\mu > \mu^{\#}(\lambda)$, we have that $f^{\#}_{\lambda,\mu}$ satisfies $f^{\#}_{\lambda,\mu}(s)s < 0$ for $s \neq 0$. Hence $u\equiv0$.

We conclude that the set $\mathcal{S}$ is compact and contained in $\Omega^{\emptyset,\mathcal{J}}_{(r,\rho,R)}$.
By the homotopic invariance of the coincidence degree, we have that
\begin{equation*}
\begin{aligned}
D_{L}\bigl{(}L-N_{\lambda,\mu},\Omega^{\emptyset,\mathcal{J}}_{(r,\rho,R)} \bigr{)} & =
\text{\rm deg}_{LS}\bigl{(}Id-\Psi_{1},\Omega^{\emptyset,\mathcal{J}}_{(r,\rho,R)},0 \bigr{)} \\
& = \text{\rm deg}_{LS}\bigl{(}Id-\Psi_{0},\Omega^{\emptyset,\mathcal{J}}_{(r,\rho,R)},0 \bigr{)} \\
& = \text{\rm deg}_{B} \bigl{(}-QN_{\lambda,\mu}|_{\ker L},\Omega^{\emptyset,\mathcal{J}}_{(r,\rho,R)}\cap \ker L,0 \bigr{)} \\
& = \text{\rm deg}_{B} \bigl{(}-f^{\#}_{\lambda,\mu}|_{\ker L},\mathopen{]}-d,d\mathclose{[},0 \bigr{)} = 1,
\end{aligned}
\end{equation*}
where $d= r$ or $d =R$ according to whether $\mathcal{J} \neq \{1,\ldots,m\}$
or $\mathcal{J} = \{1,\ldots,m\}$. This concludes the proof.
\end{proof}

\begin{remark}\label{rem-3.1}
When dealing with other differential operators $L$ or with Neumann and Dirichlet boundary conditions,
some changes are required.

First of all we notice that Lemma~\ref{lem-deg0} and Lemma~\ref{lem-deg1}
hold exactly the same for the $T$-periodic problem and the differential
operator $u \mapsto - u'' - c u'$. The same is true for Neumann boundary conditions: we have only to assume for
equation \eqref{eq-lem-deg0} and \eqref{eq-lem-deg1} that $u(t)$ is a solution satisfying $u'(0) = u'(T) = 0$.
For these cases, no relevant changes are needed in the proofs.

Concerning the Dirichlet problem the following modifications are in order. First, in all the degree formulas
the terms $D_{L}(L - N_{\lambda,\mu},\cdot)$ have to be replaced by $\text{\rm deg}_{LS}{(}Id-L^{-1} N_{\lambda,\mu},\cdot,0)$.
Secondly, in equations \eqref{eq-lem-deg0} and \eqref{eq-lem-deg1}
we have to suppose that $u(t)$ is a solution satisfying $u(0) = u(T) = 0$. Finally, we strongly simplify the argument
in the proof of Lemma~\ref{lem-deg1} since, when $\vartheta = 0$, we directly reduce to
the trivial equation $u = 0$. Therefore the homotopic invariance of the Leray-Schauder degree
(with respect to the parameter $\vartheta\in \mathopen{[}0,1\mathclose{]}$) yields
\begin{equation*}
D_{L}(L - N_{\lambda,\mu},\Omega^{\emptyset,\mathcal{J}}_{(r,\rho,R)})
= \text{\rm deg}_{LS}{(}Id,\Omega^{\emptyset,\mathcal{J}}_{(r,\rho,R)},0) = 1,
\end{equation*}
because $0 \in \Omega^{\emptyset,\mathcal{J}}_{(r,\rho,R)}$. In this case the condition $\mu > \mu^{\#}(\lambda)$
is not required in Lemma~\ref{lem-deg1}. However, the largeness of $\mu$
will be in any case needed later in subsequent technical estimates.
$\hfill\lhd$
\end{remark}

\section{Proof of Theorem~\ref{main-theorem}: the details}\label{section-4}

In view of the general strategy for the proof described in Section~\ref{section-3}, we are going to prove
that the assumptions $(H_{1})$, $(H_{2})$ of Lemma~\ref{lem-deg0} and $(H_{3})$ of Lemma~\ref{lem-deg1}
are satisfied for suitable choices of $r, \rho, R$ and $\lambda, \mu$ large enough. These proofs are given in the
second part of this section (see Section~\ref{section-4.3} and Section~\ref{section-4.4}).
Lemma~\ref{lem-deg0} and Lemma~\ref{lem-deg1} involve the study of the solutions of
\eqref{eq-lem-deg0} and \eqref{eq-lem-deg1}, respectively. These equations, although different, present
common features and, for this reason, we premise some technical estimates on the solutions
which will help and simplify our subsequent proofs.

\medskip

Keeping in mind that all the assumptions on $a(t)$ and $g(s)$ in Theorem~\ref{main-theorem} are assumed,
we introduce now the following notation.
For any constant $d > 0$, we set
\begin{equation}\label{zetagamma}
\zeta(d) := \max_{\frac{d}{2}\leq s \leq d} \dfrac{g(s)}{s}, \qquad
\gamma(d) := \min_{\frac{d}{2}\leq s \leq d} \dfrac{g(s)}{s}.
\end{equation}
Moreover, we also define
\begin{equation*}
g^{*}(d) := \max_{0 \leq s \leq d} g(s), \qquad
g_{*}(d,D) := \min_{d \leq s \leq D} g(s),
\end{equation*}
where $D > d$ is another arbitrary constant. Furthermore, recalling $(a_{*})$ and the positions in \eqref{Ipm},
for all $i=1,\ldots, m$, we set
\begin{equation*}
\|a\|_{\pm,i} := \int_{I^{\pm}_{i}} a^{\pm}(t)~\!dt
\end{equation*}
and
\begin{equation*}
A_{i}(t) := \int_{\tau_{i}}^{t} a^{-}(\xi)~\!d\xi, \quad t\in I^{-}_{i}, \qquad\quad \|A_{i}\| :=\int_{I^{-}_{i}} A_{i}(t)~\!dt.
\end{equation*}

\subsection{Technical estimates}\label{section-4.1}

We present now some preliminary technical lemmas. We stress the fact that all the results in this subsection concern the
properties of solutions of given equations without any reference to the boundary conditions.

\begin{lemma}\label{lem-rho}
For any $\rho > 0$, there exists $\lambda^{*} = \lambda^{*}(\rho) > 0$ such that, for any $\lambda > \lambda^{*}$, $\alpha \geq 0$
and $i \in \{1,\ldots,m\}$,
there are no non-negative solutions $u(t)$ to
\begin{equation}\label{eq-lem-rho}
u'' + \lambda a^{+}(t)g(u) + \alpha = 0,
\end{equation}
with $u(t)$ defined for all $t \in I^{+}_{i}$, and
such that $\max_{t\in I^{+}_{i}} u(t) = \rho$.
\end{lemma}

\begin{proof}
We fix $\varepsilon > 0$ such that, for each $i\in\{1,\ldots,m\}$,
$\varepsilon < (\tau_{i}- \sigma_{i})/2$
and, moreover,
$\int_{\sigma_{i}+\varepsilon}^{\tau_{i}-\varepsilon} a^{+}(t)~\!dt > 0$.
In this manner, the quantity
\begin{equation*}
\nu_{\varepsilon} := \min_{i=1,\ldots,m} \int_{\sigma_{i}+\varepsilon}^{\tau_{i}-\varepsilon} a^{+}(t)~\!dt
\end{equation*}
is well defined and positive.

Let $\rho>0$ be fixed and consider $\alpha \geq 0$ and $i \in \{1,\ldots,m\}$.
Suppose that $u(t)$ is a non-negative solution of \eqref{eq-lem-rho} defined on $I^{+}_{i}$ and such that
\begin{equation*}
\max_{t\in I^{+}_{i}} u(t) = \rho.
\end{equation*}
We claim that
\begin{equation*}
|u'(t)| \leq \dfrac{u(t)}{\varepsilon}, \quad \forall \, t\in \mathopen{[}\sigma_{i}+\varepsilon,\tau_{i}-\varepsilon\mathclose{]}.
\end{equation*}
Indeed, if $t\in \mathopen{[}\sigma_{i}+\varepsilon,\tau_{i}-\varepsilon\mathclose{]}$ is such that $u'(t)=0$, the result is trivially true.
If $u'(t)>0$, we have
\begin{equation*}
u(t) \geq u(t) - u(\sigma_{i}) = \int_{\sigma_{i}}^{t} u'(\xi)~\!d\xi
\geq u'(t) (t-\sigma_{i}) \geq u'(t)\varepsilon.
\end{equation*}
Analogously, if $u'(t)<0$, we have
\begin{equation*}
u(t) \geq u(t) - u(\tau_{i}) = - \int_{t}^{\tau_{i}} u'(\xi)~\!d\xi
\geq -u'(t) (\tau_{i}-t) \geq - u'(t)\varepsilon.
\end{equation*}
The claim is thus proved. As a consequence,
\begin{equation}\label{eq-4.1}
|u'(t)| \leq \dfrac{\rho}{\varepsilon}, \quad \forall \, t\in \mathopen{[}\sigma_{i}+\varepsilon,\tau_{i}-\varepsilon\mathclose{]}.
\end{equation}
On the other hand, the concavity of $u(t)$ on $I^{+}_{i}$ ensures that
\begin{equation}\label{eq-4.2}
u(t)\geq \dfrac{\rho}{|I^{+}_{i}|}\min\{t-\sigma_{i},\tau_{i}-t\}, \quad \forall \, t\in I^{+}_{i}.
\end{equation}
We introduce now the positive constant
\begin{equation*}
\eta_{\varepsilon,\rho} := \min \Biggl{\{} g(s) \colon \dfrac{\varepsilon \rho}{\displaystyle{\max_{i=1,\ldots,m}}|I^{+}_{i}|} \leq s \leq \rho \Biggr{\}}.
\end{equation*}
Integrating equation \eqref{eq-lem-rho} on $\mathopen{[}\sigma_{i}+\varepsilon,\tau_{i}-\varepsilon\mathclose{]}$
and using \eqref{eq-4.1} and \eqref{eq-4.2}, we obtain
\begin{equation*}
\begin{aligned}
\lambda \eta_{\varepsilon,\rho} \int_{\sigma_{i}+\varepsilon}^{\tau_{i}-\varepsilon} a^{+}(t)~\!dt
   &\leq \lambda \int_{\sigma_{i}+\varepsilon}^{\tau_{i}-\varepsilon} a^{+}(t) g(u(t))~\!dt =
   \int_{\sigma_{i}+\varepsilon}^{\tau_{i}-\varepsilon} (- u''(t) - \alpha)~\!dt
\\ &= u'(\sigma_{i}+\varepsilon) - u'(\tau_{i}-\varepsilon) - \alpha \, (\tau_{i}-\sigma_{i}-2\varepsilon)
\leq \dfrac{2\rho}{\varepsilon}.
\end{aligned}
\end{equation*}
Now, we set
\begin{equation*}
\lambda^{*} = \lambda^{*}(\rho):= \dfrac{2 \rho}{\varepsilon \nu_{\varepsilon} \eta_{\varepsilon,\rho}}.
\end{equation*}
Arguing by contradiction, from the last inequality we immediately conclude that there are no non-negative solutions $u(t)$ of \eqref{eq-lem-rho} with
$\max_{t\in I^{+}_{i}} u(t) = \rho$, if $\lambda > \lambda^{*}$.
\end{proof}

\begin{lemma}\label{lem-rR}
Let $\lambda, \mu >0$. Let $d>0$ be such that
\begin{equation}\label{cond-d}
\zeta(d) < \dfrac{1}{2 \lambda \, \displaystyle{\max_{i=1,\ldots,m}}(|I^{+}_{i}|+|I^{-}_{i}|)\|a\|_{+,i}}.
\end{equation}
Suppose that $u(t)$ is a non-negative solution of
\begin{equation*}
u'' + \vartheta \bigl{(}\lambda a^{+}(t) - \mu a^{-}(t) \bigr{)} g(u) = 0, \quad \vartheta\in\mathopen{]}0,1\mathclose{]},
\end{equation*}
defined on $I^{+}_{i} \cup I^{-}_{i}$ for some $i\in \{1,\ldots,m\}$ and such that
\begin{equation*}
\max_{t\in I^{+}_{i}} u(t) = d \quad \text{ and } \quad u'(\sigma_{i}) \geq 0.
\end{equation*}
Then it holds that
\begin{equation*}
u(\sigma_{i+1}) \geq d \, \biggl{[} 1 + \dfrac{\vartheta}{2} \Bigl{(} \mu \gamma(d) \|A_{i}\| - 1\Bigr{)} \biggr{]}
\end{equation*}
and
\begin{equation*}
u'(\sigma_{i+1}) \geq \vartheta d \biggl{(}\mu \dfrac{\gamma(d)}{2} \|a\|_{-,i} - \lambda \|a\|_{+,i} \zeta(d) \biggr{)}.
\end{equation*}
\end{lemma}

\begin{proof}
The proof is split into two parts. In the first one we provide some estimates for $u(\tau_{i})$ and $u'(\tau_{i})$,
while in the second part we obtain the desired inequality on $u(\sigma_{i+1})$ and $u'(\sigma_{i+1})$.

Let $\hat{t}_{i}\in I^{+}_{i}$ be such that
\begin{equation*}
\max_{t\in I^{+}_{i}} u(t) = d = u(\hat{t}_{i}).
\end{equation*}
Observe that $u'(\hat{t}_{i}) = 0$, if $\sigma_{i} \leq \hat{t}_{i} < \tau_{i}$ (since $u'(\sigma_{i}) \geq 0$), while
$u'(\hat{t}_{i}) \geq 0$, if $\hat{t}_{i} = \tau_{i}$.
As a first instance, suppose that
\begin{equation*}
u'(\hat{t}_{i}) = 0.
\end{equation*}
Let $\mathopen{[}s_{1},s_{2}\mathclose{]} \subseteq I^{+}_{i}$ be the maximal closed interval containing $\hat{t}_{i}$ and such that
$u(t)\geq d/2$ for all $t\in\mathopen{[}s_{1},s_{2}\mathclose{]}$. We claim that $\mathopen{[}s_{1},s_{2}\mathclose{]} = I^{+}_{i}$.
From
\begin{equation*}
u''(t) = \vartheta \lambda a^{+}(t) g(u(t)), \quad t \in I^{+}_{i},
\end{equation*}
and
\begin{equation*}
u'(t) = u'(\hat{t}_{i}) + \int_{\hat{t}_{i}}^{t} u''(\xi) ~\!d\xi, \quad \forall \, t\in I^{+}_{i},
\end{equation*}
it follows that
\begin{equation*}
|u'(t)| \leq \vartheta \lambda \|a\|_{+,i} \zeta(d) d, \quad \forall \, t\in \mathopen{[}s_{1},s_{2}\mathclose{]}.
\end{equation*}
Then, in view of \eqref{cond-d},
\begin{equation*}
u(t) = u(\hat{t}_{i}) + \int_{\hat{t}_{i}}^{t} u'(\xi) ~\!d\xi
\geq d - \vartheta \lambda |I^{+}_{i}| \|a\|_{+,i} \zeta(d) d > \dfrac{d}{2}, \quad \forall \, t\in \mathopen{[}s_{1},s_{2}\mathclose{]}.
\end{equation*}
This inequality, together with the maximality of $\mathopen{[}s_{1},s_{2}\mathclose{]}$,
implies that $\mathopen{[}s_{1},s_{2}\mathclose{]} = I^{+}_{i}$. Hence
\begin{equation}\label{eq-4.1tau}
u'(t) \geq - \vartheta \lambda \|a\|_{+,i} \zeta(d) d, \quad \forall \, t\in I^{+}_{i},
\end{equation}
and, a fortiori,
\begin{equation}\label{eq-4.2tau}
u'(\tau_{i}) \geq - \vartheta \lambda \|a\|_{+,i} \zeta(d) d.
\end{equation}
Moreover, after an integration of \eqref{eq-4.1tau} on $\mathopen{[}\hat{t}_{i},\tau_{i}\mathclose{]}$, we obtain
\begin{equation}\label{eq-4.3tau}
u(\tau_{i}) \geq d \bigl{(} 1 - \vartheta \lambda |I^{+}_{i}| \|a\|_{+,i} \zeta(d) \bigr{)}.
\end{equation}
On the other hand, if we suppose that
$\hat{t}_{i} = \tau_{i}$ and $u'(\hat{t}_{i})> 0$, we immediately have
\begin{equation*}
u(\tau_{i}) =d \geq d \bigl{(} 1 - \vartheta \lambda |I^{+}_{i}| \|a\|_{+,i} \zeta(d) \bigr{)} \quad \text{ and } \quad
u'(\tau_{i}) >0 \geq - \vartheta \lambda \|a\|_{+,i} \zeta(d) d.
\end{equation*}
Thus, in any case, \eqref{eq-4.2tau} and \eqref{eq-4.3tau} hold.
Having produced some estimates on $u(\tau_{i})$ and $u'(\tau_{i})$ we are in position now to proceed with the second part of the proof.

We consider the subsequent (adjacent) interval $I^{-}_{i}=\mathopen{[}\tau_{i},\sigma_{i+1}\mathclose{]}$
where the weight is non-positive. Since $u'(t)$ is non-decreasing, from \eqref{eq-4.2tau} we get
\begin{equation*}
u'(t) \geq - \vartheta \lambda \|a\|_{+,i} \zeta(d) d, \quad \forall \, t\in I^{-}_{i}.
\end{equation*}
Therefore, integrating on $\mathopen{[}\tau_{i},t\mathclose{]}$ and using \eqref{eq-4.3tau}, we have
\begin{equation}\label{eq-4.9}
\begin{aligned}
u(t)
   &= u(\tau_{i}) + \int_{\tau_{i}}^{t} u'(\xi) ~\!d\xi
    \geq d \bigl{(} 1 - \vartheta \lambda |I^{+}_{i}| \|a\|_{+,i} \zeta(d) - \vartheta \lambda |I^{-}_{i}| \|a\|_{+,i} \zeta(d)\bigr{)}
\\ &\geq d \bigl{(} 1 - \lambda (|I^{+}_{i}|+|I^{-}_{i}|) \|a\|_{+,i} \zeta(d)\bigr{)} > \dfrac{d}{2}, \quad \forall \, t\in I^{-}_{i},
\end{aligned}
\end{equation}
where the last inequality follows from \eqref{cond-d}. On the other hand, integrating
\begin{equation*}
u''(t) = \vartheta \mu a^{-}(t) g(u(t)), \quad t \in I^{-}_{i},
\end{equation*}
on $\mathopen{[}\tau_{i},t\mathclose{]}$ and using \eqref{eq-4.2tau} and \eqref{eq-4.9}, we find
\begin{equation*}
\begin{aligned}
u'(t)
   &= u'(\tau_{i}) + \int_{\tau_{i}}^{t} \vartheta \mu a^{-}(\xi) g(u(\xi)) ~\!d\xi
\\ &\geq - \vartheta \lambda \|a\|_{+,i} \zeta(d) d + \vartheta \dfrac{d}{2} \mu \gamma(d) A_{i}(t), \quad \forall \, t \in I^{-}_{i}.
\end{aligned}
\end{equation*}
In particular,
\begin{equation*}
u'(\sigma_{i+1}) \geq \vartheta d \biggl{(}\mu \dfrac{\gamma(d)}{2} \|a\|_{-,i} - \lambda \|a\|_{+,i} \zeta(d) \biggr{)}.
\end{equation*}
Finally, a further integration and condition \eqref{cond-d} yield
\begin{equation*}
\begin{aligned}
u(\sigma_{i+1})
   &= u(\tau_{i}) + \int_{\tau_{i}}^{\sigma_{i+1}} u'(t) ~\!dt
\\ &\geq d - \vartheta \lambda (|I^{+}_{i}|+|I^{-}_{i}|) \|a\|_{+,i} \zeta(d)d + \vartheta \dfrac{d}{2} \mu \gamma(d) \|A_{i}\|
\\ &\geq d \, \biggl{[} 1 + \vartheta \biggl{(} \mu \dfrac{\gamma(d)}{2} \|A_{i}\| - \lambda (|I^{+}_{i}|+|I^{-}_{i}|) \|a\|_{+,i} \zeta(d)\biggr{)} \biggr{]}
\\ &\geq d \, \biggl{[} 1 + \dfrac{\vartheta}{2} \Bigl{(} \mu \gamma(d) \|A_{i}\| - 1 \Bigr{)} \biggr{]}.
\end{aligned}
\end{equation*}
This concludes the proof.
\end{proof}

Symmetrically, we have the following.

\begin{lemma}\label{lem-rR-back}
Let $\lambda, \mu >0$. Let $d>0$ be such that
\begin{equation*}
\zeta(d) < \dfrac{1}{2 \lambda \, \displaystyle{\max_{i=1,\ldots,m}}(|I^{+}_{i}|+|I^{-}_{i}|)\|a\|_{+,i}}.
\end{equation*}
Suppose that $u(t)$ is a non-negative solution of
\begin{equation*}
u'' + \vartheta \bigl{(}\lambda a^{+}(t) - \mu a^{-}(t) \bigr{)} g(u) = 0, \quad \vartheta\in\mathopen{]}0,1\mathclose{]},
\end{equation*}
defined on $I^{-}_{i-1} \cup I^{+}_{i}$ for some $i\in \{1,\ldots,m\}$ and such that
\begin{equation*}
\max_{t\in I^{+}_{i}} u(t) = d \quad \text{ and } \quad u'(\tau_{i}) \leq 0.
\end{equation*}
Then it holds that
\begin{equation*}
u(\tau_{i-1}) \geq d \, \biggl{[} 1 + \dfrac{\vartheta}{2} \Bigl{(} \mu \gamma(d)  \|A_{i-1}\| - 1 \Bigr{)} \biggr{]}
\end{equation*}
and
\begin{equation*}
u'(\tau_{i-1}) \geq \vartheta d \biggl{(} \mu \dfrac{\gamma(d)}{2} \|a\|_{-,i-1} - \lambda \|a\|_{+,i} \zeta(d) \biggr{)}.
\end{equation*}
\end{lemma}

\begin{remark}\label{rem-4per}
In the sequel, when dealing with the periodic problem, we observe that
the solutions we consider are defined on $\mathopen{[}0,T\mathclose{]}$ and satisfy $T$-periodic boundary conditions
$u(T) - u(0) = u'(T) - u'(0) = 0$. Hence
it is convenient to count the intervals cyclically.
Accordingly, in the special case in which $i =1$, we apply Lemma~\ref{lem-rR-back} with the agreement $I^{-}_{0}= I^{-}_{m}$.
This makes sense because, if we extend the solution by $T$-periodicity on
the whole real line, we can consider the interval $I^{-}_{m} - T$ as adjacent on the left to $I^{+}_{1}$.
$\hfill\lhd$
\end{remark}

\begin{lemma}\label{lem-forward}
Let $\lambda>0$ and $0 < d < D$.
For any $i\in\{1,\ldots,m\}$ there exists a constant
\begin{equation*}
\mu^{*,+}_{i} = \mu^{*,+}_{i}(I^{-}_{i},I^{+}_{i+1}) > 0
\end{equation*}
such that for all $\mu >\mu^{*,+}_{i}$ any non-negative solution $u(t)$ of
\begin{equation*}
u'' + \vartheta \bigl{(}\lambda a^{+}(t) - \mu a^{-}(t) \bigr{)} g(u) = 0,\quad \vartheta\in\mathopen{]}0,1\mathclose{]},
\end{equation*}
defined on $I^{-}_{i} \cup I^{+}_{i+1}$ and such that
\begin{equation*}
\|u\|_{\infty} \leq D, \quad u(\tau_{i})>d \quad \text{ and } \quad u'(\tau_{i})>0,
\end{equation*}
satisfies
\begin{equation*}
u(t) > d, \quad u'(t)>0, \quad \forall \, t\in I^{-}_{i} \cup I^{+}_{i+1}.
\end{equation*}
\end{lemma}

\begin{proof}
Clearly, by the convexity of $u(t)$ on $I^{-}_{i}$, we have
\begin{equation*}
u(t) > d, \quad u'(t)>0, \quad \forall \, t\in I^{-}_{i}.
\end{equation*}
Integrating
\begin{equation*}
u''(t) = \vartheta \mu a^{-}(t) g(u(t)) \geq \vartheta \mu a^{-}(t) g_{*}(d,D), \quad t\in I^{-}_{i},
\end{equation*}
on $\mathopen{[}\tau_{i},t\mathclose{]} \subseteq I^{-}_{i}$ we find
\begin{equation*}
u'(t) = u'(\tau_{i}) + \int_{\tau_{i}}^{t} u''(\xi) ~\!d\xi > \vartheta \mu A_{i}(t) g_{*}(d,D), \quad \forall \, t\in I^{-}_{i},
\end{equation*}
so that
\begin{equation*}
u'(\sigma_{i+1}) > \vartheta \mu A_{i}(\sigma_{i+1}) g_{*}(d,D) = \vartheta \mu \|a\|_{-,i}\, g_{*}(d,D).
\end{equation*}
On the other hand, integrating
\begin{equation*}
u''(t) = - \vartheta \lambda a^{+}(t) g(u(t)) \geq - \vartheta \lambda a^{+}(t) g^{*}(D), \quad t\in I^{+}_{i+1},
\end{equation*}
on $\mathopen{[}\sigma_{i+1},t\mathclose{]}\subseteq I^{+}_{i+1}$ we find
\begin{equation*}
\begin{aligned}
u'(t)
   &= u'(\sigma_{i+1}) + \int_{\sigma_{i+1}}^{t} u''(\xi) ~\!d\xi
\\ &> \vartheta \Bigl{(} \mu \|a\|_{-,i} g_{*}(d,D) - \lambda \|a\|_{+,i+1} g^{*}(D) \Bigr{)}
> 0, \quad \forall \, t\in I^{+}_{i+1},
\end{aligned}
\end{equation*}
where the last inequality holds provided that
\begin{equation*}
\mu > \mu^{*,+}_{i} = \mu^{*,+}_{i}(I^{-}_{i},I^{+}_{i+1}) := \dfrac{\lambda \|a\|_{+,i+1} g^{*}(D)}{\|a\|_{-,i}g_{*}(d,D)}.
\end{equation*}
Then the solution $u(t)$ is increasing in $I^{+}_{i+1} = \mathopen{[}\sigma_{i+1},\tau_{i+1}\mathclose{]}$ and hence
\begin{equation*}
u(t) > u(\sigma_{i+1}) > d, \quad \forall \, t\in I^{+}_{i+1}.
\end{equation*}
The proof is thus completed.
\end{proof}

Symmetrically, we have the following.

\begin{lemma}\label{lem-back}
Let $\lambda>0$ and $0 < d < D$.
For any $i\in\{1,\ldots,m\}$ there exists a constant
\begin{equation*}
\mu^{*,-}_{i}= \mu^{*,-}_{i}(I^{+}_{i-1},I^{-}_{i-1}) > 0
\end{equation*}
such that for all $\mu >\mu^{*,-}_{i}$ any non-negative solution $u(t)$ of
\begin{equation*}
u'' + \vartheta \bigl{(}\lambda a^{+}(t) - \mu a^{-}(t) \bigr{)} g(u) = 0,\quad \vartheta\in\mathopen{]}0,1\mathclose{]},
\end{equation*}
defined on $I^{+}_{i-1} \cup I^{-}_{i-1}$ and such that
\begin{equation*}
\|u\|_{\infty} \leq D, \quad u(\sigma_{i})>d \quad \text{ and } \quad u'(\sigma_{i})<0,
\end{equation*}
satisfies
\begin{equation*}
u(t) > d, \quad u'(t) < 0, \quad \forall \, t\in I^{+}_{i-1} \cup I^{-}_{i-1}.
\end{equation*}
\end{lemma}

\begin{remark}\label{rem-4.2}
Similarly as in Remark~\ref{rem-4per}, in order to make the statements of Lemma~\ref{lem-forward} and Lemma~\ref{lem-back}
meaningful for each possible choice of $i\in \{1,\ldots,m\}$, when dealing with the periodic problem we shall use the cyclic agreement
$I^{-}_{0}= I^{-}_{m}$ (as above) and, moreover, $I^{+}_{m+1} = I^{+}_{1}$, $I^{+}_{0}= I^{+}_{m}$.
$\hfill\lhd$
\end{remark}

\subsection{Fixing the constants $\rho$, $\lambda$, $r$ and $R$}\label{section-4.2}

First of all, we arbitrarily choose a constant $\rho > 0$. Then,
we determine the constant $\lambda^{*} = \lambda^{*}(\rho) > 0$ according to Lemma~\ref{lem-rho}
and we take an arbitrary $\lambda > \lambda^{*}$.
Next, we fix two positive constants $r,R$ with
\begin{equation*}
0 < r < \rho < R
\end{equation*}
and such that
\begin{equation}\label{cond-rR}
\zeta(s) < \dfrac{1}{2 \lambda \,
\displaystyle{\max_{i=1,\ldots,m}}(|I^{+}_{i}|+|I^{-}_{i}|)\|a\|_{+,i}}, \quad \forall \, 0<s\leq r , \; \forall \, s \geq R,
\end{equation}
where $\zeta(s)$ is defined in \eqref{zetagamma}.
The existence of $r$ and $R$ with the above property is guaranteed by the fact that $g(s)/s\to 0^{+}$ for $s\to 0^{+}$ and
for $s\to +\infty$, namely conditions $(g_{0})$ and $(g_{\infty})$.

With this choice of $r$, $\rho$ and $R$, we consider the sets $\Omega^{\mathcal{I},\mathcal{J}}_{(r,\rho,R)}$ defined in \eqref{eq-Omega}. We are ready now to
prove Theorem~\ref{deg-Omega}, by checking that Lemma~\ref{lem-deg0} and Lemma~\ref{lem-deg1} can be applied
for $\mu>0$ sufficiently large (say $\mu>\mu^{*}(\lambda,r,R)$).

In the proofs of the next two subsections we deal with solutions satisfying $T$-periodic boundary conditions.
Accordingly, we apply Lemma~\ref{lem-rR}, Lemma~\ref{lem-rR-back}, Lemma~\ref{lem-forward} and Lemma~\ref{lem-back} with the cyclic convention
about the labelling of the intervals described in Remark~\ref{rem-4per} and Remark~\ref{rem-4.2}.

\subsection{Checking the assumptions of Lemma~\ref{lem-deg0} for $\mu$ large}\label{section-4.3}

In this section we are going to prove the first part of Theorem~\ref{deg-Omega}, that is
\begin{equation}\label{eq-deg0}
D_{L}\bigl{(}L-N_{\lambda,\mu},\Omega^{\mathcal{I},\mathcal{J}}_{(r,\rho,R)} \bigr{)} = 0,
\quad \text{ if } \; \mathcal{I} \neq \emptyset.
\end{equation}
As usual, we implicitly suppose that $\mathcal{I},\mathcal{J}\subseteq\{1,\ldots,m\}$ with $\mathcal{I}\cap\mathcal{J}=\emptyset$.

\smallskip

Given $\mathcal{I},\mathcal{J}$ as above, with $\mathcal{I}\neq\emptyset$,
it is sufficient to check that the assumptions of Lemma~\ref{lem-deg0} are satisfied,
taking as $v(t)$ the indicator function of the set $\bigcup_{i\in\mathcal{I}}I^{+}_{i}$, that is
\begin{equation*}
v(t) =
\begin{cases}
\, 1, & \text{if } t \in \bigcup_{i\in\mathcal{I}}I^{+}_{i}; \\
\, 0, & \text{if } t \in \mathopen{[}0,T\mathclose{]} \setminus \bigcup_{i\in\mathcal{I}}I^{+}_{i}.
\end{cases}
\end{equation*}

\smallskip
\noindent
\textit{Verification of $(H_{1})$. }
Let $\alpha \geq 0$. By contradiction, suppose that there exists a non-negative $T$-periodic solution $u(t)$ of \eqref{eq-lem-deg0}
with $\|u\|_{\infty}\leq R$
such that at least one of the following conditions holds:
\begin{itemize}
\item[$(a_{1})$] there is an index $i \notin \mathcal{I}\cup\mathcal{J}$ such that $\max_{t\in I^{+}_{i}} u(t) = r$;
\item[$(a_{2})$] there is an index $i \in \mathcal{I}$ such that $\max_{t\in I^{+}_{i}} u(t) = \rho$;
\item[$(a_{3})$] there is an index $i \in \mathcal{J}$ such that $\max_{t\in I^{+}_{i}} u(t) = R$.
\end{itemize}

Suppose that $(a_{1})$ holds. On the interval $I^{+}_{i}\cup I^{-}_{i}$ (with $i \notin \mathcal{I}\cup\mathcal{J}$) equation \eqref{eq-lem-deg0} reads as
\begin{equation*}
u'' + \bigl{(} \lambda a^{+}(t) - \mu a^{-}(t) \bigr{)} g(u) = 0.
\end{equation*}
Consider at first the case $u'(\sigma_{i}) \geq 0$.
By Lemma~\ref{lem-rR} (with $\vartheta=1$ and $d=r$), we have that
\begin{equation*}
u(\sigma_{i+1}) \geq r \, \biggl{(} 1 + \mu \dfrac{\gamma(r)}{2} \|A_{i}\| - \dfrac{1}{2}\biggr{)} \geq \mu \, r\dfrac{\gamma(r)}{2} \|A_{i}\|.
\end{equation*}
Thus, taking
\begin{equation*}
\mu > \hat{\mu}_{i}:= \dfrac{2R}{r\gamma(r) \|A_{i}\|},
\end{equation*}
we obtain
\begin{equation*}
u(\sigma_{i+1}) > R,
\end{equation*}
a contradiction.
On the other hand, if $u'(\sigma_{i})< 0$, by the concavity of $u(t)$ in $I^{+}_{i}$ we have that
$u'(\tau_{i})< 0$. In this case we reach the contradiction
\begin{equation*}
u(\tau_{i-1}) > R
\end{equation*}
using Lemma~\ref{lem-rR-back} (with $\vartheta=1$ and $d=r$)
and taking
\begin{equation*}
\mu > \hat{\mu}_{i-1}= \dfrac{2R}{r\gamma(r) \|A_{i-1}\|}.
\end{equation*}

Suppose that $(a_{2})$ holds. This fact contradicts Lemma~\ref{lem-rho} in view of our choice of $\lambda > \lambda^{*}$.
In this case no assumption on $\mu > 0$ is needed.

Finally, if $(a_{3})$ holds,
we obtain again a contradiction arguing as in case $(a_{1})$ and using
Lemma~\ref{lem-rR} (with $\vartheta=1$ and $d=R$). Indeed, $u'(\sigma_{i})$ cannot be negative, otherwise $u(\sigma_{i})=R$ and we get a contradiction
with $\max_{t\in I^{+}_{i}} u(t) = R = \|u\|_{\infty}$. Hence, only the instance $u'(\sigma_{i})\geq 0$ may occur
and we have a contradiction for
\begin{equation}\label{eq-check}
\mu > \check{\mu}_{i}:= \dfrac{1}{\gamma(R)\|A_{i}\|}.
\end{equation}

We conclude that $(H_{1})$ holds true for
\begin{equation*}
\mu > \mu^{(H_{1})} := \max_{i=1,\ldots,m} \bigl{\{} \hat{\mu}_{i},\check{\mu}_{i} \bigr{\}}.
\end{equation*}

\smallskip

\noindent
\textit{Verification of $(H_{2})$. }
Let $u(t)$ be an arbitrary non-negative $T$-periodic solution of \eqref{eq-lem-deg0} (with $\alpha\geq0$) such
that $u(t) \leq \rho$ for every $t\in\bigcup_{i\in\mathcal{I}} I^{+}_{i}$.

We fix an index $j\in\mathcal{I}$ and observe that on the interval $I^{+}_{j}$ equation \eqref{eq-lem-deg0}
reads as
\begin{equation*}
u'' + \lambda a^{+}(t) g(u) + \alpha = 0.
\end{equation*}
Now, we choose a constant $\varepsilon \in \mathopen{]}0, (\tau_{j}- \sigma_{j})/2\mathclose{[}$ and we notice that the inequality
\begin{equation*}
|u'(t)| \leq \dfrac{|u(t)|}{\varepsilon}, \quad \forall \, t\in \mathopen{[}\sigma_{j}+\varepsilon,\tau_{j}-\varepsilon\mathclose{]},
\end{equation*}
used in the proof of Lemma~\ref{lem-rho} is still valid.
Integrating the differential equation on $\mathopen{[}\sigma_{j}+\varepsilon,\tau_{j}-\varepsilon\mathclose{]}$ and using
the above inequality, we obtain
\begin{equation*}
\alpha \, (\tau_{i} - \sigma_{i} - 2\varepsilon)
= u'(\sigma_{i}+\varepsilon) - u'(\tau_{i}-\varepsilon) -
\lambda \int_{\sigma_{i}+\varepsilon}^{\tau_{i}-\varepsilon} a^{+}(t)g(u(t))~\!dt \leq \dfrac{2\rho}{\varepsilon}.
\end{equation*}
This yields a contradiction if $\alpha > 0$ is sufficiently large.
Hence $(H_{2})$ is verified (with $\alpha^{*}> 2\rho / \varepsilon (\tau_{i} - \sigma_{i} - 2\varepsilon)$).
Notice that for the validity of $(H_{2})$ we do not impose any condition on $\mu>0$.

\smallskip

Summing up, we can apply Lemma~\ref{lem-deg0} for $\mu > \mu^{(H_{1})}$ and therefore formula \eqref{eq-deg0} is
verified.
\qed

\subsection{Checking the assumptions of Lemma~\ref{lem-deg1} for $\mu$ large}\label{section-4.4}

In this section we are going to prove the second part of Theorem~\ref{deg-Omega}, that is
\begin{equation}\label{eq-deg1}
D_{L}\bigl{(}L-N_{\lambda,\mu},\Omega^{\emptyset,\mathcal{J}}_{(r,\rho,R)} \bigr{)} = 1,
\end{equation}
where $\mathcal{J}\subseteq\{1,\ldots,m\}$.

\smallskip

Given an arbitrary $\mathcal{J}\subseteq\{1,\ldots,m\}$,
it is sufficient to check that the assumption of Lemma~\ref{lem-deg1} is satisfied.

\smallskip
\noindent
\textit{Verification of $(H_{3})$. }
Let $\vartheta\in \mathopen{]}0,1\mathclose{]}$. By contradiction, suppose that there exists a non-negative $T$-periodic solution $u(t)$ of
\eqref{eq-lem-deg1} with $\|u\|_{\infty} \leq R$ such that at least one of the following conditions holds:
\begin{itemize}
\item[$(b_{1})$] there is an index $i \notin \mathcal{J}$ such that $\max_{t\in I^{+}_{i}} u(t) = r$;
\item[$(b_{2})$] there is an index $i \in \mathcal{J}$ such that $\max_{t\in I^{+}_{i}} u(t) = R$.
\end{itemize}

Suppose that $(b_{1})$ holds.
Consider at first the case $u'(\sigma_{i})\geq0$.
Applying Lemma~\ref{lem-rR} (with $d=r$), we obtain
\begin{equation*}
u(\sigma_{i+1}) \geq r \, \biggl{[} 1 + \dfrac{\vartheta}{2} \Bigl{(} \mu \gamma(r)\|A_{i}\| - 1 \Bigr{)} \biggr{]}
\end{equation*}
and
\begin{equation*}
u'(\sigma_{i+1}) \geq \vartheta r \biggl{(}\mu \dfrac{\gamma(r)}{2}\|a\|_{-,i} - \lambda \|a\|_{+,i} \zeta(r) \biggr{)}.
\end{equation*}
On the interval $I^{+}_{i+1}$ equation \eqref{eq-lem-deg1} yields
\begin{equation*}
u''(t) = -\vartheta \lambda a^{+}(t) g(u(t)) \geq -\vartheta\lambda a^{+}(t) g^{*}(R).
\end{equation*}
Then, integrating on $\mathopen{[}\sigma_{i+1},t\mathclose{]} \subseteq I^{+}_{i+1}$
and using the above estimates on $u(\sigma_{i+1})$ and $u'(\sigma_{i+1})$, we obtain
\begin{equation*}
\begin{aligned}
u'(t) &= u'(\sigma_{i+1}) + \int_{\sigma_{i+1}}^{t} u''(\xi)~\!d\xi \geq u'(\sigma_{i+1}) - \vartheta\lambda \|a\|_{+,i+1}g^{*}(R)
\\ &\geq \vartheta r \biggl{(}\mu \dfrac{\gamma(r)}{2}\|a\|_{-,i} -
\lambda \|a\|_{+,i} \zeta(r) -\lambda \|a\|_{+,i+1}\dfrac{g^{*}(R)}{r} \biggr{)}, \quad \forall \, t \in I^{+}_{i+1},
\end{aligned}
\end{equation*}
and
\begin{equation*}
\begin{aligned}
u(t) &= u(\sigma_{i+1}) + \int_{\sigma_{i+1}}^{t} u'(\xi)~\!d\xi \geq u(\sigma_{i+1}) - \vartheta\lambda |I^{+}_{i+1}| \|a\|_{+,i+1}g^{*}(R)
\\ &\geq r \, \biggl{[} 1 + \vartheta \biggl{(} \mu \dfrac{\gamma(r)}{2} \|A_{i}\|
- \dfrac{1}{2} - \lambda |I^{+}_{i+1}| \|a\|_{+,i+1} \dfrac{g^{*}(R)}{r} \biggr{)}\biggr{]}, \quad \forall \, t \in I^{+}_{i+1}.
\end{aligned}
\end{equation*}
Taking $\mu$ sufficiently large, precisely
\begin{equation}\label{condfor1}
\mu > \tilde{\mu}_{i}:=
\dfrac{1 + 2 \lambda |I^{+}_{i+1}| \|a\|_{+,i+1} \dfrac{g^{*}(R)}{r}}{\gamma(r) \|A_{i}\|},
\end{equation}
we obtain that
\begin{equation*}
u(t) > r, \quad u'(t)>0, \quad \forall \, t\in I^{+}_{i+1},
\end{equation*}
and, in particular,
\begin{equation*}
u(\tau_{i+1}) > r \quad \text{ and } \quad u'(\tau_{i+1})>0.
\end{equation*}
Now we can apply Lemma~\ref{lem-forward} (with $d=r$ and $D=R$) on the interval $I^{-}_{i+1} \cup I^{+}_{i+2}$, which ensures that
\begin{equation*}
u(t) > r, \quad u'(t)>0, \quad \forall \, t\in I^{-}_{i+1} \cup I^{+}_{i+2},
\end{equation*}
provided that
\begin{equation}\label{condfor2}
\mu > \mu_{i+1}^{*,+}=\mu^{*,+}(I^{-}_{i+1},I^{+}_{i+2}) =
\dfrac{\lambda \|a\|_{+,i+2} g^{*}(R)}{\|a\|_{-,i+1} g_{*}(r,R)}.
\end{equation}
Repeating inductively the same argument $m-1$ times we cover a $T$-periodicity interval with
intervals (of the form $I^{-}_{j}\cup I^{+}_{j+1}$) where the function is strictly increasing, provided that $\mu$ is sufficiently large.
More precisely, for
\begin{equation*}
\mu > \max_{i=1,\ldots,m} \mu_{i}^{*,+}
\end{equation*}
it holds that
\begin{equation*}
u(t) > r, \quad u'(t)>0, \quad \forall \, t\in \mathopen{[}0,T\mathclose{]}.
\end{equation*}
This clearly contradicts the $T$-periodicity of $u(t)$.

Consider now the case $u'(\sigma_{i}) < 0$, which implies (by the concavity of $u(t)$ in $I^{+}_{i}$) that
$u'(\tau_{i})< 0$. The same proof as above leads to a contradiction, proceeding backward
and using at first Lemma~\ref{lem-rR-back} (with $d=r$) and then Lemma~\ref{lem-back} (with $d=r$ and $D=R$), inductively.
Conditions \eqref{condfor1} and \eqref{condfor2}
will be replaced by analogous inequalities of the form
\begin{equation*}
\mu > \bar{\mu}_{i}:=
\dfrac{1 + 2\lambda |I^{+}_{i-1}| \|a\|_{+,i-1} \dfrac{g^{*}(R)}{r}}{\gamma(r) \|A_{i-1}\|},
\end{equation*}
and
\begin{equation*}
\mu > \mu_{i-1}^{*,-}=\mu^{*,-}(I^{+}_{i-2},I^{-}_{i-2})=
\dfrac{\lambda \|a\|_{+,i-2} g^{*}(R)}{\|a\|_{-,i-2}g_{*}(r,R)},
\end{equation*}
so that a contradiction comes for
\begin{equation*}
\mu > \max_{i=1,\ldots,m} \mu_{i}^{*,-},
\end{equation*}
by showing that $u'(t) < 0$ for all $t \in \mathopen{[}0,T\mathclose{]}$.

Taking into account all the possible situations we conclude that the case $(b_{1})$ never occurs if
\begin{equation*}
\mu > \mu_{1}^{(H_{3})} : = \max_{i=1,\ldots,m} \bigl{\{} \tilde{\mu}_{i}, \bar{\mu}_{i}, \mu_{i}^{*,+}, \mu_{i}^{*,-} \bigr{\}}.
\end{equation*}

To conclude the proof, suppose now that $(b_{2})$ holds. As observed in the previous proof,
the fact that $\max_{t\in I^{+}_{i}} u(t) = R = \|u\|_{\infty}$ prevents the possibility that
$u'(\sigma_{i}) < 0$. Hence only the instance $u'(\sigma_{i})\geq 0$ may occur.
Applying Lemma~\ref{lem-rR} (with $d=R$), we obtain
\begin{equation*}
u(\sigma_{i+1}) \geq R \,
\biggl{[} 1 + \dfrac{\vartheta}{2} \Bigl{(} \mu \gamma(R) \|A_{i}\| - 1 \Bigr{)} \biggr{]}.
\end{equation*}
Hence, if
\begin{equation*}
\mu > \check{\mu}_{i} = \dfrac{1}{\gamma(R)\|A_{i}\|}
\end{equation*}
(already defined in \eqref{eq-check}) we get $u(\sigma_{i+1}) > R$ and thus a contradiction with $\|u\|_{\infty} \leq R$.
We conclude that the case $(b_{2})$ never occurs if
\begin{equation*}
\mu > \mu_{2}^{(H_{3})} : = \max_{i=1,\ldots,m} \check{\mu}_{i}.
\end{equation*}

\smallskip

Summing up, we can apply Lemma~\ref{lem-deg1} for
\begin{equation*}
\mu > \mu^{(H_{3})}:=\max \Bigl{\{} \mu_{1}^{(H_{3})},\mu_{2}^{(H_{3})}, \mu^{\#}(\lambda) \Bigr{\}}
\end{equation*}
and therefore formula \eqref{eq-deg1} is
verified.
\qed

\subsection{Completing the proof of Theorem~\ref{main-theorem}}\label{section-4.5}

With reference to Section~\ref{section-3} we summarize what we have proved until now and
we give the final details of the proof of our main theorem.

First, we have fixed an arbitrary constant $\rho > 0$ and determined a constant $\lambda^{*} = \lambda^{*}(\rho) > 0$
via Lemma~\ref{lem-rho}. We stress the fact that $\lambda^{*}$ depends only on $g(s)$ for $s\in \mathopen{[}0,\rho\mathclose{]}$
and on the behavior of $a(t)$ in each of the intervals $I^{+}_{i}$.

Next, for $\lambda > \lambda^{*}$, we have found two constants (a small one $r$ and a large one $R$)
with $0 < r < \rho < R$ such that condition \eqref{cond-rR} holds. To choose $r$ and $R$ we only require
conditions on the smallness of $g(s)/s$ for $s$ near zero and near infinity, which is an obvious consequence of
$(g_{0})$ and $(g_{\infty})$. We notice also that condition \eqref{cond-rR} depends
on the behavior of $a(t)$ in each of the intervals $I^{+}_{i}$ as well as on the lengths of
pairs of consecutive intervals.

As a further step, we have shown that both Lemma~\ref{lem-deg0} and Lemma~\ref{lem-deg1} can be applied
provided that
\begin{equation*}
\mu > \mu^{*}(\lambda) = \mu^{*}(\lambda,r,R):= \max \Bigl{\{}\mu^{(H_{1})},\mu^{(H_{3})} \Bigr{\}}.
\end{equation*}
Checking carefully the estimates leading to $\mu^{(H_{1})}$ and $\mu^{(H_{3})}$ one realizes that again only
local conditions about the behavior of $a(t)$ on the intervals $I^{\pm}_{i}$ are involved.

As a consequence, for all $\mu > \mu^{*}(\lambda)$, formula \eqref{eq-3.1} in Theorem~\ref{deg-Omega} holds. From this latter
result, via a purely combinatorial argument (independent on the particular equation under consideration), we achieve
formula \eqref{eq-3.2} in Theorem~\ref{deg-Lambda} and the existence of $3^{m}-1$ positive $T$-periodic solutions to
\eqref{eq-main} is guaranteed, as already explained at the end of Section~\ref{section-3.1}.
\qed

\section{General properties for globally defined solutions and some a posteriori bounds}\label{section-5}

In this section we focus our attention to non-negative solutions of \eqref{eq-main} which are defined for all $t \in \mathbb{R}$.
On one hand, we show how some computations in the proofs of the technical lemmas in Section~\ref{section-4} are still valid in this setting;
this will be useful in view of further applications of Theorem~\ref{main-theorem} described in Section~\ref{section-6}.
On the other hand, we provide some additional information for the solutions when $\mu\to+\infty$.

In order to avoid repetitions, throughout this section we assume that the constants $\rho > 0$, $\lambda > \lambda^{*}$,
$0 < r < \rho <R$ and $\mu > \mu^{*}(\lambda)$ are all fixed as in Section~\ref{section-4.2} and Section~\ref{section-4.5}.
We stress the fact that even if these constants have been determined with respect to the $T$-periodic problem, all the results below are
valid for arbitrary globally defined non-negative solutions.

The first result concerns the behavior of the solutions with respect to the constant $R$.

\begin{proposition}\label{prop-5.1}
Let $g \colon \mathbb{R}^{+} \to \mathbb{R}^{+}$ be a continuous function satisfying $(g_{*})$, $(g_{0})$ and $(g_{\infty})$.
Let $a \colon \mathbb{R} \to \mathbb{R}$ be a $T$-periodic locally integrable function satisfying $(a_{*})$.
If $w(t)$ is any non-negative solution of \eqref{eq-main} with $\sup_{t\in\mathbb{R}} w(t) \leq R$, then
$w(t) < R$ for all $t\in \mathbb{R}$.
\end{proposition}

\begin{proof}
Suppose by contradiction that there exists $t^{*}\in \mathbb{R}$ such that $w(t^{*}) = \max_{t\in\mathbb{R}} w(t) = R$.
Let also $\ell \in \mathbb{Z}$ be such that $t^{*}\in \mathopen{[}\ell T,(\ell+1)T\mathclose{]}$.
In this case, thanks to the $T$-periodicity of the weight coefficient $a_{\lambda,\mu}(t)$,
the function $u(t):= w(t+\ell T)$ is still a (non-negative) solution of \eqref{eq-main} with
$\max_{t\in \mathopen{[}0,T\mathclose{]}} u(t) = u(t^{*}-\ell T) = w(t^{*}) = R$. From now on,
the proof uses exactly the same argument as for the discussion of the case $(a_{3})$ in the verification of $(H_{1})$ in Section~\ref{section-4.3}
(for $\alpha = 0$) and the same contradiction can be achieved.
\end{proof}

A straightforward application of Lemma~\ref{lem-rho} gives the following result (the obvious proof is omitted).

\begin{proposition}\label{prop-5.2}
Let $g \colon \mathbb{R}^{+} \to \mathbb{R}^{+}$ be a continuous function satisfying $(g_{*})$, $(g_{0})$ and $(g_{\infty})$.
Let $a \colon \mathbb{R} \to \mathbb{R}$ be a $T$-periodic locally integrable function satisfying $(a_{*})$.
If $w(t)$ is any non-negative solution of \eqref{eq-main} and $I^{+}_{i,\ell}:= I^{+}_{i} + \ell T$ is any interval of the real line where $a(t)\succ 0$,
then $\max_{t\in I^{+}_{i,\ell}} w(t) \neq \rho$.
\end{proposition}

The next result concerns the behavior of the solutions with respect to the constant $r$.

\begin{proposition}\label{prop-5.3}
Let $g \colon \mathbb{R}^{+} \to \mathbb{R}^{+}$ be a continuous function satisfying $(g_{*})$, $(g_{0})$ and $(g_{\infty})$.
Let $a \colon \mathbb{R} \to \mathbb{R}$ be a $T$-periodic locally integrable function satisfying $(a_{*})$.
If $w(t)$ is any non-negative solution of \eqref{eq-main} with $\sup_{t\in\mathbb{R}} w(t) \leq R$
and $I^{+}_{i,\ell}:= I^{+}_{i} + \ell T$ is any interval of the real line where $a(t)\succ 0$,
then $\max_{t\in I^{+}_{i,\ell}} w(t) \neq r$.
\end{proposition}

\begin{proof}
We follow the same scheme as for Proposition~\ref{prop-5.1}.
Suppose by contradiction that there exists $t^{*}\in I^{+}_{i,\ell}$ such that $w(t^{*}) = \max_{t\in I^{+}_{i,\ell}} w(t) = r$.
The function $u(t):= w(t+\ell T)$ is a non-negative solution of \eqref{eq-main} with
$\max_{t\in I^{+}_{i}} u(t) = w(t^{*}) = r$. From now on,
the proof uses exactly the same argument as for the discussion of the case $(a_{1})$ in the verification of $(H_{1})$ in Section~\ref{section-4.3}
(for $\alpha = 0$) and the same contradiction can be achieved, in the sense that we find a point where $w(t) > R$.
\end{proof}

\medskip

We now focus on some properties of globally defined non-negative solutions of \eqref{eq-main} when $\mu\to+\infty$.
The first result in this direction concerns the behavior on the intervals where $a(t)\succ0$: roughly speaking, any ``very small'' solution
becomes arbitrarily small as $\mu\to+\infty$.

\begin{proposition}\label{prop-5.4}
Let $g \colon \mathbb{R}^{+} \to \mathbb{R}^{+}$ be a continuous function satisfying $(g_{*})$, $(g_{0})$ and $(g_{\infty})$.
Let $a \colon \mathbb{R} \to \mathbb{R}$ be a $T$-periodic locally integrable function satisfying $(a_{*})$.
Then for every $\varepsilon$ with $0 < \varepsilon \leq r$ there exists $\mu^{\star}_{\varepsilon} \geq \mu^{*}(\lambda)$
such that for any fixed $\mu > \mu^{\star}_{\varepsilon}$ the following holds:
if $w(t)$ is any non-negative solution of \eqref{eq-main} with $\sup_{t\in\mathbb{R}} w(t) \leq R$ and $\max_{t\in I^{+}_{i,\ell}} w(t) \leq r$,
where $I^{+}_{i,\ell}:= I^{+}_{i} + \ell T$ is any interval of the real line where $a(t)\succ 0$,
then $\max_{t\in I^{+}_{i,\ell}} w(t) < \varepsilon$.
\end{proposition}

\begin{proof}
Repeating the same approach as in the proof of the previous propositions and using the $T$-periodicity of the weight,
without loss of generality, we can restrict ourselves to the analysis of the non-negative solution $w(t)$ on an interval
$I^{+}_{i}$, for $i=1,\ldots,m$.

The proof uses exactly the same argument as for the discussion of the case $(a_{1})$ in the verification of $(H_{1})$ in Section~\ref{section-4.3}
(for $\alpha = 0$). Let $\varepsilon\in\mathopen{]}0,r\mathclose{]}$. By contradiction, suppose that there exists a non-negative solution $w(t)$ of \eqref{eq-main}
such that $\sup_{t\in\mathbb{R}} w(t) \leq R$ and $\max_{t\in I^{+}_{i}} w(t) = \varepsilon_{0} \in \mathopen{[}\varepsilon,r\mathclose{]}$.
Consider at first the case $w'(\sigma_{i}) \geq 0$.
Recalling condition \eqref{cond-rR}, by Lemma~\ref{lem-rR} (with $\vartheta=1$ and $d=\varepsilon_{0}$), we have that
\begin{equation*}
w(\sigma_{i+1}) \geq \mu \, \varepsilon_{0} \dfrac{\gamma(\varepsilon_{0})}{2} \|A_{i}\|.
\end{equation*}
Observing that
\begin{equation*}
\gamma(\varepsilon_{0})=\min_{\frac{\varepsilon_{0}}{2}\leq s\leq\varepsilon_{0}}\dfrac{g(s)}{s}
\geq \min_{\frac{\varepsilon}{2}\leq s\leq r}\dfrac{g(s)}{s} =: \gamma^{*}(\varepsilon,r) >0
\end{equation*}
and thus taking
\begin{equation*}
\mu > \mu^{\star}_{i}(\varepsilon) := \dfrac{2R}{\varepsilon\gamma^{*}(\varepsilon,r) \|A_{i}\|},
\end{equation*}
we obtain $w(\sigma_{i+1}) > R$, a contradiction.
On the other hand, if $w'(\sigma_{i})< 0$, by the concavity of $w(t)$ in $I^{+}_{i}$ we have that
$w'(\tau_{i})< 0$. In this case we reach the contradiction $w(\tau_{i-1}) > R$
using Lemma~\ref{lem-rR-back} (with $\vartheta=1$ and $d=\varepsilon_{0}$)
and taking
\begin{equation*}
\mu > \mu^{\star}_{i-1}(\varepsilon) =\dfrac{2R}{\varepsilon\gamma^{*}(\varepsilon,r) \|A_{i-1}\|}
\end{equation*}
(if $i = 1$, we count cyclically and consider the interval $I^{-}_{0}$ as $I^{+}_{m}$).
In conclusion, taking
\begin{equation*}
\mu > \mu^{\star}_{\varepsilon} := \max_{i=1,\ldots,m} \bigl{\{} \mu^{\star}_{i}(\varepsilon),\mu^{*}(\lambda){\}},
\end{equation*}
the proposition follows.
\end{proof}

Our final result in this section concerns the behavior of non-negative solutions to \eqref{eq-main} on the intervals where
$a(t)\prec 0$. With reference to condition $(a_{*})$, for technical reasons we further suppose that \textit{$a(t)\not\equiv 0$ in each right neighborhood of $\tau_{i}$
and in each left neighborhood of $\sigma_{i+1}$}. Such an assumption
does not require any new constraint on the weight function, but just a more careful selection of the points $\tau_{i}$ and $\sigma_{i+1}$.
What we mean is that for a weight function $a(t)$ satisfying $(a_{*})$ the way to select the intervals $I^{+}_{i}$ and
$I^{-}_{i}$ may be not univocal. Indeed, we could have an interval $J$ where $a(t)\equiv 0$ between an interval of positivity and an interval
of negativity for the weight. Up to now the decision whether incorporate such an interval $J$ as a part of $I^{+}_{i}$ or
$I^{-}_{i}$ was completely arbitrary. On the contrary, for the next result, we prefer to consider an interval as
$J$ as a part of $I^{+}_{i}$. In any case, we can allow a closed interval where $a(t)\equiv 0$ to lie in the interior
of one of the $I^{-}_{i}$. With this in mind, we can now present our next result.

\begin{proposition}\label{prop-5.5}
Let $g \colon \mathbb{R}^{+} \to \mathbb{R}^{+}$ be a continuous function satisfying $(g_{*})$, $(g_{0})$ and $(g_{\infty})$.
Let $a \colon \mathbb{R} \to \mathbb{R}$ be a $T$-periodic locally integrable function satisfying $(a_{*})$.
Then for every $\varepsilon$ with $0 < \varepsilon \leq r$ there exists $\mu^{\star\star}_{\varepsilon} \geq \mu^{*}(\lambda)$
such that for any fixed $\mu > \mu^{\star\star}_{\varepsilon}$ the following holds:
if $w(t)$ is any non-negative solution of \eqref{eq-main} with $\sup_{t\in\mathbb{R}} w(t) \leq R$
and $I^{-}_{i,\ell}:= I^{-}_{i} + \ell T$ is any interval of the real line where $a(t)\prec 0$,
then $\max_{t\in I^{-}_{i,\ell}} w(t) < \varepsilon$.
\end{proposition}

\begin{proof}
Without loss of generality, we can restrict ourselves to the analysis of the non-negative solution $w(t)$ on an interval
$I^{-}_{i}$, for $i=1,\ldots,m$.

Given $\varepsilon \in \mathopen{]}0,r\mathclose{]}$, we consider the values of the solution $w(t)$ at the boundary of the interval $I^{-}_{i}$,
for an arbitrary but fixed index $i\in\{1,\ldots,m\}$.
If $w(\tau_{i}) < \varepsilon$ and $w(\sigma_{i+i}) < \varepsilon$, then, by convexity,
$w(t) < \varepsilon$ for all $t\in I^{-}_{i}$ and we have nothing to prove.
Therefore, we discuss only the cases when $w(\tau_{i}) \geq \varepsilon$ or $w(\sigma_{i+1}) \geq \varepsilon$. We are going to show that this cannot occur if $\mu$ is sufficiently large.
Accordingly, suppose that $w(\tau_{i}) \geq \varepsilon$.
Knowing that $w(t) \leq R$ on the whole real line, in particular in the interval $I^{+}_{i}$, we easily find that
there is at least a point $t_{0}\in I^{+}_{i}$ such that $|w'(t_{0})| \leq R/|I^{+}_{i}|$. On the other hand, equation \eqref{eq-main}
on $I^{+}_{i}$ reads as $w'' = - \lambda a^{+}(t) g(w)$, so that an integration on $\mathopen{[}t_{0},\tau_{i}\mathclose{]}$ yields
\begin{equation*}
w'(\tau_{i}) = w'(t_{0}) - \lambda \int_{t_{0}}^{\tau_{i}} a^{+}(t) g(w(t)) ~\!dt \geq - \dfrac{R}{|I^{+}_{i}|} - \lambda \|a\|_{+,i}g^{*}(R) =:- \kappa_{i},
\end{equation*}
where the constants $\|a\|_{+,i}$ and $g^{*}(R)$ are those defined at the beginning of Section~\ref{section-4}.
The convexity of $w(t)$
in $I^{-}_{i}$ guarantees that $w'(t) \geq - \kappa_{i}$ for all $t \in I^{-}_{i}$. Hence,
if we fix a constant $\delta_{i} > 0$ with $\tau_{i} + \delta_{i} < \sigma_{i+1}$ and such that
$\delta_{i} < \varepsilon/(2 \kappa_{i})$, it is clear that $w(t) \geq \varepsilon/2$ for all $t\in \mathopen{[}\tau_{i},\tau_{i} + \delta_{i}\mathclose{]}$.
On the interval $I^{-}_{i}$ equation \eqref{eq-main}
reads as $w'' =\mu a^{-}(t) g(w)$, so that an integration on $\mathopen{[}\tau_{i},t\mathclose{]} \subseteq \mathopen{[}\tau_{i},\tau_{i} + \delta_{i}\mathclose{]}$ yields
\begin{equation*}
w'(t) = w'(\tau_{i}) + \mu \int_{\tau_{i}}^{t} a^{-}(\xi) g(w(\xi)) ~\!d\xi \geq - \kappa_{i} + \mu A_{i}(t) \,g_{*}(\varepsilon/2,R),
\end{equation*}
where the function $A_{i}(t)$ and the constant $g_{*}(\varepsilon/2,R)$ are defined at the beginning of Section~\ref{section-4}.
Since we have supposed that $a^{-}(t)$ is not identically zero in each right neighborhood of $\tau_{i}$, we know that
the function $A_{i}(t):=\int_{\tau_{i}}^{t} a^{-}(\xi)~\!d\xi$ is strictly positive for each $t \in \mathopen{]}\tau_{i},\sigma_{i+1}\mathclose{]}$.
Then, integrating the above inequality on $\mathopen{[}\tau_{i},\tau_{i} + \delta_{i}\mathclose{]}$, we obtain
\begin{equation*}
w(\tau_{i} + \delta_{i}) = w(\tau_{i}) + \int_{\tau_{i}}^{\tau_{i} + \delta_{i}} w'(t)~\!dt \geq
\varepsilon - \kappa_{i} \delta_{i} + \mu \,g_{*}(\varepsilon/2,R) \int_{\tau_{i}}^{\tau_{i} + \delta_{i}}A_{i}(t)~\!dt.
\end{equation*}
This latter inequality implies $w(\tau_{i} + \delta_{i}) > R$ (and hence a contradiction) for
\begin{equation*}
\mu > \mu^{\rm{left}}_{i}(\varepsilon):=\dfrac{R + \kappa_{i} \delta_{i}}{g_{*}(\varepsilon/2,R)\int_{\tau_{i}}^{\tau_{i} + \delta_{i}}A_{i}(t)~\!dt}.
\end{equation*}
On the other hand, if we suppose that $w(\sigma_{i+1}) \geq \varepsilon$, then by the same argument we have
\begin{equation*}
w'(\sigma_{i+1}) \leq \kappa_{i+i} := \dfrac{R}{|I^{+}_{i+1}|} + \lambda \|a\|_{+,i+1}g^{*}(R)
\end{equation*}
(if $i = m$, we count cyclically and consider the interval $I^{+}_{m+1}$ as $I^{+}_{1}$). As before,
we fix a constant $\delta_{i+1} > 0$ with $\sigma_{i+1} - \delta_{i+1} > \tau_{i}$ and such that
$\delta_{i+1} < \varepsilon/(2\kappa_{i+1})$, so that $u(t) \geq \varepsilon /2$ for all
$t\in \mathopen{[}\sigma_{i+1} - \delta_{i+1},\sigma_{i+1}\mathclose{]}$.
An integration of the equation on $\mathopen{[}t,\sigma_{i+1}\mathclose{]}$ yields
\begin{equation*}
w'(t) \leq \kappa_{i+1} - \mu B_{i}(t)\,g_{*}(\varepsilon/2,R),
\end{equation*}
where we have set $B_{i}(t):= \int_{t}^{\sigma_{i+1}} a^{-}(\xi)~\!d\xi$.
Since we have supposed that $a^{-}(t)$ is not identically zero in each left neighborhood of $\sigma_{i+1}$, we know that
the function $B_{i}(t)$ is strictly positive for each $t \in \mathopen{[}\tau_{i},\sigma_{i+1}\mathclose{[}$.
Then, integrating the above inequality on $\mathopen{[}\sigma_{i+1} - \delta_{i+1},\sigma_{i+1}\mathclose{]}$, we obtain
\begin{equation*}
w(\sigma_{i+1} - \delta_{i+1}) \geq
\varepsilon - \kappa_{i+1} \delta_{i+1} + \mu g_{*}(\varepsilon/2,R) \int_{\sigma_{i+1} - \delta_{i+1}}^{\sigma_{i+1}}B_{i}(t)~\!dt.
\end{equation*}
This latter inequality implies $w(\sigma_{i+1} - \delta_{i+1}) > R$ (and hence a contradiction) for
\begin{equation*}
\mu > \mu^{\rm{right}}_{i}(\varepsilon) := \dfrac{R + \kappa_{i+1} \delta_{i+1}}{g_{*}(\varepsilon/2,R)\int_{\sigma_{i+1} - \delta_{i+1}}^{\sigma_{i+1}}B_{i}(t)~\!dt}.
\end{equation*}

In conclusion, for
\begin{equation}\label{eq-mu**}
\mu > \mu^{\star\star}_{\varepsilon} := \max_{i=1,\ldots,m}\bigl{\{}\mu^{\rm{left}}_{i}(\varepsilon),\mu^{\rm{right}}_{i}(\varepsilon),\mu^{*}(\lambda)\bigr{\}}
\end{equation}
our result is proved.
\end{proof}

We conclude this section by briefly describing, as typical in singular perturbation problems, the limit behavior of
positive solutions of \eqref{eq-main} for $\mu \to +\infty$ (compare with \cite{BaBoVe-15},
where a similar discussion was performed in the superlinear case).
We focus our attention to the solutions found in Theorem~\ref{main-theorem} for the $T$-periodic problem;
however, similar considerations are valid for Dirichlet and Neumann boundary conditions,
as well as for globally defined positive solutions.

Let us fix a non-null string $\mathcal{S} \in \{0,1,2\}^{m}$. Theorem~\ref{main-theorem} ensures the existence (in general, not the uniqueness)
of a positive $T$-periodic solution of \eqref{eq-main} associated with it, if $\lambda > \lambda^{*}$ and $\mu > \mu^{*}(\lambda)$;
in order to emphasize its dependence on the parameter $\mu$, we will denote it by $u_{\mu}(t)$.
Then, as a direct consequence of Proposition~\ref{prop-5.4} and Proposition~\ref{prop-5.5}, we have that $u_{\mu}(t)$ converges uniformly to zero
both in the intervals $I^{+}_{i}$ with $\mathcal{S}_{i}=0$ as well as in the intervals $I^{-}_{i}$, for $\mu \to +\infty$.
As for the behavior of $u_{\mu}(t)$ on the intervals $I^{+}_{i}$ such that $\mathcal{S}_{i}\in\{1,2\}$,
with a standard compactness argument (based on the facts that $0 \leq u_{\mu}(t) \leq R$ and that
equation \eqref{eq-main} is independent on the parameter $\mu$ in the intervals $I^{+}_{i}$), we can prove that
the family $\{u_{\mu}|_{I^{+}_{i}}\}_{\mu > \mu^{*}(\lambda)}$ is relatively compact in $\mathcal{C}(I^{+}_{i})$ and that each of its
cluster points $u_{\infty}(t)$ has to be a non-negative solution of $u'' + \lambda a^{+}(t)g(u) = 0$ on $I^{+}_{i}$.
We claim that $u_{\infty}(t)$ is actually a positive solution, satisfies Dirichlet boundary condition on $I^{+}_{i}$ and is
``small'' if $\mathcal{S}_{i} = 1$ and ``large'' if $\mathcal{S}_{i} = 2$. Indeed, the first assertion follows from the fact that, passing to the limit,
$r \leq \max_{t \in I^{+}_{i}} u_{\infty}(t) \leq \rho$ if $\mathcal{S}_{i} = 1$ and
$\rho \leq \max_{t \in I^{+}_{i}} u_{\infty}(t) \leq R$ if $\mathcal{S}_{i} = 2$. As for Dirichlet boundary condition on $I^{+}_{i}$, this is a consequence
of $u_{\mu}(t) \to 0$ on every interval of negativity. Finally, using Lemma~\ref{lem-rho}, we infer $r \leq \max_{t \in I^{+}_{i}} u_{\infty}(t) < \rho$
if $\mathcal{S}_{i} = 1$ (that is, $u_{\infty}(t)$ is ``small'')  and
$\rho < \max_{t \in I^{+}_{i}} u_{\infty}(t) \leq R$ if $\mathcal{S}_{i} = 2$ (that is, $u_{\infty}(t)$ is ``large'').

In conclusion, up to subsequences, $u_{\mu}(t) \to u_{\infty}(t)$ uniformly for $\mu \to +\infty$, with $u_{\infty}(t)$ a function
made up of ``null'', ``small'' and ``large'' solutions of Dirichlet problems in the intervals $I^{+}_{i}$
(depending on $\mathcal{S}_{i} = 0,1,2$ respectively) connected by null functions in $I^{-}_{i}$.
See Figure~\ref{fig-02} for a numerical simulation.
Notice that this discussion is simplified whenever we are able to prove that each Dirichlet problem associated with
$u'' + \lambda a^{+}(t) g(u) = 0$ on $I^{+}_{i}$ has \textit{exactly} two positive solutions;
indeed, in this case every string $\mathcal{S} \in \{0,1,2\}^{m}$ uniquely determines a limit profile
$u_{\infty}(t)$ and $u_{\mu}(t) \to u_{\infty}(t)$ uniformly, without the need of taking subsequences
(even if $u_{\mu}(t)$ could be not unique in the class of positive solutions to \eqref{eq-main} associated with $\mathcal{S}$).

\begin{figure}[h!]
\centering
\begin{tikzpicture}[x=60pt,y=15pt]
\node at (1.555,-7.778) {\includegraphics[width=0.684\textwidth]{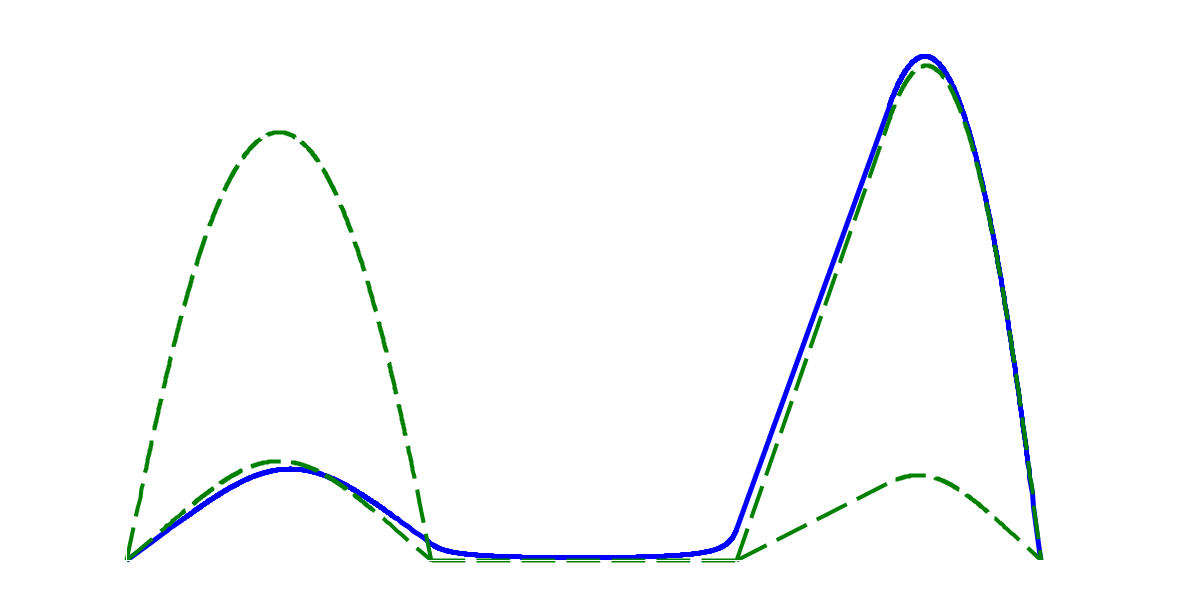}};
\draw (-0.4,0) -- (3.5,0);
\draw (0,-2.9) -- (0,4.2);
\draw (3.4,0) node[anchor=south] {$t$};
\draw (0,3.8) node[anchor=east] {$a(t)$};
\draw [line width=1pt, color=red] (0,1) -- (1,1);
\draw [line width=1pt, color=red] (1,-2) -- (2,-2);
\draw [line width=1pt, color=red] (2,0) -- (2.5,0);
\draw [line width=1pt, color=red] (2.5,2) -- (3,2);
\draw [dashed] (1,1) -- (1,-2);
\draw [dashed] (2,-2) -- (2,0);
\draw [dashed] (2.5,0) -- (2.5,2);
\draw [dashed] (3,2) -- (3,0);
\draw (0,0) node[anchor=north east] {$0$};
\draw (1,0) node[anchor=north east] {$1$};
\draw (1,-0.05) -- (1,0);
\draw (2,0) node[anchor=north east] {$2$};
\draw (2,-0.05) -- (2,0);
\draw (3,0) node[anchor=north] {$3$};
\draw (3,-0.1) -- (3,0);
\draw (0,2) node[anchor=east] {$2$};
\draw (-0.03,2) -- (0,2);
\draw (0,1) node[anchor=east] {$1$};
\draw (-0.03,1) -- (0,1);
\draw (0,-2) node[anchor=east] {$-2$};
\draw (-0.03,-2) -- (0,-2);
\draw (0,-4.4) node[anchor=east] {$u(t)$};
\draw (3.4,-11.2) node[anchor=south] {$t$};
\draw (-0.4,-11.2) -- (3.5,-11.2);
\draw (0,-11.7) -- (0,-4);
\draw (1,-11.3) -- (1,-11.2);
\draw (2,-11.3) -- (2,-11.2);
\draw (3,-11.3) -- (3,-11.2);
\end{tikzpicture}
\caption{\small{The lower part of the figure shows a positive solution of equation \eqref{eq-main} for the super-sublinear
nonlinearity $g(s) = \arctan(s^{3})$, for $s\geq0$, and Dirichlet boundary conditions.
For this simulation we have chosen the interval $\mathopen{[}0,T\mathclose{]}$ with $T=3$ and the weight function
$a_{\lambda,\mu}(t)$ with $a(t)$ having a stepwise graph as represented in the upper part of the figure.
First, with a dashed line we have drawn the Dirichlet solutions (``small'' and ``large'') on the intervals $\mathopen{[}0,1\mathclose{]}$
and $\mathopen{[}2,3\mathclose{]}$. Then, for $\lambda = 20$ and $\mu = 10000$, we have exhibited
a solution of the form ``small'' in the first interval of
positivity $\mathopen{[}0,1\mathclose{]}$ and ``large'' in the second interval of positivity $\mathopen{[}2,3\mathclose{]}$.
Such a solution is very close to the limit profile for the class of solutions  associated with the string $(1,2)$, which is made by a ``small'' solution of the
Dirichlet problem in $\mathopen{[}0,1\mathclose{]}$ and a ``large'' solution of the Dirichlet problem in $\mathopen{[}2,3\mathclose{]}$
connected by the null solution in $\mathopen{[}1,2\mathclose{]}$.
Notice that, for the given weight function
which is identically zero on the interval $\mathopen{[}2,2.5\mathclose{]}$ separating the negative and the positive hump, the solution is very small
(and the limit profile is zero) only in the interval $\mathopen{[}1,2\mathclose{]}$ where the weight is negative. This is in complete accordance with
Proposition~\ref{prop-5.5} and the choice of the endpoints of the intervals $I^{\pm}_{i}$.}}
\label{fig-02}
\end{figure}
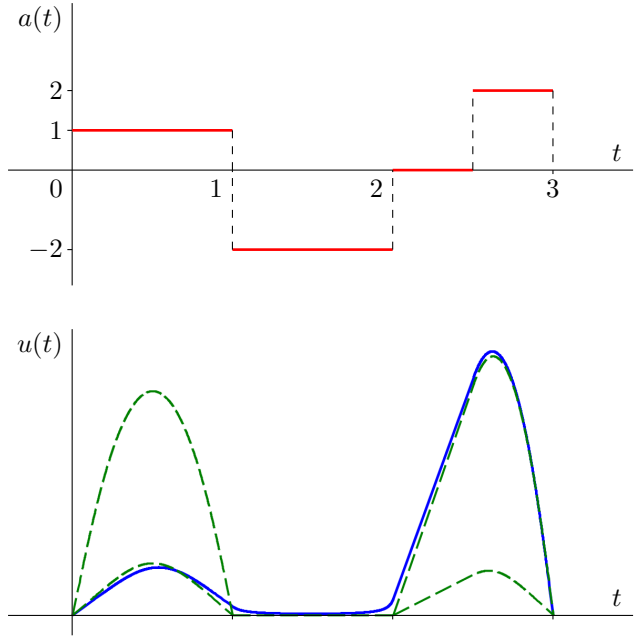

\section{Subharmonics and symbolic dynamics}\label{section-6}

As remarked in Section~\ref{section-4.5}, the estimates which allow to determine the value $\lambda^{*}$,
$r$, $R$ and $\mu^{*}(\lambda)$ are of local nature. We exploit this fact by applying Theorem~\ref{main-theorem}
on intervals of the form $\mathopen{[}0,kT\mathclose{]}$, with $k\geq 2$ an integer, and thus proving the existence of subharmonic solutions.
Next, letting $k\to \infty$ and using a Krasnosel'ski\u{\i}-Mawhin lemma for bounded solutions, we
obtain positive bounded solutions which are not necessarily periodic and can reproduce an arbitrary coin-tossing sequence.
This is a hint of complex dynamics and indeed we conclude the section by describing some dynamical consequences of our results.

Throughout the section we suppose that $g \colon {\mathbb{R}}^{+} \to {\mathbb{R}}^{+}$ is a continuous function satisfying
$(g_{*})$ as well as $(g_{0})$ and $(g_{\infty})$ and $a \colon \mathbb{R} \to \mathbb{R}$ is a $T$-periodic locally
integrable function satisfying $(a_{*})$. For convenience in the next discussion, we also suppose that $T > 0$ is the
\textit{minimal period} of $a(t)$. Moreover, we recall the notation
\begin{equation}\label{eq-Ipm}
I^{\pm}_{i,\ell}:= I^{\pm}_{i} + \ell\, T, \quad \text{ for } \; i=1,\ldots,m \; \text{ and } \; \ell \in \mathbb{Z}.
\end{equation}

\subsection{Positive subharmonic solutions}\label{section-6.1}

In this subsection we investigate the existence and multiplicity of \textit{positive subharmonic solutions}
to equation \eqref{eq-main}. Let $k\geq 2$ be a fixed integer. Following a standard definition,
we recall that a \textit{subharmonic solution of order $k$} is a $kT$-periodic solution which is not $lT$-periodic
for any integer $l=1,\ldots,k-1$. As observed in \cite{FeZa-pp2015} (as a consequence of $(g_{*})$
and the fact that $T>0$ is the minimal period of $a(t)$) any positive subharmonic solution of order $k$ has actually $kT$ as
minimal period.

As a further remark, we observe that if $u(t)$ is a positive $kT$-periodic solution of \eqref{eq-main}
then, for any integer $\ell$ with $1 \leq \ell \leq k-1$, also its time-translated $v_{\ell}(t):= u(t + \ell T)$
is a positive $kT$-periodic solution of the same minimal period.
Therefore, if we find a subharmonic solution
of order $k$, we also obtain altogether a family of $k$ subharmonic solutions of the same order.
These solutions, even if formally distinct,
will be considered as belonging to the same periodicity class.

We split the search of subharmonic solutions to \eqref{eq-main} into two steps.
In the first one we present a theorem of existence and multiplicity of positive $kT$-periodic solutions which is
a direct application of Theorem~\ref{main-theorem} for the interval $\mathopen{[}0,kT\mathclose{]}$.
As a second step, we show how the code ``very small/small/large'' allows us to prove the minimality
of the period for some of such $kT$-periodic solutions and determine a lower bound for the number of
$k$-th order subharmonics.

First of all, in order to apply Theorem~\ref{main-theorem} to the interval $\mathopen{[}0,kT\mathclose{]}$, we need to observe that
now $a(t)$ is treated as a $kT$-periodic function (even if it has $T$ as minimal period). Recalling the notation in \eqref{eq-Ipm},
in the ``new'' periodicity interval $\mathopen{[}0,kT\mathclose{]}$ the weight $a(t)$ turns out to be
a function with $km$ positive humps $I^{+}_{i,\ell}$ separated by $km$ negative ones $I^{-}_{i,\ell}$
(for $i=1,\ldots, m$ and $\ell= 0,\ldots, k-1$).

\medskip

In this setting, Theorem~\ref{main-theorem} reads as follows.

\begin{theorem}\label{main-theorem-sub}
Let $g \colon \mathbb{R}^{+} \to \mathbb{R}^{+}$ be a continuous function satisfying $(g_{*})$, $(g_{0})$ and $(g_{\infty})$.
Let $a \colon \mathbb{R} \to \mathbb{R}$ be a locally integrable periodic function of minimal period
$T > 0$ satisfying $(a_{*})$.
Then there exists $\lambda^{*} > 0$ such that for each $\lambda > \lambda^{*}$ there exists $\mu^{*}(\lambda) > 0$ such that,
for each $\mu > \mu^{*}(\lambda)$ and each integer $k\geq2$,
equation \eqref{eq-main} has at least $3^{km}-1$ positive $kT$-periodic solutions.

More precisely, fixed an arbitrary constant $\rho > 0$ there exists $\lambda^{*} = \lambda^{*}(\rho) > 0$ such that
for each $\lambda > \lambda^{*}$ there exist two constants $r,R$ with $0 < r < \rho < R$ and $\mu^{*}(\lambda) =
\mu^{*}(\lambda,r,R)>0$ such that, for any $\mu > \mu^{*}(\lambda)$ and for any integer $k\geq2$, the following holds:
given any finite string $\mathcal{S} = (\mathcal{S}_{1},\ldots,\mathcal{S}_{km}) \in \{0,1,2\}^{km}$, with $\mathcal{S} \neq (0,\ldots,0)$,
there exists a positive $kT$-periodic solution $u(t)$ of \eqref{eq-main} such that
\begin{itemize}
\item $\max_{t \in I^{+}_{i,\ell}} u(t) < r$, if $\mathcal{S}_{j} = 0$ for $j= i + \ell m$;
\item $r < \max_{t \in I^{+}_{i,\ell}} u(t) < \rho$, if $\mathcal{S}_{j} = 1$ for $j= i + \ell m$;
\item $\rho < \max_{t \in I^{+}_{i,\ell}} u(t) < R$, if $\mathcal{S}_{j} = 2$ for $j= i + \ell m$.
\end{itemize}
\end{theorem}

\begin{proof}
This statement follows from Theorem~\ref{main-theorem} (for the search of positive $kT$-periodic solutions and
the weight $a(t)$ considered as a $kT$-periodic function), after having checked that the constants $\lambda^{*}$, $r$, $R$ and $\mu^{*}(\lambda)$
can be chosen independently on $k$. This is a consequence of the fact that, for the part in which they depend on $a(t)$, these constants involve either integrals of $a^{\pm}(t)$ on $I^{\pm}_{i}$
or interval lengths of the form $|I^{\pm}_{i}|$, with $i=1,\ldots,m$ (compare with the discussion in Section~\ref{section-4.5}), and of the fact that
the ``new'' intervals $I^{\pm}_{i,\ell}$ (for $i=1,\ldots,m$ and $\ell = 0,\ldots, k-1$)
are just $\ell T$-translations of the original $I^{\pm}_{i}$ (with $a(t)$ $T$-periodic).
\end{proof}

\begin{remark}\label{rem-6.1}
As a further information, up to selecting the intervals $I^{\pm}_{i}$
so that $a(t)\not\equiv 0$ on each right neighborhood of $\tau_{i}$ and on each left neighborhood of
$\sigma_{i+1}$, among the properties of the positive $kT$-periodic
solutions listed in Theorem~\ref{main-theorem-sub}, we can add the following one
(if $\mu$ is sufficiently large):
\begin{itemize}
\item \textit{$0 < u(t) < r$ on $I^{-}_{i,\ell}$, for all $i=1,\ldots, m$ and $\ell= 0,\ldots,k-1$.}
\end{itemize}
This assertion is justified by Proposition~\ref{prop-5.5} taking
$\mu > \mu^{\star\star}_{r}$ defined in \eqref{eq-mu**} for $\varepsilon = r$, and
observing also that the constants $\mu_{i}^{\rm left}(r)$ and
$\mu_{i}^{\rm right}(r)$ depend on $a(t)$ on a $T$-periodicity interval and do not depend on $k$.
$\hfill\lhd$
\end{remark}

From now on, we can use Theorem~\ref{main-theorem-sub} to produce subharmonics. The trick is that of selecting strings
which are minimal in some sense, in order to obtain the minimality of the period. On the other hand, in counting the
subharmonic solutions we wish to avoid duplications, in the sense that we count only once subharmonics belonging to the
same periodicity class. To this end, we can take advantage of some combinatorial results related to the concept of
Lyndon words. We recall that a \textit{$n$-ary Lyndon word of length $k$} is a string of $k$ digits of
an alphabet $\mathscr{B}$ with $n$ symbols
which is strictly smaller in the lexicographic ordering than all of its nontrivial rotations. It is possible
to see that there is a one-to-one correspondence between the $n$-ary Lyndon words of length $k$ and
the aperiodic necklaces made by arranging $k$ beads whose color is chosen from a list of $n$ colors
(see \cite[Remark~6.3]{FeZa-pp2015}).

We denote by $\mathcal{L}_{n}(k)$ the number of $n$-ary Lyndon words of length $k$.
According to \cite[\S~5.1]{Lo-97} we have that
\begin{equation*}
\mathcal{L}_{n}(k) = \dfrac{1}{k} \sum_{l|k} \mu(l) \, n^{\frac{k}{l}},
\end{equation*}
where $\mu(\cdot)$ is the M\"{o}bius function, defined on $\mathbb{N}\setminus\{0\}$ by $\mu(1) = 1$,
$\mu(l) = (-1)^{s}$ if $l$ is the
product of $s$ distinct primes and $\mu(l) = 0$ otherwise.
For instance, the values of $\mathcal{L}_{3}(k)$ (number of ternary Lyndon words of length $k$) for
$k=2,\ldots,10$ are $3$, $8$, $18$, $48$, $116$, $312$, $810$, $2184$, $5880$.

In this setting we can now provide the following consequence of Theorem~\ref{main-theorem-sub}.

\begin{theorem}\label{main-theorem-sub2}
Let $g \colon \mathbb{R}^{+} \to \mathbb{R}^{+}$ be a continuous function satisfying $(g_{*})$, $(g_{0})$ and $(g_{\infty})$.
Let $a \colon \mathbb{R} \to \mathbb{R}$ be a locally integrable periodic function of minimal period
$T > 0$ satisfying $(a_{*})$.
Then there exists $\lambda^{*} > 0$ such that for each $\lambda > \lambda^{*}$ there exists $\mu^{*}(\lambda) > 0$ such that
for each $\mu > \mu^{*}(\lambda)$ and each integer $k\geq2$,
equation \eqref{eq-main} has at least $\mathcal{L}_{3^{m}}(k)$ positive subharmonic solutions of order $k$.
\end{theorem}

\begin{proof}
We consider an alphabet $\mathscr{B}$ made by $3^{m}$ symbols and defined as
\begin{equation*}
\mathscr{B}:=\{0,1,2\}^{m}.
\end{equation*}
Let us fix a non-null $k$-tuple $\mathcal{T}^{[k]} := (\mathcal{T}_{\ell})_{\ell=0,\ldots,k-1}$ in the alphabet $\mathscr{B}$.
We have that for each $\ell=0,\ldots,k-1$, the element
$\mathcal{T}_{\ell}\in {\mathscr{B}}$ can be written as
$\mathcal{T}_{\ell}=(\mathcal{T}_{\ell}^{i})_{i=1,\ldots,m}$, where $\mathcal{T}_{\ell}^{i}\in\{0,1,2\}$ for $i=1,\ldots,m$
and $\ell=0,\ldots, k-1$. By Theorem~\ref{main-theorem-sub},
there exists at least one positive $kT$-periodic solution $u(t)$ of equation \eqref{eq-main} such that
\begin{itemize}
\item $\max_{t \in I^{+}_{i,\ell}} u(t) < r$, if $\mathcal{T}_{\ell}^{i} = 0$;
\item $r < \max_{t \in I^{+}_{i,\ell}} u(t) < \rho$, if $\mathcal{T}_{\ell}^{i} = 1$;
\item $\rho < \max_{t \in I^{+}_{i,\ell}} u(t) < R$, if $\mathcal{T}_{\ell}^{i} = 2$.
\end{itemize}
In fact, the $k$-tuple $\mathcal{T}^{[k]}$ determines the string $\mathcal{S}$ of length $km$ with
\begin{equation*}
\mathcal{S}_{j}:= \mathcal{T}_{\ell}^{i},\quad \text{for } j = i + \ell m.
\end{equation*}
It remains to see whether, on the basis of the information we have on $u(t)$, we are able
first to prove the minimality of the period and next to distinguish among solutions not belonging to
the same periodicity class. In view of the above listed properties of the solution $u(t)$,
the minimality of the period is guaranteed when the string $\mathcal{T}^{[k]}$ has
$k$ as a minimal period (when repeated cyclically).
For the second question, given any string of this kind, we count
as the same all those strings (of length $k$) which are equivalent by cyclic permutations.
To choose exactly one string in each of these equivalence classes,
we can take the minimal one in the lexicographic order, namely a Lyndon word.
As a consequence, we find that each $3^{m}$-ary Lyndon word of length $k$ determines at least
one $kT$-periodic solution which is not $pT$-periodic for every $p = 1,\ldots, k-1$.
This solution has indeed $kT$ as minimal period. Moreover, by definition, solutions associated with
different Lyndon words are not in the same periodicity class.
\end{proof}

\subsection{Positive solutions with complex behavior}\label{section-6.2}

Having shown the existence of a mechanism producing subharmonic solutions of arbitrary order,
letting $k\to\infty$ we can provide positive (not necessarily periodic) bounded solutions
coded by a non-null bi-infinite string of three symbols.
A similar procedure has been performed in \cite{BaBoVe-15} and \cite[\S~6]{FeZa-pp2015} for the superlinear case.

Our proof is based on the following diagonal lemma borrowed from \cite[Lemma~8.1]{Kr-68} and \cite[Lemma~4]{Ma-96}.

\begin{lemma}\label{lem-6.1}
Let $f\colon\mathbb{R}\times\mathbb{R}^{d}\to\mathbb{R}^{d}$ be an $L^{1}$-Carath\'{e}odory function.
Let $(t_{n})_{n\in\mathbb{N}}$ be an increasing sequence of positive numbers and $(x_{n})_{n\in\mathbb{N}}$ be a sequence
of functions from $\mathbb{R}$ to $\mathbb{R}^{d}$ with the following properties:
\begin{itemize}
\item[$(i)$] $t_{n}\to+\infty$ as $n\to\infty$;
\item[$(ii)$] for each $n\in\mathbb{N}$, $x_{n}(t)$ is a solution of
\begin{equation}\label{eq-lem-6.1}
x'=f(t,x)
\end{equation}
defined on $\mathopen{[}-t_{n},t_{n}\mathclose{]}$;
\item[$(iii)$] there exists a closed and bounded set $B\subseteq\mathbb{R}^{d}$ such that, for each $n\in\mathbb{N}$,
$x_{n}(t)\in B$ for every $t\in\mathopen{[}-t_{n},t_{n}\mathclose{]}$.
\end{itemize}
Then there exists a subsequence $(\tilde{x}_{n})_{n\in\mathbb{N}}$ of $(x_{n})_{n\in\mathbb{N}}$ which converges uniformly
on the compact subsets of $\mathbb{R}$ to a solution $\tilde{x}(t)$ of system \eqref{eq-lem-6.1};
in particular $\tilde{x}(t)$ is defined on $\mathbb{R}$ and
$\tilde{x}(t)\in B$ for all $t\in\mathbb{R}$.
\end{lemma}

In order to simplify the exposition, as in \cite{BaBoVe-15, FeZa-pp2015} we suppose
that the coefficient $a(t)$ has a positive hump followed by a negative one in a period interval (i.e.~$m=1$ in hypothesis $(a_{*})$).
In this framework, the next result follows.

\begin{theorem}\label{th-6.3}
Let $g \colon {\mathbb{R}}^{+} \to {\mathbb{R}}^{+}$ be a continuous function satisfying $(g_{*})$,
$(g_{0})$ and $(g_{\infty})$.
Let $a \colon \mathbb{R} \to \mathbb{R}$ be a $T$-periodic locally integrable function
such that there exist $\alpha < \beta$
so that $a(t) \succ 0$ on $\mathopen{[}\alpha,\beta\mathclose{]}$
and $a(t) \prec 0$ on $\mathopen{[}\beta,\alpha + T\mathclose{]}$.
Then, fixed an arbitrary constant $\rho > 0$ there exists $\lambda^{*} = \lambda^{*}(\rho) > 0$ such that
for each $\lambda > \lambda^{*}$ there exist two constants $r,R$ with $0 < r < \rho < R$ and $\mu^{*}(\lambda) =
\mu^{*}(\lambda,r,R)>0$ such that for any $\mu > \mu^{*}(\lambda)$ the following holds:
given any two-sided sequence
$\mathcal{S} = (\mathcal{S}_{j})_{j\in \mathbb{Z}}\in \{0,1,2\}^{\mathbb{Z}}$ which is not identically zero,
there exists at least one positive solution $u(t)$ of \eqref{eq-main} such that
\begin{itemize}
\item $\max_{t \in \mathopen{[}\alpha + jT,\beta+jT\mathclose{]}} u(t) < r$, if $\mathcal{S}_{j} = 0$;
\item $r <\max_{t \in \mathopen{[}\alpha + jT,\beta+jT\mathclose{]}} u(t) < \rho$, if $\mathcal{S}_{j} = 1$;
\item $\rho < \max_{t \in \mathopen{[}\alpha + jT,\beta+jT\mathclose{]}} u(t) < R$, if $\mathcal{S}_{j} = 2$.
\end{itemize}
\end{theorem}

\begin{proof}
Without loss of generality, we suppose that $\alpha = 0$
and set $\tau:=\beta - \alpha$, so that $a(t) \succ 0$ on $\mathopen{[}0,\tau\mathclose{]}$ and
$a(t) \prec 0$ on $\mathopen{[}\tau,T\mathclose{]}$. We also introduce the intervals
\begin{equation}\label{eq-intJ}
J^{+}_{j}:= \mathopen{[}jT,\tau + jT\mathclose{]},
\quad J^{-}_{j}:= \mathopen{[}\tau + jT,(j+1)T\mathclose{]}, \quad j\in \mathbb{Z}.
\end{equation}
Let $\rho$, $\lambda > \lambda^{*}$, $r$, $R$ and $\mu^{*}(\lambda)$ be fixed as in Section~\ref{section-4.2}
and Section~\ref{section-4.5} for $m=1$. Once more, we emphasize that all our constants
can be chosen independently on $k$. Thus, having fixed all these constants and taken $\mu > \mu^{*}(\lambda)$,
we can produce $kT$-periodic solutions following any $k$-periodic two-sided sequence of three symbols,
as in Theorem~\ref{main-theorem-sub}.

Consider now an arbitrary sequence $\mathcal{S} = (\mathcal{S}_{j})_{j\in\mathbb{Z}}\in \{0,1,2\}^{\mathbb{Z}}$
which is not identically zero.
We fix a positive integer $n_{0}$ such that there is at least an index
$j\in\{-n_{0},\ldots,n_{0}\}$ such that $\mathcal{S}_{j}\neq 0$.
Then, for each $n \geq n_{0}$ we consider the $(2n+1)$-periodic sequence ${\mathcal{S}}^{n} =(\mathcal{S}'_{j})_{j}\in \{0,1,2\}^{\mathbb{Z}}$
which is obtained by truncating $\mathcal{S}$ between $-n$ and $n$,
and then repeating that string by periodicity.
We apply Theorem~\ref{main-theorem-sub}, with $m=1$,
on the periodicity interval $\mathopen{[}-nT,(n+1)T\mathclose{]}$ and find
a \textit{positive} periodic solution $u_{n}(t)$ such that $u_{n}(t + (2n+1)T) = u_{n}(t)$
for all $t\in\mathbb{R}$ and $\|u_{n}\|_{\infty} < R$
(by the concavity of the solutions in the intervals $J^{-}_{j}$ where
$a(t) \prec 0$). Moreover, we also know that
\begin{itemize}
\item $\max_{t\in J^{+}_{j}} u_{n}(t) < r$, if $\mathcal{S}'_{j} = 0$;
\item $r <\max_{t\in J^{+}_{j}} u_{n}(t) < \rho$, if $\mathcal{S}'_{j} = 1$;
\item $\rho < \max_{t\in J^{+}_{j}} u_{n}(t) < R$, if $\mathcal{S}'_{j} = 2$.
\end{itemize}

In each interval $J^{+}_{j}$ (of length $\tau$) the positive solution $u_{n}(t)$ is bounded by $R$ and therefore there exists at least a point
$t_{n,j}\in J^{+}_{j}$ such that $|u'_{n}(t_{n,j})|\leq R/\tau$. Hence, for each $t\in J^{+}_{j}$ and every $n \geq n_{0}$, it holds that
\begin{equation}\label{eq-6.5}
\begin{aligned}
|u'_{n}(t)| &= \biggl{|}u'_{n}(t_{n,j}) + \int_{t_{n,j}}^{t} u''_{n}(\xi)~\!d\xi \biggr{|}
          \leq \dfrac{R}{\tau} + \lambda \int_{J^{+}_{j}} a^{+}(\xi)g(u_{n}(\xi))~\!d\xi
\\ &\leq \dfrac{R}{\tau} + \lambda \|a\|_{+,1} g^{*}(R) =:K,
\end{aligned}
\end{equation}
where the constants $\|a\|_{+,1}$ and $g^{*}(R)$ are those defined at the beginning of Section~\ref{section-4}.
Notice that $K$ is independent on $j$ and this provides a uniform estimate for all the intervals where the weight is positive.
On the other hand, using the convexity of $u_{n}(t)$ in the intervals $J^{-}_{j}$, we know that
\begin{equation*}
|u'_{n}(t)|\leq \max_{\xi\in \partial J^{-}_{j}} |u'_{n}(\xi)| \leq \max_{\xi\in J^{+}_{j}\cup J^{+}_{j+1}} |u'_{n}(\xi)| \leq K,
\quad \forall \, t\in J^{-}_{j}, \; \forall \, n\geq n_{0},
\end{equation*}
and thus we are able to find the global uniform estimate
\begin{equation*}
|u'_{n}(t)|\leq K, \quad \forall \, t \in \mathbb{R}, \; \forall \, n\geq n_{0}.
\end{equation*}

Now we write equation \eqref{eq-main} as the planar system
\begin{equation*}
\begin{cases}
\, u' = y \\
\, y' = - \bigl{(}\lambda a^{+}(t)-\mu a^{-}(t)\bigr{)}g(u).
\end{cases}
\end{equation*}
From the above estimates, one can see that (up to a reparametrization of indices, counting from $n_{0}$)
assumptions $(i)$, $(ii)$ and $(iii)$ of Lemma~\ref{lem-6.1} are satisfied,
taking $t_{n}:=nT$, $f(t,x)=(y,-(\lambda a^{+}(t)-\mu a^{-}(t))g(u))$, with $x=(u,y)$, and
\begin{equation*}
B := \bigl{\{} x=(x_{1},x_{2})\in\mathbb{R}^{2} \colon 0 \leq x_{1} \leq R, \; |x_{2}| \leq K \bigr{\}},
\end{equation*}
as closed and bounded set in $\mathbb{R}^{2}$.
By Lemma~\ref{lem-6.1}, there is a solution $\tilde{u}(t)$ of equation \eqref{eq-2.3} which is defined on $\mathbb{R}$ and such that
$0 \leq \tilde{u}(t) \leq R$ for all $t\in \mathbb{R}$.
Moreover, such a solution $\tilde{u}(t)$ is the limit of a subsequence $(\tilde{u}_{n})_{n}$
of the sequence of the periodic solutions $u_{n}(t)$.

We claim that
\begin{itemize}
\item $\max_{t\in J^{+}_{j}} \tilde{u}(t) < r$, if $\mathcal{S}_{j} = 0$;
\item $r <\max_{t\in J^{+}_{j}} \tilde{u}(t) < \rho$, if $\mathcal{S}_{j} = 1$;
\item $\rho < \max_{t\in J^{+}_{j}} \tilde{u}(t) < R$, if $\mathcal{S}_{j} = 2$.
\end{itemize}
To prove our claim, let us fix $j\in \mathbb{Z}$ and consider the interval $J^{+}_{j}$ introduced in \eqref{eq-intJ}.
For each $n \geq |j|$ (and $n\geq n_{0}$) the periodic solution $u_{n}(t)$ is defined on $\mathbb{R}$ and
such that $\max_{J^{+}_{j}} u_{n} < r$ if $\mathcal{S}_{j} = 0$, $r <\max_{J^{+}_{j}} u_{n} < \rho$
if $\mathcal{S}_{j} = 1$, $\rho < \max_{J^{+}_{j}} u_{n} < R$ if $\mathcal{S}_{j} = 2$.
Passing to the limit on the subsequence $(\tilde{u}_{n})_{n}$, we obtain that
\begin{itemize}
\item $\max_{t\in J^{+}_{j}} \tilde{u}(t) \leq r$, if $\mathcal{S}_{j} = 0$;
\item $r \leq \max_{t\in J^{+}_{j}} \tilde{u}(t) \leq \rho$, if $\mathcal{S}_{j} = 1$;
\item $\rho \leq \max_{t\in J^{+}_{j}} \tilde{u}(t) \leq R$, if $\mathcal{S}_{j} = 2$.
\end{itemize}
By Proposition~\ref{prop-5.1} we get that $\tilde{u}(t) < R$, for all $t\in \mathbb{R}$. Moreover,
since there exists at least one index $j\in \mathbb{Z}$ such that $\mathcal{S}_{j} \neq 0$, we know that $\tilde{u}(t)$ is not identically zero.
Hence, a maximum principle argument shows that $\tilde{u}(t)$ never vanishes. In conclusion, we have proved that
\begin{equation*}
0 < \tilde{u}(t) < R, \quad \forall \, t\in \mathbb{R}.
\end{equation*}
Next, using this fact, by Proposition~\ref{prop-5.2} we observe that
\begin{equation*}
\max_{t\in J^{+}_{j}} \tilde{u}(t) \neq \rho, \quad \forall \, j\in \mathbb{Z},
\end{equation*}
and by Proposition~\ref{prop-5.3} we have
\begin{equation*}
\max_{t\in J^{+}_{j}} \tilde{u}(t) \neq r, \quad \forall \, j\in \mathbb{Z},
\end{equation*}
since, at the beginning, $\mu$ has been chosen large enough (note also that we apply those propositions in the case
$m =1$ and so the sets $I^{+}_{i,\ell}$ reduce to the intervals $\mathopen{[}0,\tau\mathclose{]} + \ell T$).
Our claim is thus verified and this completes the proof of the theorem.
\end{proof}

Theorem~\ref{th-6.3} can be compared with the main result in \cite{BoZa-12b},
providing (under a few technical conditions on $a(t)$ and $g(s)$) globally defined positive solutions to \eqref{eq-main}
according to a symbolic dynamics on \textit{two} symbols. More precisely, using a dynamical systems technique it was shown in \cite[Theorem~2.3]{BoZa-12b}
the existence of two disjoint compact sets $\mathcal{K}_{1},\mathcal{K}_{2} \subseteq \mathbb{R}^{2}$ such that for any two-sided sequence
$\mathcal{S} = (\mathcal{S}_{j})_{j\in \mathbb{Z}}\in \{1,2\}^{\mathbb{Z}}$ there is a positive solution $u(t)$ to \eqref{eq-main} satisfying
$(u(\alpha + jT),u'(\alpha+jT)) \in \mathcal{K}_{\mathcal{S}_{j}}$ for all $j \in \mathbb{Z}$.
Even if this conclusion is not directly comparable with the one of Theorem~\ref{th-6.3} (in which solutions are distinguished
in dependence of the value $\max_{t \in \mathopen{[}\alpha + jT,\beta+jT\mathclose{]}} u(t)$), a careful reading of the arguments in \cite{BoZa-12b}
should convince us that the solutions obtained therein correspond to solutions which are ``small'' or ``large'' according to the code of the present paper.
From this point of view, Theorem~\ref{th-6.3} can thus be seen as an improvement of \cite[Theorem~2.3]{BoZa-12b}, providing in addition solutions
which are ``very small'' on some intervals of positivity of the weight function and thus leading to a symbolic dynamics on \textit{three} symbols.
It has to be noticed, however, that in \cite{BoZa-12b} some further information for the Poincar\'{e} map associated with \eqref{eq-main} were obtained;
we will comment again on this point in Section~\ref{section-6.3}.

\medskip

Theorem~\ref{th-6.3} can be extended to the case of a weight function with more than one positive hump in the interval $\mathopen{[}0,T\mathclose{]}$,
as described in hypothesis $(a_{*})$. The corresponding more general result is given in the next theorem.

\begin{theorem}\label{th-6.4}
Let $g \colon {\mathbb{R}}^{+} \to {\mathbb{R}}^{+}$ be a continuous function satisfying $(g_{*})$,
$(g_{0})$ and $(g_{\infty})$.
Let $a \colon \mathbb{R} \to \mathbb{R}$ be a locally integrable periodic function of minimal period
$T > 0$ satisfying $(a_{*})$.
Then, fixed an arbitrary constant $\rho > 0$ there exists $\lambda^{*} = \lambda^{*}(\rho) > 0$ such that
for each $\lambda > \lambda^{*}$ there exist two constants $r,R$ with $0 < r < \rho < R$ and $\mu^{*}(\lambda) =
\mu^{*}(\lambda,r,R)>0$ such that for any $\mu > \mu^{*}(\lambda)$ the following holds:
given any two-sided sequence
$\mathcal{S} = (\mathcal{S}_{j})_{j\in\mathbb{Z}}$ in the alphabet $\mathscr{A}:=\{0,1,2\}$ which is not identically zero,
there exists at least one positive solution $u(t)$ of \eqref{eq-main} such that
\begin{itemize}
\item $\max_{t \in I^{+}_{i,\ell}} u(t) < r$, if $\mathcal{S}_{j} = 0$ for $j= i + \ell m$;
\item $r < \max_{t \in I^{+}_{i,\ell}} u(t) < \rho$, if $\mathcal{S}_{j} = 1$ for $j= i + \ell m$;
\item $\rho < \max_{t \in I^{+}_{i,\ell}} u(t) < R$, if $\mathcal{S}_{j} = 2$ for $j= i + \ell m$.
\end{itemize}
\end{theorem}

\begin{proof}
The proof requires only minor modifications in the argument applied for Theorem~\ref{th-6.3} and thus
the details are omitted. We only observe that the uniform bound $K$ for $|u'_{n}(t)|$ is now achieved by working
separately on each interval $I^{+}_{i,\ell}$. When arguing like in \eqref{eq-6.5} one obtains
\begin{equation*}
|u'_{n}(t)| \leq \dfrac{R}{|I^{+}_{i}|} + \lambda \|a\|_{+,i} g^{*}(R) =:K_{i},\quad \forall \, t\in I^{+}_{i,\ell}, \; \forall \, n\geq n_{0}.
\end{equation*}
Now all the rest works fine for
\begin{equation*}
K:= \max_{i=1,\ldots,m} K_{i}.
\end{equation*}
The same final arguments allow us to obtain the theorem.
\end{proof}

\begin{remark}\label{rem-6.2}
As a further information, up to selecting the intervals $I^{\pm}_{i}$
so that $a(t)\not\equiv 0$ on each right neighborhood of $\tau_{i}$ and on each left neighborhood of
$\sigma_{i+1}$, among the properties of the positive
solutions listed in Theorem~\ref{th-6.3} and Theorem~\ref{th-6.4}, we can add the following one
(if $\mu$ is sufficiently large):
\begin{itemize}
\item \textit{$0<u(t)<r$ on $I^{-}_{i,\ell}$, for all $i\in\{1,\ldots,m\}$ and for all $\ell \in \mathbb{Z}$.}
\end{itemize}
This assertion is justified by Proposition~\ref{prop-5.5} taking
$\mu > \mu^{\star\star}_{r}$ defined in \eqref{eq-mu**} for $\varepsilon = r$, and
observing also that the constants $\mu_{i}^{\rm left}(r)$ and
$\mu_{i}^{\rm right}(r)$ depend on $a(t)$ on a $T$-periodicity interval.
$\hfill\lhd$
\end{remark}

\subsection{A dynamical systems perspective}\label{section-6.3}

In the last two previous sections we have proved the presence of chaotic-like dynamics
which is highlighted by the coexistence of infinitely many subharmonic solutions together with
non-periodic bounded solutions which can be coded by sequences of three symbols. Our next goal is to
show that our results allow us to enter a classical framework for complex
dynamical systems, namely the semiconjugation with the Bernoulli shift.

We start with some formal definitions.
Let $\mathscr{B}$ be a finite set of $n\geq 2$ elements (called \textit{symbols}), conventionally denoted as
$\mathscr{B} := \{b_{1},\ldots,b_{n}\}$, which is endowed with the discrete topology.
Let $\Sigma_{n}:= \mathscr{B}^{\mathbb{Z}}$ be the set of all two-sided
sequences $\mathcal{T} = (\mathcal{T}_{\ell})_{\ell\in \mathbb{Z}}$ where, for each $\ell\in \mathbb{Z}$, the element $\mathcal{T}_{\ell}$ is a symbol
of the alphabet $\mathscr{B}$. The set $\Sigma_{n}= \prod_{\ell\in \mathbb{Z}} \mathscr{B}$, endowed with the product topology,
turns out to be a compact metrizable space. As a suitable distance on $\Sigma_{n}$ we take
\begin{equation*}
d(\mathcal{T}',\mathcal{T}''):= \sum_{\ell\in \mathbb{Z}} \dfrac{\delta(\mathcal{T}'_{\ell},\mathcal{T}''_{\ell})}{2^{|\ell|}}, \quad \mathcal{T}',\mathcal{T}''\in\Sigma_{n},
\end{equation*}
where $\delta$ is the discrete distance on $\mathscr{B}$,
that is $\delta(s',s'') = 0$ if $s'=s''$ and $\delta(s',s'') = 1$ if $s'\neq s''$. We introduce a map $\sigma \colon \Sigma_{n} \to \Sigma_{n}$
called the \textit{shift automorphism} (cf.~\cite[p.~770]{Sm-67})
or \textit{Bernoulli shift} (cf.~\cite{Wa-82}) and defined as
\begin{equation*}
\sigma(\mathcal{T}) = \mathcal{T}', \quad \text{with } \, \mathcal{T}'_{\ell} := \mathcal{T}_{\ell+1}, \; \forall \, \ell\in \mathbb{Z}.
\end{equation*}
The map $\sigma$ is a bijective continuous map (a homeomorphism) of $\Sigma_{n}$ which possesses all
the features usually associated with the concept of chaos, such as transitivity, density of the set of periodic points
and positive topological entropy (which is $\log (n)$ for an alphabet of $n$ symbols).

Given a topological space $X$ and a continuous map $\psi \colon X \to X$, a typical way to prove that $\psi$ is ``chaotic''
consists into verifying that $\psi$ has the shift map as a \textit{factor}, namely that there exist a compact set
$Y \subseteq X$ which is invariant for $\psi$ (i.e.~$\psi(Y) = Y$) and a \textit{continuous and
surjective} map $\pi \colon Y \to \Sigma_{n}$ such that the diagram
\begin{equation*}
\xymatrix{
Y \ar[d]_{\mathlarger{\pi}} \ar[r]^{\mathlarger{\psi}} & Y \ar[d]^{\mathlarger{\pi}}\\
\Sigma_{n} \ar[r]_{\mathlarger{\sigma}} & \Sigma_{n}}
\end{equation*}
commutes, that is
\begin{equation}\label{eq-comm}
\pi \circ \psi = \sigma \circ \pi.
\end{equation}
If we are in this situation we say that the map $\psi|_{Y}$ is \textit{semiconjugate} with the shift on $n$ symbols.
Usually the best form of chaos occurs when the map $\pi \colon Y \to \Sigma_{n}$ is a homeomorphism. In this latter case the map
$\psi|_{Y}$ is said to be \textit{conjugate} with the shift $\sigma$. This, for instance, occurs for the classical
Smale horseshoe (see \cite{Mo-73, Sm-67}). In many concrete examples of differential equations, the conjugation with the
shift map is not feasible and many investigations have been addressed toward the proof of a
semiconjugation with the Bernoulli shift, possibly accompanied by some
further information, such as density of periodic points, in order to provide a
description of chaotic dynamics which is still interesting for the applications.
Quoting Block and Coppel from \cite[Introduction]{BlCo-92},
\begin{quote}
`` \dots there is no generally accepted definition of chaos.
It is our view that any definition for more general spaces should agree with ours in the case of an interval.
\dots we show that a map is chaotic if and only if some iterate has the shift map as a \textit{factor}, and
we propose this as a general definition.''
\end{quote}
Indeed, the semiconjugation of an iterate of a map $\psi$ with the
Bernoulli shift is defined as \textit{B/C-chaos} in \cite{AuKi-01}.

We plan to prove the existence of a strong form of B/C-chaos coming from Theorem~\ref{main-theorem-sub} and
Theorem~\ref{th-6.4}, namely the existence of a compact invariant set $Y$ for a continuous homeomorphism
$\psi$ such that $\psi|_{Y}$ satisfies \eqref{eq-comm} and such that to any periodic sequence
of symbols corresponds a periodic solution of \eqref{eq-main}. Such a stronger form of chaos has been
produced by several authors using dynamical systems techniques
(see, for instance, \cite{BoZa-12b, CaKwMi-00, MiMr-95a, MiMr-95b, SrWo-97, Zg-96, ZgGi-04}). The obtention of
this kind of results with the coincidence degree approach appears new in the literature.

\medskip

Let us start by defining a suitable metric space and a homeomorphism on it.
Let $X$ be the set of the continuous functions $z = (x,y) \colon \mathbb{R} \to \mathbb{R}^{2}$.
For each $z_{1} = (x_{1},y_{1}), z_{2} = (x_{2},y_{2})\in X$, we define
\begin{equation*}
\vartheta_{N}(z_{1},z_{2}) := \max\bigl{\{}|x_{1}(t)-x_{2}(t)| + |y_{1}(t)-y_{2}(t)|\colon t\in\mathopen{[}-N,N\mathclose{]}\bigr{\}},
\quad N\in\mathbb{N}\setminus\{0\},
\end{equation*}
and we set
\begin{equation*}
\text{\rm dist}\,(z_{1},z_{2}) := \sum_{N=1}^{\infty} \dfrac{1}{2^{N}}\dfrac{\vartheta_{N}(z_{1},z_{2})}{1+\vartheta_{N}(z_{1},z_{2})}.
\end{equation*}
It is a standard task to check that $(X,\text{\rm dist})$ is a complete metric space.
Moreover, given a sequence of functions $(z_{k})_{k}$ in $X$ and a function $\hat{z}\in X$, we have that
$z_{k} \to \hat{z}$ with respect to the distance of $X$ if and only if $z_{k}(t)$ converges uniformly to $\hat{z}(t)$
in each compact interval of $\mathbb{R}$ (cf.~\cite[ch.~1]{BhSz-70}, \cite[ch.~III]{Se-71} and \cite[\S~20]{Si-75}).
We also recall that a family of functions $\mathcal{M} \subseteq X$ is relatively compact
if and only if for every compact interval $J$ the set of restrictions to $J$ of the functions belonging to $\mathcal{M}$
is relatively compact in $\mathcal{C}(J,\mathbb{R}^2)$ (cf.~\cite[p.~2]{Co-73}).
Next, recalling that $T>0$ is the minimal period of the weight function $a(t)$,
we introduce the shift map $\psi \colon X \to X$ defined by
\begin{equation*}
(\psi u) (t) := u(t+T), \quad t\in\mathbb{R},
\end{equation*}
which is a homeomorphism of $X$ onto itself. The discrete dynamical system induced by $\psi$ is usually
referred to as a \textit{Bebutov dynamical system} on $X$.

\medskip

For the next results we assume the standard hypotheses on the nonlinearity $g(s)$ and on the
coefficient $a(t)$, that is, $g \colon {\mathbb{R}}^{+} \to {\mathbb{R}}^{+}$ is a continuous
function satisfying $(g_{*})$, $(g_{0})$, $(g_{\infty})$, $a \colon \mathbb{R} \to \mathbb{R}$ is
a $T$-periodic locally integrable function satisfying $(a_{*})$ with minimal period $T$.
We suppose also that all the positive constants $\rho$, $\lambda > \lambda^{*}$, $r$, $R$ and $\mu^{*}(\lambda)$
are fixed as in Section~\ref{section-4.2} and Section~\ref{section-4.5}. Let also $\mu > 0$.

We consider the first order differential system
\begin{equation}\label{eq-xy}
\begin{cases}
\, x' = y \\
\, y' = - \bigl{(}\lambda a^{+}(t)-\mu a^{-}(t)\bigr{)}g(x)
\end{cases}
\end{equation}
associated with \eqref{eq-main}. Even if all our results concern non-negative solutions of \eqref{eq-main}, in dealing with system \eqref{eq-xy}
it would be convenient to have the vector field (i.e.~the right hand side of the system) defined for
all $t\in \mathbb{R}$ and $(x,y)\in \mathbb{R}^{2}$.
For this reason, we extend $g(s)$ to the whole real line, for instance by setting $g(s) = 0$
for $s \leq 0$ (any extension we choose will have no effect in what follows).
As usual the solutions of \eqref{eq-xy} are meant in the Carath\'{e}odory sense.

Next, we denote by $Y_{0}$ the subset of $X$ made up of the globally defined solutions $(x(t),y(t))$ of \eqref{eq-xy}
such that $0 \leq x(t) \leq R$, for all $t\in \mathbb{R}$.
Observe that $(0,0) \in Y_{0}$ (as $u(t)\equiv 0$ is the trivial solution of \eqref{eq-main}).
On the other hand, if $(x,y) \in Y_{0}$ with $x\not\equiv 0$, then $x(t) > 0$ for all $t\in \mathbb{R}$.

\begin{lemma}\label{lem-6.2}
There exists a constant $K > 0$ such that for each $(x,y) \in Y_{0}$ it holds that
\begin{equation}\label{eq-6.8}
|y(t)| \leq K,\quad \forall \, t\in \mathbb{R}.
\end{equation}
Moreover, $Y_{0}$ is a compact subset of $X$ which is invariant for the map $\psi$.
\end{lemma}

\begin{proof}
The estimates needed to prove this result have been already obtained along the proof of Theorem~\ref{th-6.3}.
We briefly repeat the argument since the context here is slightly different.
Let $(x,y) \in Y_{0}$. Since $0 \leq x(t) \leq R$ for all $t\in \mathbb{R}$, we have that,
for all $i\in\{1,\ldots,m\}$ and $\ell\in\mathbb{Z}$,
there exists at least a point $\hat{t}_{i,\ell}\in I^{+}_{i,\ell}$ such that $|y(\hat{t}_{i,\ell})|\leq R/|I^{+}_{i}|$
(recall the definition of $I^{+}_{i,\ell}$ in \eqref{eq-Ipm}). Hence, for each $t\in I^{+}_{i,\ell}$, it holds that
\begin{equation*}
\begin{aligned}
|y(t)| &= \biggl{|}y(\hat{t}_{i,\ell}) + \int_{\hat{t}_{i,\ell}}^{t} y'(\xi)~\!d\xi \biggr{|}
          \leq \dfrac{R}{|I^{+}_{i}|} + \lambda \int_{I^{+}_{i,\ell}} a^{+}(\xi)g(x(\xi))~\!d\xi
\\ &\leq \dfrac{R}{|I^{+}_{i}|} + \lambda \|a\|_{+,i}g^{*}(R) =: K_{i}.
\end{aligned}
\end{equation*}
Note that the constant $K_{i}$ does not depend on the index $\ell$. Therefore, setting
\begin{equation*}
K:= \max_{i=1,\ldots,m} K_{i},
\end{equation*}
we get
\begin{equation*}
|y(t)| \leq K, \quad \forall \, t \in I^{+}_{i,\ell}, \; \forall \, i=1,\ldots,m, \; \forall \, \ell\in\mathbb{Z}.
\end{equation*}
On the other hand, using the convexity of $x(t)$ in the intervals $I^{-}_{i,\ell}$ we know that
\begin{equation*}
|y(t)| = |x'(t) | \leq \max_{\xi \in \partial I^{-}_{i,\ell}} |x'(\xi)| \leq K,
\quad \forall \, t \in I^{+}_{i,\ell}, \; \forall \, i=1,\ldots,m, \; \forall \, \ell\in\mathbb{Z}.
\end{equation*}
This proves inequality \eqref{eq-6.8}.

From system \eqref{eq-xy}, we know that the absolutely continuous vector function $(x,y)\in Y_{0}$ satisfies
\begin{equation*}
|x'(t)| + |y'(t)| \leq K + \bigl{(} \lambda a^{+}(t) + \mu a^{-}(t) \bigr{)} g^{*}(R), \quad \text{for a.e. } t \in \mathbb{R}.
\end{equation*}
Therefore, Ascoli-Arzel\`{a} theorem implies that the set of restrictions of the functions in $Y_{0}$
to any compact interval is relatively compact in the uniform norm. Thus we conclude that the closed set $Y_{0}$ is a compact subset of $X$.

Finally, we observe that the invariance of $Y_{0}$ under the map $\psi$ follows from the $T$-periodicity of the coefficients
in system \eqref{eq-xy}, which in turn implies that $(x(t),y(t))$ is a solution of \eqref{eq-xy} if and only if
$(x(t+T),y(t+T))$ is a solution of the same system.
\end{proof}

The next result summarizes the properties obtained in Proposition~\ref{prop-5.1}, Proposition~\ref{prop-5.2} and Proposition~\ref{prop-5.3}.

\begin{lemma}\label{lem-6.3}
Suppose that $\mu > \mu^{*}(\lambda)$. Then, given any $(x,y)\in Y_{0}$, for each $i\in \{1,\ldots,m\}$
and $\ell\in \mathbb{Z}$ we have that one of the following alternatives holds:
$\max_{t \in I^{+}_{i,\ell}} x(t) < r$, $r < \max_{t \in I^{+}_{i,\ell}} x(t) < \rho$
or $\rho < \max_{t \in I^{+}_{i,\ell}} x(t) < R$.
\end{lemma}

Let
\begin{equation*}
\mathscr{B}:=\{0,1,2\}^{m}
\end{equation*}
be the alphabet of the $3^{m}$ elements
of the form $(\omega_{1},\ldots,\omega_{m})$, where $\omega_{i}\in \{0,1,2\}$ for each $i=1,\ldots,m$.

We define a semiconjugation $\pi$ between $Y_{0}$ and the set $\Sigma_{3^{m}}$ associated with $\mathscr{B}$ as follows.
Suppose that $\mu > \mu^{*}(\lambda)$. To each element $z=(x,y)\in Y_{0}$ the map $\pi$ associates a sequence
$\pi(z) = \mathcal{T} = (\mathcal{T}_{\ell})_{\ell\in\mathbb{Z}} \in \Sigma_{3^{m}}$ defined as
\begin{equation*}
\mathcal{T}_{\ell} = (\mathcal{T}_{\ell}^{1},\ldots,\mathcal{T}_{\ell}^{m})\in \mathscr{B}, \quad \ell\in\mathbb{Z},
\end{equation*}
where, for $i=1,\ldots,m$,
\begin{itemize}
\item  $\mathcal{T}_{\ell}^{i} = 0$, if  $\max_{t \in I^{+}_{i,\ell}} x(t) < r$;
\item  $\mathcal{T}_{\ell}^{i} = 1$, if $r < \max_{t \in I^{+}_{i,\ell}} x(t) < \rho$;
\item  $\mathcal{T}_{\ell}^{i} = 2$, if $\rho < \max_{t \in I^{+}_{i,\ell}} x(t) < R$.
\end{itemize}
Lemma~\ref{lem-6.3} guarantees that the above map is well-defined.

Now we are in position to state the main result of this section.

\begin{theorem}\label{th-6.5}
Suppose that $\mu > \mu^{*}(\lambda)$. Then the map $\pi \colon Y_{0} \to \Sigma_{3^{m}}$
is continuous, surjective and such that the diagram
\begin{equation*}
\xymatrix{
Y_{0} \ar[d]_{\mathlarger{\pi}} \ar[r]^{\mathlarger{\psi}} & Y_{0} \ar[d]^{\mathlarger{\pi}}\\
\Sigma_{3^{m}} \ar[r]_{\mathlarger{\sigma}} & \Sigma_{3^{m}}}
\end{equation*}
commutes. Furthermore, for every integer $k\geq1$, the counterimage of any $k$-periodic sequence in
$\Sigma_{3^{m}}$ contains at least a point $(u,y)\in Y_{0}$
such that $u(t)$ is a $kT$-periodic solution of \eqref{eq-main}.
\end{theorem}

\begin{proof}
Part of the statement follows immediately from our previous results. The surjectivity of the map
$\pi$ is a consequence of Theorem~\ref{th-6.4}. Indeed, if $\mathcal{T} \in \Sigma_{3^{m}}$ is the null sequence
then it is the image of the trivial solution $(0,0) \in Y_{0}$. On the other hand, given any non-null sequence
$\mathcal{T} = (\mathcal{T}_{\ell})_{\ell \in \mathbb{Z}}$, with $\mathcal{T}_{\ell} = (\mathcal{T}_{\ell}^{1},\ldots,\mathcal{T}_{\ell}^{m})$ for each $\ell\in\mathbb{Z}$,
there exists at least one globally defined positive solution $u(t)$ to equation \eqref{eq-main} such that
\begin{itemize}
\item $\max_{t \in I^{+}_{i,\ell}} u(t) < r$, if $\mathcal{T}_{\ell}^{i} = 0$;
\item $r < \max_{t \in I^{+}_{i,\ell}} u(t) < \rho$, if $\mathcal{T}_{\ell}^{i} = 1$;
\item $\rho < \max_{t \in I^{+}_{i,\ell}} u(t) < R$, if $\mathcal{T}_{\ell}^{i} = 2$.
\end{itemize}
Then $\pi$ maps $(u(t),u'(t)) = (x(t),y(t)) \in Y_{0}$ to $\mathcal{T}$. In a similar way, Theorem~\ref{main-theorem-sub}
ensures that, for any integer $k \geq 1$, the counterimage of a $k$-periodic sequence in
$\Sigma_{3^{m}}$ can be chosen as a $kT$-periodic solution of \eqref{eq-xy}.

The commutativity of the diagram follows from the fact that, if $(x(t),y(t))$ is a solution of \eqref{eq-xy}, then
$(x(t+T),y(t+T))$ is also a solution of the same system and, moreover,
if $(\mathcal{T}_{\ell})_{\ell\in\mathbb{Z}}$ is the sequence of symbols associated with $(x(t),y(t))$,
then the sequence corresponding to $(x(t+T),y(t+T))$ must be $(\mathcal{T}_{\ell + 1})_{\ell\in\mathbb{Z}}$.
This proves \eqref{eq-comm}.

Thus we have only to check the continuity of $\pi$. Let $\tilde{z}=(\tilde{x},\tilde{y}) \in Y_{0}$
and $\tilde{\mathcal{T}} = \pi(\tilde{z})$. Let $z_{n}=(x_{n},y_{n}) \in Y_{0}$ be a sequence such that
$z_{n} \to \tilde{z}$ in $Y_{0}$. This means that
$(x_{n}(t),y_{n}(t))$ converges uniformly to $(\tilde{x}(t),\tilde{y}(t))$ on any compact interval $\mathopen{[}-NT,NT\mathclose{]}$ of the real line.
For any interval $I^{+}_{i,\ell} \subseteq \mathopen{[}-NT,NT\mathclose{]}$,
we have that either $\max_{I^{+}_{i,\ell}}\tilde{x} < r$ or $r < \max_{I^{+}_{i,\ell}}\tilde{x} < \rho$ or
$\rho < \max_{I^{+}_{i,\ell}}\tilde{x} < R$. By the uniform convergence of the sequence of solutions
on $I^{+}_{i,\ell}$, there exists an index $n^{*}_{i,\ell}$ such that, for each $n \geq n^{*}_{i,\ell}$, the solution
$x_{n}(t)$ satisfies the same inequalities as $\tilde{x}(t)$ on the interval $I^{+}_{i,\ell}$. Hence, for any fixed $N$, there is an index
\begin{equation*}
n^{*}_{N} := \max \bigl{\{} n^{*}_{i,\ell} \colon i =1,\ldots,m, \, \ell=-N,\ldots,N-1 \bigr{\}}
\end{equation*}
such that, setting $\mathcal{T}^{n} = \pi(z_{n})$, it holds $\mathcal{T}^{n}_{\ell} = \tilde{\mathcal{T}}_{\ell}$ for all $n \geq n^{*}_{N}$ and
$\ell = -N,\ldots, N-1$. By the topology of $\Sigma_{3^{m}}$,
this means that $\mathcal{T}^{n}$ converges to $\tilde{\mathcal{T}}$. This concludes the proof.
\end{proof}

From Theorem~\ref{th-6.5} many consequences can be produced. For instance, we can
refine the set $Y_{0}$ in order to obtain an invariant set with dense periodic trajectories
of any period. This follows via a standard procedure that we describe below for the reader's convenience.

Let $Y_{\rm per}$ be the set of all the pairs $(x,y)\in Y_{0}$ which are $kT$-periodic solutions of \eqref{eq-xy}
for some integer $k\geq 1$ and let
\begin{equation*}
Y:= \text{cl}(Y_{\rm per}) \subseteq Y_{0},
\end{equation*}
where the closure is taken with respect to the distance in the space $X$.
Clearly, the set $Y$ is compact, invariant for the map $\psi$ and
$Y_{\rm per}$ is dense in $Y$. Then, from Theorem~\ref{th-6.5} we immediately have that
for  $\mu > \mu^{*}(\lambda)$ the map $\psi|_{Y} \colon Y \to Y$
is semiconjugate (via the surjection $\pi|_{Y}$) with the shift $\sigma$ on $\Sigma_{3^{m}}$
and, moreover, for every integer $k\geq1$, the counterimage by $\pi$ of any $k$-periodic sequence in
$\Sigma_{3^{m}}$ contains at least a point $(u,y)\in Y$
such that $u(t)$ is a $kT$-periodic solution of \eqref{eq-main}.

\medskip

As a last step, we want to express our results in terms of the Poincar\'{e} map associated with system \eqref{eq-xy}. To this end,
we further suppose that the nonlinearity $g(s)$ is locally Lipschitz continuous on $\mathbb{R}^{+}$. This, in turn,
implies the uniqueness of the solutions for the initial value problems associated with \eqref{eq-xy}.
We recall that the Poincar\'{e} map associated with system \eqref{eq-xy} is defined as
\begin{equation*}
\Psi_{T} \colon \text{\rm dom}\,\Psi_{T} (\subseteq \mathbb{R}^{2}) \to \mathbb{R}^{2},\quad
z_{0} = (x_{0},y_{0})\mapsto z(T,z_{0}),
\end{equation*}
where $z(t,z_{0}) = (x(t,z_{0}),y(t,z_{0}))$ is the solution of system \eqref{eq-xy} such that
$x(0) = x_{0}$ and $y(0) = y_{0}$. The map $\Psi_{T}$ is defined provided that the solutions can be
extended to the interval $\mathopen{[}0,T\mathclose{]}$. In general the domain of $\Psi_{T}$ is an open subset of $\mathbb{R}^{2}$
and $\Psi_{T}$ is a homeomorphism of $\text{\rm dom}\,\Psi_{T}$ onto its image.
In our case, due to the sublinear growth at infinity $(g_{\infty})$, we have that $\text{\rm dom}\,\Psi_{T}=\mathbb{R}^{2}$ and
$\Psi_T$ is a homeomorphism of $\mathbb{R}^2$ onto itself.

Let
\begin{equation*}
\mathcal{W}_{0}:= \bigl{\{} (x(0),y(0))\in \mathopen{[}0,R\mathclose{]}\times \mathopen{[}-K,K\mathclose{]}\colon (x,y)\in Y_{0} \bigr{\}}
\end{equation*}
and define $\Pi \colon \mathcal{W}_{0} \to \Sigma_{3^{m}}$ as
\begin{equation*}
\Pi(z_{0}) := \pi(z(\cdot,z_{0})), \quad z_{0} \in \mathcal{W}_{0}.
\end{equation*}
Notice that the map $\Pi$ is well defined; indeed, if $z_{0} \in \mathcal{W}_{0}$, then $z(\cdot,z_{0})\in Y_{0}$.

The next result is an equivalent version of Theorem~\ref{th-6.5} where \textit{chaotic dynamics} are described
in terms of the Poincar\'{e} map.

\begin{theorem}\label{th-6.6}
Suppose that $\mu > \mu^{*}(\lambda)$. Then the map $\Pi \colon \mathcal{W}_{0} \to \Sigma_{3^{m}}$
is continuous, surjective and such that the diagram
\begin{equation*}
\xymatrix{
\mathcal{W}_{0} \ar[d]_{\mathlarger{\Pi}} \ar[r]^{\mathlarger{\Psi_{T}}} & \mathcal{W}_{0} \ar[d]^{\mathlarger{\Pi}}\\
\Sigma_{3^{m}} \ar[r]_{\mathlarger{\sigma}} & \Sigma_{3^{m}}}
\end{equation*}
commutes. Furthermore, for every integer $k\geq1$, the counterimage of any $k$-periodic sequence in
$\Sigma_{3^{m}}$ contains at least a point $w\in \mathcal{W}_{0}$
which is a $k$-periodic point of the Poincar\'{e} map and so that the solution
$u(t)$ of \eqref{eq-main}, with $(u(0),u'(0))=w$, is a $kT$-periodic solution of \eqref{eq-main}.
\end{theorem}

\begin{proof}
Let $\zeta: \mathcal{W}_{0} \to Y_{0}$ be the map which associates to
any initial point $z_{0}$ the solution $z(\cdot,z_{0})$ of \eqref{eq-xy} with $(x(0),y(0)) = z_{0}$.
We consider the diagram
\begin{equation*}
\xymatrix{
\mathcal{W}_{0} \ar[d]_{\mathlarger{\zeta}} \ar[r]^{\mathlarger{\Psi_{T}}} & \mathcal{W}_{0} \ar[d]^{\mathlarger{\zeta}}\\
Y_{0} \ar[r]_{\mathlarger{\psi}} & Y_{0}}
\end{equation*}
and observe that the map $\zeta$ is bijective, continuous and with continuous inverse. Indeed, if $z_{n} \to z_{0}$
in $\mathbb{R}^2$, then $z(t,z_{n})$ converges uniformly to $z(t,z_{0})$ on the compact subsets of $\mathbb{R}$.
The above diagram is also commutative because (by the uniqueness of the solutions to the initial value problems)
the solution of \eqref{eq-xy} starting at the point $z(T,z_{0})$ coincides with $z(t+T,z_{0})$.
From these remarks and the commutativity of the diagram in Theorem~\ref{th-6.5} we easily conclude.
\end{proof}

We conclude this section with a final remark concerning a dynamical consequence of Theorem~\ref{th-6.6}.
Consider again the alphabet $\mathscr{B}$ of $3^{m}$ elements of the form
$\omega = (\omega_{1},\ldots,\omega_{m})$, where $\omega_{i} \in \{0,1,2\}$ for each $i=1,\ldots,m$.
To each element $\omega \in \mathcal{B}$ we associate the set
\begin{equation*}
\mathcal{K}_{\omega} := \left\{ w \in \mathcal{W}_{0} \colon
\begin{array}{l}
\max_{t \in I^{+}_{i}}x(t,w) <r,  \text{ if } \omega_{i} = 0
\\
r<\max_{t \in I^{+}_{i}}x(t,w)<\rho,  \text{ if } \omega_{i} = 1
\\
\rho < \max_{t \in I^{+}_{i}}x(t,w)<R,  \text{ if } \omega_{i} = 2
\end{array} \right\},
\end{equation*}
which is compact, as an easy consequence of Lemma~\ref{lem-6.3}. By definition, the sets
$\mathcal{K}_{\omega}$ for $\omega \in \mathscr{B}$ are pairwise disjoint subsets of $\mathopen{[}0,R\mathclose{]} \times \mathopen{[}-K,K\mathclose{]}$.
Hence, another way to describe our results is the following.
\begin{quote}
\textit{For each two-sided sequence $(\mathcal{T}_{\ell})_{\ell \in \mathbb{Z}}$
there exists a corresponding sequence $(w_{\ell})_{\ell \in \mathbb{Z}} \in (\mathcal{W}_{0})^{\mathbb{Z}}$
such that, for all $\ell \in \mathbb{Z}$,
\begin{equation}\label{chaos}
w_{\ell + 1} = \Psi_T(w_{\ell}) \quad \text{ and } \quad w_{\ell} \in \mathcal{K}_{\mathcal{T}_{\ell}};
\end{equation}
moreover,
whenever $(\mathcal{T}_{\ell})_{\ell \in \mathbb{Z}}$ is a $k$-periodic sequence for some integer $k \geq 1$, there exists a
$k$-periodic sequence $(w_{\ell})_{\ell \in \mathbb{Z}} \in (\mathcal{W}_{0})^{\mathbb{Z}}$ satisfying condition \eqref{chaos}.}
\end{quote}

In this manner, we enter a setting of coin-tossing type dynamics widely explored in the literature.
As a consequence, in the case $m=1$, we obtain a dynamics on three symbols, described as itineraries
for the Poincar\'{e} map jumping among three compact mutually disjoint sets $\mathcal{K}_{0},\mathcal{K}_{1},\mathcal{K}_{2}$.
A previous result in this direction, but involving only two symbols, was obtained in \cite{BoZa-12b} with a completely different approach.

\section{Related results}\label{section-7}

In this final section we briefly describe some results which can be obtained
by minor modifications of the arguments developed along this paper.

\subsection{The non-Hamiltonian case}\label{section-7.1}

One of the advantages in obtaining existence/multiplicity results with a topological degree technique
lies in the fact that the degree is stable with respect to small perturbations of the operator.
Such a remark, when applied to equation \eqref{eq-main}, allows us to establish the same result for the equation
\begin{equation}\label{eq-7.1}
u'' + cu' + \bigr{(} \lambda a^{+}(t) - \mu a^{-}(t) \bigr{)} g(u) = 0
\end{equation}
where $c\in\mathbb{R}$ and $c\neq 0$. More precisely, in the same setting of Theorem~\ref{main-theorem}, once that $\lambda > \lambda^{*}$ and
$\mu > \mu^{*}(\lambda)$ are fixed, there exists a constant $\varepsilon = \varepsilon(\lambda,\mu)>0$ such that the statement of the
theorem is still true for any $c\in\mathbb{R}$ with $|c| < \varepsilon$. The same remark applies
to the results in Section~\ref{section-6}, so that we can prove the existence of infinitely many positive subharmonic solutions as well as
the presence of chaotic dynamics on $3^{m}$ symbols also for equation \eqref{eq-7.1}.
Typically, results about multiplicity of subharmonic solutions are achieved by exploiting the Hamiltonian structure
of the equation and therefore using variational or symplectic techniques. Our approach shows that, for equations with a sign-indefinite weight,
we can achieve such results also in the non-Hamiltonian case.

A possibly interesting question which naturally arises is whether these multiplicity results are still valid for an arbitrary $c \in \mathbb{R}$.
In the superlinear indefinite case, Capietto, Dambrosio and Papini in
\cite{CaDaPa-02} produced such kind of results for sign-changing (oscillatory) solutions. More recently, in \cite{FeZa-pp2015} complex dynamics
for positive solutions has been obtained. Concerning our super-sublinear setting,
all the abstract approach and the strategy for the proof work exactly the same for the linear differential operator
$u \mapsto -u'' - cu'$ for an arbitrary $c \in \mathbb{R}$ (see Remark~\ref{rem-2.1} and Remark~\ref{rem-3.1}).
Thus, the only problem in extending all our results of the previous sections to equation \eqref{eq-7.1} comes from
some additional difficulties related to the technical estimates. In particular, we have often exploited the convexity of the solutions in the intervals $I^{-}_{i}$
and their concavity in the intervals $I^{+}_{i}$. In the recent paper
\cite{BoFeZa-15} we have proved the existence of two positive $T$-periodic solutions to equation \eqref{eq-7.1}
by effectively replacing the convexity/concavity properties with suitable monotonicity properties
for the map $t \mapsto e^{ct}u'(t)$. Similar tricks have been successfully applied in \cite{FeZa-pp2015} to obtain multiplicity results
for equation \eqref{eq-7.1} with a superlinear $g(s)$. It is therefore quite reasonable that
these arguments can be adapted to our case. However, due to the lengthy and complex technical details
required in Section~\ref{section-4}, we have preferred to skip further investigations in this direction.

\subsection{Neumann and Dirichlet boundary conditions}\label{section-7.2}

As anticipated, versions of Theorem~\ref{main-theorem} for both Neumann and Dirichlet boundary conditions can be given.
In these cases, we can consider a slightly more general sign condition for the measurable weight function
$a \colon \mathopen{[}0,T\mathclose{]} \to \mathbb{R}$, which reads as follows:
\begin{itemize}
\item [$(a_{**})$] \textit{there exist $2m + 2$ points (with $m \geq 1$)
\begin{equation*}
0 = \tau_{0} \leq \sigma_{1} < \tau_{1} < \ldots < \sigma_{i} < \tau_{i} < \ldots < \sigma_{m} < \tau_{m} \leq \sigma_{m+1} = T
\end{equation*}
such that $a(t) \succ 0$ on $\mathopen{[}\sigma_{i},\tau_{i}\mathclose{]}$, for $i=1,\ldots,m$, and $a(t) \prec 0$ on $\mathopen{[}\tau_{i},\sigma_{i+1}\mathclose{]}$, for $i=0,\ldots,m$.}
\end{itemize}
This means that $a(t)$ has $m$ positive humps $\mathopen{[}\sigma_{i},\tau_{i}\mathclose{]}$ ($i=1,\ldots,m$) separated by $m-1$ negative ones
$\mathopen{[}\tau_{i},\sigma_{i+1}\mathclose{]}$ ($i=1,\ldots,m-1$); in addition, $a(t)$ might have one/two further negativity intervals, precisely an initial one
$\mathopen{[}\tau_{0},\sigma_{1}\mathclose{]} = \mathopen{[}0,\sigma_{1}\mathclose{]}$
or/and a final one $\mathopen{[}\tau_{m},\sigma_{m+1}\mathclose{]} = \mathopen{[}\tau_{m},T\mathclose{]}$
(compare with Remark~\ref{rem-1.1}). In this setting, the following result holds true.

\begin{theorem}\label{th-7.1}
Let $g \colon \mathbb{R}^{+} \to \mathbb{R}^{+}$ be a continuous function satisfying $(g_{*})$, $(g_{0})$ and $(g_{\infty})$.
Let $a \colon \mathopen{[}0,T\mathclose{]} \to \mathbb{R}$ be an integrable function satisfying $(a_{**})$.
Then there exists $\lambda^{*} > 0$ such that for each $\lambda > \lambda^{*}$ there exists $\mu^{*}(\lambda) > 0$ such that
for each $\mu > \mu^{*}(\lambda)$ the Neumann problem
\begin{equation*}
\begin{cases}
\, u'' + \bigr{(} \lambda a^{+}(t) - \mu a^{-}(t) \bigr{)} g(u) = 0 \\
\, u'(0) = u'(T) = 0
\end{cases}
\end{equation*}
has at least $3^{m}-1$ positive solutions. The same result holds for the Dirichlet problem
\begin{equation*}
\begin{cases}
\, u'' + \bigr{(} \lambda a^{+}(t) - \mu a^{-}(t) \bigr{)} g(u) = 0 \\
\, u(0) = u(T) = 0.
\end{cases}
\end{equation*}
\end{theorem}

Of course, such solutions can again be coded via a non-null string $\mathcal{S} \in \{0,1,2\}^{m}$ as described in Theorem~\ref{main-theorem}.
We also remark that, as usual, a positive solution of the Dirichlet problem is a function $u(t)$
solving the equation and such that $u(0) = u(T) = 0$ and $u(t) > 0$ for any $t \in \mathopen{]}0,T\mathclose{[}$.

For the proof of Theorem~\ref{th-7.1}, we rely on the abstract setting of Section~\ref{section-2} (with the changes underlined in Remark~\ref{rem-2.1})
and on the general strategy presented in Section~\ref{section-3.1}. The key point is then the verification of the assumptions of Lemma~\ref{lem-deg0}
and Lemma~\ref{lem-deg1} (in the slightly modified versions described in Remark~\ref{rem-3.1}).
To this end, we can take advantage of the technical estimates developed in Section~\ref{section-4.1} (which indeed are independent of the boundary conditions)
and we can prove the result with minor modifications of the arguments in the remaining part of Section~\ref{section-4}.

Finally, we observe that the same result can be obtained for positive solutions of equation \eqref{eq-main} satisfying
the mixed boundary conditions $u(0) = u'(T) = 0$ or $u'(0) = u(T) = 0$ (compare with \cite[\S~5.4]{FeZa-jde2015}).

\subsection{Radially symmetric positive solutions}\label{section-7.3}

As a standard consequence of Theorem~\ref{th-7.1}, we can produce multiplicity results for radially symmetric positive
solutions to elliptic BVPs on an annulus.

More precisely, let $\|\cdot\|$ be the Euclidean norm in $\mathbb{R}^{N}$ (for $N \geq 2$) and let
\begin{equation*}
\Omega := \bigl{\{}x\in {\mathbb{R}}^{N} \colon R_{1} < \|x\| < R_{2}\bigr{\}}
\end{equation*}
be an open annular domain, with $0 < R_{1} < R_{2}$. We deal with the elliptic partial differential equation
\begin{equation}\label{eq-pde-rad}
-\Delta \,u = \bigl{(} \lambda q^{+}(x) - \mu q^{-}(x)\bigr{)} g(u) \quad \text{ in } \Omega
\end{equation}
together with Neumann boundary conditions
\begin{equation}\label{neu}
\dfrac{\partial u}{\partial {\bf n}} = 0 \quad \text{ on } \partial\Omega
\end{equation}
or Dirichlet boundary conditions
\begin{equation}\label{dir}
u = 0 \quad \text{ on } \partial\Omega.
\end{equation}
For simplicity, we look for classical solutions to \eqref{eq-pde-rad} (namely, $u \in \mathcal{C}^2(\overline{\Omega})$) and, accordingly,
we assume that $q \colon \overline{\Omega} \to \mathbb{R}$ is a \textit{continuous} function.
Moreover, in order to transform the partial differential equation \eqref{eq-pde-rad} into a second order ordinary differential equation
of the form \eqref{eq-main} so as to apply Theorem~\ref{th-7.1}, we also require that $q(x)$ is a radially symmetric function,
i.e.~there exists a continuous function $\mathcal{Q} \colon \mathopen{[}R_{1},R_{2}\mathclose{]} \to \mathbb{R}$ such that
\begin{equation}\label{qQ}
q(x) = \mathcal{Q}(\|x\|), \quad \forall \, x \in \overline{\Omega}.
\end{equation}
We also set
\begin{equation*}
\mathcal{Q}_{\lambda,\mu}(r) := \lambda \mathcal{Q}^{+}(r) - \mu \mathcal{Q}^{-}(r), \quad r\in \mathopen{[}R_{1},R_{2}\mathclose{]},
\end{equation*}
where, as usual, $\lambda,\mu > 0$.

Looking for radially symmetric (classical) solutions to \eqref{eq-pde-rad}, i.e.~solutions of the
form $u(x) = \mathcal{U}(\|x\|)$ where $\mathcal{U}(r)$ is a scalar function defined on $\mathopen{[}R_{1},R_{2}\mathclose{]}$,
we transform equation \eqref{eq-pde-rad} into
\begin{equation}\label{eq-rad}
\bigl{(}r^{N-1}\, \mathcal{U}'\bigr{)}' + r^{N-1} \mathcal{Q}_{\lambda,\mu}(r) g(\mathcal{U}) = 0.
\end{equation}
Moreover, the boundary conditions \eqref{neu} and \eqref{dir}
become
\begin{equation*}
\mathcal{U}(R_{1}) = \mathcal{U}(R_{2}) = 0 \quad \text{ and } \quad \mathcal{U}'(R_{1}) = \mathcal{U}'(R_{2}) = 0,
\end{equation*}
respectively. Via the change of variable
\begin{equation*}
t = h(r):= \int_{R_{1}}^{r} \xi^{1-N} ~\!d\xi
\end{equation*}
and the positions
\begin{equation*}
T:= \int_{R_{1}}^{R_{2}} \xi^{1-N} ~\!d\xi, \quad r(t):= h^{-1}(t) \quad \text{and} \quad v(t)={\mathcal{U}}(r(t)),
\end{equation*}
we can further convert \eqref{eq-rad} and the corresponding boundary conditions into the Neumann and Dirichlet problems
\begin{equation*}
\begin{cases}
\, v'' + a_{\lambda,\mu}(t) g(v) = 0 \\
\, v'(0) = v'(T) = 0
\end{cases}
\quad \text{ and } \quad
\begin{cases}
\, v'' + a_{\lambda,\mu}(t) g(v) = 0 \\
\, v(0) = v(T) = 0,
\end{cases}
\end{equation*}
respectively, where
\begin{equation*}
a(t):= r(t)^{2(N-1)}{\mathcal{Q}}(r(t)), \quad t \in \mathopen{[}0,T\mathclose{]},
\end{equation*}
and $a_{\lambda,\mu}(t) := \lambda a^{+}(t) - \mu a^{-}(t)$, for $t \in \mathopen{[}0,T\mathclose{]}$.

In this setting, Theorem~\ref{th-7.1} gives the following result. The straightforward proof is omitted.

\begin{theorem}\label{th-7.2}
Let $g \colon \mathbb{R}^{+} \to \mathbb{R}^{+}$ be a continuous function satisfying $(g_{*})$, $(g_{0})$ and $(g_{\infty})$.
Let $\mathcal{Q} \colon \mathopen{[}R_{1},R_{2}\mathclose{]} \to \mathbb{R}$ be a continuous function satisfying
\begin{itemize}
\item [$(\mathcal{Q}_{**})$]
\textit{there exist $2m + 2$ points (with $m \geq 1$)
\begin{equation*}
R_{1} = \tau_{0} \leq \sigma_{1} < \tau_{1} < \ldots < \sigma_{i} < \tau_{i} < \ldots < \sigma_{m} < \tau_{m} \leq \sigma_{m+1} = R_{2}
\end{equation*}
such that $\mathcal{Q}(r) \succ 0$ on $\mathopen{[}\sigma_{i},\tau_{i}\mathclose{]}$, for $i=1,\ldots,m$,
and $\mathcal{Q}(r) \prec 0$ on $\mathopen{[}\tau_{i},\sigma_{i+1}\mathclose{]}$, for $i=0,\ldots,m$,}
\end{itemize}
and let $q \colon \overline{\Omega} \to \mathbb{R}$ be defined as in \eqref{qQ}.
Then there exists $\lambda^{*} > 0$ such that for each $\lambda > \lambda^{*}$ there exists $\mu^{*}(\lambda) > 0$ such that
for each $\mu > \mu^{*}(\lambda)$ the Neumann problem associated with \eqref{eq-pde-rad} has at least $3^{m}-1$ radially symmetric
positive (classical) solutions. The same result holds for the Dirichlet problem associated with \eqref{eq-pde-rad}.
\end{theorem}

\appendix
\section{Combinatorial argument}\label{appendix-A}

In this appendix, we present the combinatorial argument needed in the proof of Theorem~\ref{deg-Lambda}.
In more detail, recalling the definitions of $\Omega^{\mathcal{I},\mathcal{J}}_{(r,\rho,R)}$ and $\Lambda^{\mathcal{I},\mathcal{J}}_{(r,\rho,R)}$ given in \eqref{eq-Omega}
and \eqref{eq-Lambda} respectively, from formula \eqref{eq-3.1} (concerning the degrees on $\Omega^{\mathcal{I},\mathcal{J}}_{(r,\rho,R)}$),
we prove that, for any pair of subset of indices $\mathcal{I},\mathcal{J}\subseteq \{1,\ldots,m\}$ with $\mathcal{I}\cap\mathcal{J}=\emptyset$, we have
\begin{equation*}
D_{L}\bigl{(}L-N_{\lambda,\mu},\Lambda^{\mathcal{I},\mathcal{J}}_{(r,\rho,R)}\bigr{)} =
(-1)^{\# \mathcal{I}}.
\end{equation*}

We offer two independent proofs since we believe that both possess some peculiar aspects which might be also
adapted to different situations.

\subsection{First argument}\label{section-A1}

In this first part we present a combinatorial argument which is related to the concept of \textit{valuation},
as introduced in \cite{KlRo-97}.

Let $m\in\mathbb{N}$ be a positive integer. We denote by
\begin{equation*}
\mathbb{A} := \bigl{\{} A_{1} \times A_{2} \times \ldots \times A_{m} \colon A_{i}\in\mathscr{P}(\{0,1,2\})\bigr{\}}
\end{equation*}
the set of the $8^{m}$ Cartesian products of $m$ subsets of $\{0,1,2\}$.

Let
\begin{equation}\label{eq-A.1}
\mathcal{A} := A_{1} \times A_{2} \times \ldots \times A_{m}
\end{equation}
be an element of $\mathbb{A}$, let $i\in\{1,\ldots,m\}$ be a fixed index and let also $B_{i}\in\mathscr{P}(\{0,1,2\})$.
We introduce the following notation
\begin{equation*}
\mathcal{A}[i:B_{i}] := A_{1} \times \ldots \times A_{i-1} \times B_{i} \times A_{i+1} \times \ldots \times A_{m}.
\end{equation*}
Note that for any fixed $\mathcal{A}$ as above and $i\in\{1,\ldots,m\}$ it holds that
$\mathcal{A} = \mathcal{A}[i:A_{i}]$.

We consider a function
\begin{equation*}
d \colon \mathbb{A} \to \mathbb{Z}
\end{equation*}
which satisfies the following property.
\begin{quote}
\textit{Additivity property. }
Let $i\in\{1,\ldots,m\}$ and $B_{i}\in \mathscr{P}(\{0,1,2\})$.
Suppose that $B'_{i},B''_{i}\subseteq B_{i}$ are disjoint (possibly empty) and such that
\begin{equation*}
B_{i} = B'_{i} \cup B''_{i}.
\end{equation*}
Then, for all $\mathcal{A}\in\mathbb{A}$, it holds that
\begin{equation*}
d(\mathcal{A}[i : B_{i}]) = d(\mathcal{A}[i : B'_{i}]) + d(\mathcal{A}[i : B''_{i}]).
\end{equation*}
\end{quote}
From the additivity property (applied in the case $B_{i} = B'_{i} = B''_{i} = \emptyset$)
we immediately obtain that, if there exists an index $i\in\{1,\ldots,m\}$ such that $A_{i}=\emptyset$, then
$d(A_{1} \times \ldots \times A_{m}) = 0$.

Moreover, we assume that $d$ satisfies the following rules.
\begin{itemize}
\item [$(R1)$]
If there exists an index $i\in\{1,\ldots,m\}$ such that $A_{i}=\{0,1\}$ and $A_{j} \in \{ \{0\}, \{0,1,2\}\}$ for all $j\in\{1,\ldots,m\}$
such that $A_{j}\neq\{0,1\}$, then
\begin{equation*}
d(A_{1} \times \ldots \times A_{m}) = 0.
\end{equation*}
\item [$(R2)$]
If $A_{i} \in \{ \{0\}, \{0,1,2\} \}$, for all $i=1,\ldots,m$,
then
\begin{equation*}
d(A_{1} \times \ldots \times A_{m}) = 1.
\end{equation*}
\end{itemize}

Our goal is to compute $d(A_{1} \times \ldots \times A_{m})$ when
$A_{i} \in \{ \{0\}, \{1\}, \{2\}, \{0,1,2\} \}$, for all $i=1,\ldots,m$.

As a first step we prove a generalization of rule $(R1)$.

\begin{lemma}\label{lem-A1}
If there exists an index $i\in\{1,\ldots,m\}$ such that $A_{i}=\{0,1\}$ and
$A_{j} \in \{ \{0\}, \{2\}, \{0,1,2\} \}$ for all $j\in\{1,\ldots,m\}$
such that $A_{j}\neq\{0,1\}$, then
\begin{equation*}
d(A_{1} \times \ldots \times A_{m}) = 0.
\end{equation*}
\end{lemma}

\begin{proof}
We prove the statement by induction on the non-negative integer
\begin{equation*}
k := \# \bigl{\{} j\in\{1,\ldots,m\} \colon A_{j}=\{2\} \bigr{\}}.
\end{equation*}

\smallskip
\noindent
\textit{Case $k=0$. } If there is no $j\in\{1,\ldots,m\}$ such that $A_{j}=\{2\}$, the thesis follows by rule $(R1)$.

\smallskip
\noindent
\textit{Case $k=1$. } Suppose that there is exactly one index $j\in\{1,\ldots,m\}$ such that $A_{j}=\{2\}$.
Recalling the definition of $\mathcal{A}$ in \eqref{eq-A.1}, it is easy to see that
\begin{equation*}
\mathcal{A}[j : \{0,1,2\}] = \mathcal{A} \cup \mathcal{A}[j : \{0,1\}].
\end{equation*}
Then, by the additivity property of $d$ and rule $(R1)$, we obtain
\begin{equation*}
d(\mathcal{A}) = d(\mathcal{A}[j : \{0,1,2\}]) - d(\mathcal{A}[j : \{0,1\}]) = 0 - 0 = 0.
\end{equation*}

\smallskip
\noindent
\textit{Inductive step. } Suppose that the statement holds for $k$. We prove it for $k+1$.
Let $j\in\{1,\ldots,m\}$ be such that $A_{j}=\{2\}$. As above, from
\begin{equation*}
\mathcal{A}[j : \{0,1,2\}] = \mathcal{A} \cup \mathcal{A}[j : \{0,1\}],
\end{equation*}
we obtain
\begin{equation*}
d(\mathcal{A}) = d(\mathcal{A}[j : \{0,1,2\}]) - d(\mathcal{A}[j : \{0,1\}]).
\end{equation*}
By the inductive hypothesis, we know that $d(\mathcal{A}[j : \{0,1,2\}])=0$ and $d(\mathcal{A}[j : \{0,1\}])=0$
(since $\mathcal{A}[j : \{0,1,2\}]$ and $\mathcal{A}[j : \{0,1\}]$ both have exactly $k$ indices $i$ such that $A_{i}=\{2\}$).
The thesis immediately follows.
\end{proof}

Now we provide a generalization of rule $(R2)$.

\begin{lemma}\label{lem-A2}
If $A_{i} \in \{ \{0\}, \{2\}, \{0,1,2\} \}$, for all $i=1,\ldots,m$, then
\begin{equation*}
d(A_{1} \times \ldots \times A_{m}) = 1.
\end{equation*}
\end{lemma}

\begin{proof}
We prove the statement by induction on the non-negative integer
\begin{equation*}
k := \# \bigl{\{} j\in\{1,\ldots,m\} \colon A_{j}=\{2\} \bigr{\}}.
\end{equation*}

\smallskip
\noindent
\textit{Case $k=0$. } If there is no $j\in\{1,\ldots,m\}$ such that $A_{j}=\{2\}$, the thesis follows by rule $(R2)$.

\smallskip
\noindent
\textit{Case $k=1$. } Suppose that there is exactly one index $j\in\{1,\ldots,m\}$ such that $A_{j}=\{2\}$.
Recalling the definition of $\mathcal{A}$ in \eqref{eq-A.1}, it is easy to see that
\begin{equation*}
\mathcal{A}[j : \{0,1,2\}] = \mathcal{A} \cup \mathcal{A}[j : \{0,1\}].
\end{equation*}
Then, by the additivity property of $d$ and rules $(R1)$ and $(R2)$, we obtain
\begin{equation*}
d(\mathcal{A}) = d(\mathcal{A}[j : \{0,1,2\}]) - d(\mathcal{A}[j : \{0,1\}]) = 1 - 0 = 1.
\end{equation*}

\smallskip
\noindent
\textit{Inductive step. } Suppose that the statement holds for $k$. We prove it for $k+1$.
Let $j\in\{1,\ldots,m\}$ be such that $A_{j}=\{2\}$.
As above, from
\begin{equation*}
\mathcal{A}[j : \{0,1,2\}] = \mathcal{A} \cup \mathcal{A}[j : \{0,1\}],
\end{equation*}
we obtain
\begin{equation*}
d(\mathcal{A}) = d(\mathcal{A}[j : \{0,1,2\}]) - d(\mathcal{A}[j : \{0,1\}]).
\end{equation*}
By the inductive hypothesis, we obtain that $d(\mathcal{A}[j : \{0,1,2\}])=1$
(since $\mathcal{A}[j : \{0,1,2\}]$ has exactly $k$ indices $i$ such that $A_{i}=\{2\}$).
By Lemma~\ref{lem-A1}, we have that $d(\mathcal{A}[j : \{0,1\}])=0$.
The thesis immediately follows.
\end{proof}

Finally, using the rules presented above, we obtain the final lemma.

\begin{lemma}\label{lem-A3}
If $A_{i} \in \{ \{0\}, \{1\}, \{2\}, \{0,1,2\} \}$, for all $i=1,\ldots,m$, then
\begin{equation*}
d(A_{1} \times \ldots \times A_{m}) = (-1)^{\#\mathcal{I}},
\end{equation*}
where $\mathcal{I} := \bigl{\{} i\in\{1,\ldots,m\} \colon A_{i}=\{1\} \bigr{\}}$.
\end{lemma}

\begin{proof}
We prove the statement by induction on the non-negative integer $k:=\# \mathcal{I}$.

\smallskip
\noindent
\textit{Case $k=0$. } If there is no $i\in\{1,\ldots,m\}$ such that $A_{i}=\{1\}$, the thesis follows by Lemma~\ref{lem-A2}.

\smallskip
\noindent
\textit{Case $k=1$. } Suppose that there is exactly one index $i\in\{1,\ldots,m\}$ such that $A_{i}=\{1\}$.
Recalling the definition of $\mathcal{A}$ in \eqref{eq-A.1}, it is easy to see that
\begin{equation*}
\mathcal{A}[i : \{0,1,2\}] = \mathcal{A}[i : \{0\}] \cup \mathcal{A} \cup \mathcal{A}[i : \{2\}].
\end{equation*}
Then, by the additivity property of $d$ and Lemma~\ref{lem-A2}, we obtain
\begin{equation*}
\begin{aligned}
d(\mathcal{A})
   &= d(\mathcal{A}[i : \{0,1,2\}]) - d(\mathcal{A}[i : \{0\}]) - d(\mathcal{A}[i : \{2\}])
\\ &= 1-1-1 = -1 = (-1)^{\#\mathcal{I}}.
\end{aligned}
\end{equation*}

\smallskip
\noindent
\textit{Inductive step. } Suppose that the statement holds when the set $\mathcal{I}$ has $k$ elements.
We prove it for $\# \mathcal{I}=k+1$.
Let $i\in\{1,\ldots,m\}$ be such that $A_{i}=\{1\}$. By assumption there are $k+1$ indices with such a property.
As above, from
\begin{equation*}
\mathcal{A}[i : \{0,1,2\}] = \mathcal{A}[i : \{0\}] \cup \mathcal{A} \cup \mathcal{A}[i : \{2\}],
\end{equation*}
we obtain
\begin{equation*}
d(\mathcal{A}) = d(\mathcal{A}[i : \{0,1,2\}]) - d(\mathcal{A}[i : \{0\}]) - d(\mathcal{A}[i : \{2\}]).
\end{equation*}
Now, all the sets $\mathcal{A}[i : \{0,1,2\}]$,
$\mathcal{A}[i : \{0\}]$ and $\mathcal{A}[i : \{2\}]$
have precisely $k$ indices $j$ such that $A_{j}=\{1\}$.
Then, by the inductive hypothesis, we obtain that
\begin{equation*}
d(\mathcal{A}[i : \{0,1,2\}]) = d(\mathcal{A}[i : \{0\}]) = d(\mathcal{A}[i : \{2\}]) = (-1)^{k}
\end{equation*}
and hence
\begin{equation*}
d(\mathcal{A}) = -(-1)^{k} = (-1)^{k+1} = (-1)^{\#\mathcal{I}}.
\end{equation*}
The thesis immediately follows.
\end{proof}

\medskip

We conclude this first part by showing how to apply this approach to obtain formula \eqref{eq-3.2}.

To any element $\mathcal{A}\in\mathbb{A}$ we associate an open set $\Omega_{\mathcal{A}}$
made up of the continuous functions $u \colon \mathopen{[}0,T\mathclose{]} \to \mathbb{R}$ which, for all $i=1,\ldots,m$, satisfy
\begin{itemize}
\item $\max_{t\in I^{+}_{i}}|u(t)|<r$, if $A_{i}=\{0\}$;
\item $r<\max_{t\in I^{+}_{i}}|u(t)|<\rho$, if $A_{i}=\{1\}$;
\item $\rho<\max_{t\in I^{+}_{i}}|u(t)|<R$, if $A_{i}=\{2\}$;
\item $\max_{t\in I^{+}_{i}}|u(t)|<\rho$, if $A_{i}=\{0,1\}$;
\item either $\max_{t\in I^{+}_{i}}|u(t)|<r$ or $\rho<\max_{t\in I^{+}_{i}}|u(t)|<R$, if $A_{i}=\{0,2\}$;
\item $r<\max_{t\in I^{+}_{i}}|u(t)|<R$, if $A_{i}=\{1,2\}$;
\item $\max_{t\in I^{+}_{i}}|u(t)|<R$, if $A_{i}=\{0,1,2\}$.
\end{itemize}
By convention, we also set $\Omega_{\mathcal{A}} = \emptyset$ if there is an index $i\in \{1,\ldots,m\}$
such that $A_{i} = \emptyset$. In this manner the set $\Omega_{\mathcal{A}}$ is well define for
every $\mathcal{A}\in\mathbb{A}$.

Having fixed $\rho$, $\lambda>\lambda^{*}$, $r < \rho < R$ and $\mu > \mu^{*}(\lambda)$ as in Section~\ref{section-4},
we have that the coincidence degree $D_{L}(L-N_{\lambda,\mu},\Omega_{\mathcal{A}})$ is well defined for
every $\mathcal{A}\in\mathbb{A}$.
Hence we set
\begin{equation*}
d(\mathcal{A}):=D_{L}\bigl{(}L-N_{\lambda,\mu},\Omega_{\mathcal{A}}\bigr{)}.
\end{equation*}
Notice that the sets $\Omega^{\mathcal{I},\mathcal{J}}_{(r,\rho,R)}$ introduced in \eqref{eq-Omega} are of the form
$\Omega_{\mathcal{A}}$ for $\mathcal{A}$ with $A_{i} = \{0\}$ for any $i\in \{1,\ldots,m\} \setminus(\mathcal{I}\cup\mathcal{J})$,
$A_{i} = \{0,1\}$ for any $i\in \mathcal{I}$ and $A_{i} = \{0,1,2\}$ for any $i\in \mathcal{J}$. Similarly, the sets
$\Lambda^{\mathcal{I},\mathcal{J}}_{(r,\rho,R)}$ introduced in \eqref{eq-Lambda} are of the form
$\Omega_{\mathcal{A}}$ for $\mathcal{A}$ with $A_{i} = \{0\}$ for any $i\in \{1,\ldots,m\} \setminus(\mathcal{I}\cup\mathcal{J})$,
$A_{i} = \{1\}$ for any $i\in \mathcal{I}$ and $A_{i} = \{2\}$ for any $i\in \mathcal{J}$.

With these positions, the \textit{additivity property} of the valuation $d$ follows from the additivity property of the coincidence degree.
Moreover, rules $(R1)$ and $(R2)$ are satisfied since they correspond to formula \eqref{eq-3.1}. Then, all the above lemmas on the valuation $d$
apply and, in particular, Lemma~\ref{lem-A3} gives precisely formula \eqref{eq-3.2}. This completes the proof of Theorem~\ref{deg-Lambda}.

\subsection{Second argument}\label{section-A2}

In this second part we present a different combinatorial argument, in the same spirit of the one adopted in
\cite[Lemma~4.1]{FeZa-jde2015}.

Let $r,\rho,R$ be three positive real numbers such that $0<r<\rho<R$ and let $m\geq1$ be an integer.
Recalling the definitions of $\Omega^{\mathcal{I},\mathcal{J}}_{(r,\rho,R)}$ and $\Lambda^{\mathcal{I},\mathcal{J}}_{(r,\rho,R)}$ given in \eqref{eq-Omega}
and \eqref{eq-Lambda} respectively, we note that, for any pair of subset of indices $\mathcal{I},\mathcal{J}\subseteq \{1,\ldots,m\}$
with $\mathcal{I}\cap\mathcal{J}=\emptyset$, we have
\begin{equation}\label{eq-union}
\Omega^{\mathcal{I},\mathcal{J}}_{(r,\rho,R)}
=\bigcup_{\substack{\mathcal{I}'\subseteq\mathcal{I}\cup\mathcal{J} \\ \mathcal{J}'\subseteq\mathcal{J} \\ \mathcal{I}'\cap\mathcal{J}'=\emptyset}}
\Lambda^{\mathcal{I}',\mathcal{J}'}_{(r,\rho,R)},
\end{equation}
and the union is disjoint, since $\Lambda^{\mathcal{I}',\mathcal{J}'}_{(r,\rho,R)} \cap \Lambda^{\mathcal{I}'',\mathcal{J}''}_{(r,\rho,R)} = \emptyset$,
for $\mathcal{I}'\neq\mathcal{I}''$ or for $\mathcal{J}'\neq\mathcal{J}''$.

Observe that the set of all the pairs $(\mathcal{I},\mathcal{J})$ with $\mathcal{I},\mathcal{J}\subseteq\{1,\ldots,m\}$ such that
$\mathcal{I}\cap\mathcal{J} = \emptyset$ has cardinality equal to $3^{m}$.

Now we are in position to present the following result.

\begin{lemma}\label{lemma-1}
Let $\mathcal{I},\mathcal{J}\subseteq\{1,\ldots,m\}$ be two subsets of indices (possibly empty) such that $\mathcal{I}\cap\mathcal{J}=\emptyset$.
Suppose that the coincidence degrees $D_{L}\bigl{(}L-N_{\lambda,\mu},\Omega^{\mathcal{I}',\mathcal{J}'}_{(r,\rho,R)})$
and $D_{L}\bigl{(}L-N_{\lambda,\mu},\Lambda^{\mathcal{I}',\mathcal{J}'}_{(r,\rho,R)})$ are well defined
for all ${\mathcal{I}'}\subseteq {\mathcal{I}}\cup\mathcal{J}$ and for all ${\mathcal{J}'}\subseteq \mathcal{J}$
with $\mathcal{I}'\cap\mathcal{J}'=\emptyset$. Assume also
\begin{equation}\label{deg=1}
D_{L}\bigl{(}L-N_{\lambda,\mu},\Omega^{\mathcal{I}',\mathcal{J}'}_{(r,\rho,R)})=1, \quad \text{if } \;\mathcal{I}' = \emptyset,
\end{equation}
and
\begin{equation}\label{deg=0}
D_{L}\bigl{(}L-N_{\lambda,\mu},\Omega^{\mathcal{I}',\mathcal{J}'}_{(r,\rho,R)})=0,
\quad \text{if } \; \mathcal{I}' \neq \emptyset.
\end{equation}
Then
\begin{equation}\label{deg=-1}
D_{L}\bigl{(}L-N_{\lambda,\mu},\Lambda^{\mathcal{I},\mathcal{J}}_{(r,\rho,R)})=(-1)^{\#\mathcal{I}}.
\end{equation}
\end{lemma}

\begin{proof}
For simplicity of notation, in this proof we set
\begin{equation*}
\Omega^{\mathcal{I},\mathcal{J}}=\Omega^{\mathcal{I},\mathcal{J}}_{(r,\rho,R)}
\quad \text{ and } \quad
\Lambda^{\mathcal{I},\mathcal{J}}=\Lambda^{\mathcal{I},\mathcal{J}}_{(r,\rho,R)}.
\end{equation*}

First of all, we underline that $\Omega^{\emptyset,\emptyset}=\Lambda^{\emptyset,\emptyset}$ and, in view of \eqref{deg=1}, we have that
\begin{equation}\label{deg1}
D_{L}\bigl{(}L-N_{\lambda,\mu},\Omega^{\emptyset,\emptyset})=D_{L}\bigl{(}L-N_{\lambda,\mu},\Lambda^{\emptyset,\emptyset})=1.
\end{equation}
Hence the conclusion is
trivially satisfied when ${\mathcal{I}} = {\mathcal{J}} = \emptyset$.

Now we consider two arbitrary subsets of indices (possibly empty)
such that $\mathcal{I}\cup\mathcal{J}\neq\emptyset$ and $\mathcal{I}\cap\mathcal{J}=\emptyset$.
We are going to prove formula \eqref{deg=-1} by using an inductive argument.
Instead of a double induction on $\#{\mathcal{I}}$ and on $\#{\mathcal{J}}$, it seems more convenient
to introduce the bijection
\begin{equation*}
(i,j) \leftrightarrow i+(m+1)j
\end{equation*}
from the set of couples $(i,j)\in\{1,\ldots,m\}^{2}$ and the integers $0\leq n \leq m(m+2)$, in order to
reduce our argument to a single induction. More precisely, we define
\begin{equation*}
n:=\#{\mathcal{I}}+(m+1)\#{\mathcal{J}}\geq 1
\end{equation*}
and, for every integer $k$ with $0\leq k\leq n$, we introduce the property ${\mathscr{P}(k)}$
which reads as follows.
\begin{itemize}
\item[${\mathscr{P}(k)}$:]
\textit{The formula
\begin{equation*}
D_{L}(L-N_{\lambda,\mu},\Lambda^{\mathcal{I}',\mathcal{J}'})=(-1)^{\#\mathcal{I}'}
\end{equation*}
holds for each ${\mathcal{I}'}\subseteq\mathcal{I}\cup\mathcal{J}$ and for each ${\mathcal{J}'}\subseteq{\mathcal{J}}$ such that
$\mathcal{I}'\cap\mathcal{J}'=\emptyset$ and $\#\mathcal{I}'+(m+1)\#\mathcal{J}'\leq k$.}
\end{itemize}
In this manner, if we are able to prove ${\mathscr{P}}(n)$, then \eqref{deg=-1} immediately follows.

\smallskip
\noindent
\textit{Verification of $\mathscr{P}(0)$. } See \eqref{deg1}.

\smallskip
\noindent
\textit{Verification of $\mathscr{P}(1)$. }
For $\mathcal{I}' = \mathcal{J}' = \emptyset$ the result is already proved in \eqref{deg1}.
If $\mathcal{I}'=\{i\}$, with $i\in \mathcal{I}\cup\mathcal{J}$, and $\mathcal{J}' = \emptyset$,
by the additivity property of the coincidence degree and hypothesis \eqref{deg=0}, we have
\begin{equation*}
\begin{aligned}
        D_{L}(L-N_{\lambda,\mu},\Lambda^{\mathcal{I}',\mathcal{J}'})
      & = D_{L}(L-N_{\lambda,\mu},\Lambda^{\{i\},\emptyset})
\\    & = D_{L}(L-N_{\lambda,\mu},\Omega^{\{i\},\emptyset}\setminus \Lambda^{\emptyset,\emptyset})
\\    & = D_{L}(L-N_{\lambda,\mu},\Omega^{\{i\},\emptyset})
          - D_{L}(L-N_{\lambda,\mu},\Lambda^{\emptyset,\emptyset})
\\  & = 0-1= -1 =(-1)^{\#\mathcal{I}'}.
\end{aligned}
\end{equation*}
There are no other possible choices of $\mathcal{I}'$ and $\mathcal{J}'$
with $\#\mathcal{I}'+(m+1)\#\mathcal{J}'\leq 1$ (since $m \geq 1$).

\smallskip
\noindent
\textit{Verification of ${\mathscr{P}}(k-1)\Rightarrow{\mathscr{P}}(k)$, for $1\leq k\leq n$. }
Assuming the validity of ${\mathscr{P}}(k-1)$ we have that the formula is true
for every ${\mathcal{I}'}\subseteq\mathcal{I}\cup\mathcal{J}$ and for every ${\mathcal{J}'}\subseteq{\mathcal{J}}$ such that
$\mathcal{I}'\cap\mathcal{J}'=\emptyset$ and $\#\mathcal{I}'+(m+1)\#\mathcal{J}'\leq k-1$.
Therefore, in order to prove ${\mathscr{P}}(k)$, we have only to check that the formula is true
for any possible choice of ${\mathcal{I}'}\subseteq\mathcal{I}\cup\mathcal{J}$ and ${\mathcal{J}'}\subseteq{\mathcal{J}}$
with $\mathcal{I}'\cap\mathcal{J}'=\emptyset$ and such that
\begin{equation}\label{eek}
\#\mathcal{I}'+(m+1)\#\mathcal{J}'= k.
\end{equation}

We distinguish two cases: either $\mathcal{I}'=\emptyset$ or $\mathcal{I}'\neq\emptyset$.
As a first instance, let $\mathcal{I}'=\emptyset$ and, in view of \eqref{eek},
suppose $\mathcal{J}'\neq\emptyset$ and $\#\mathcal{J}'=k/(m+1)$.
By formula \eqref{eq-union}, $\Omega^{\emptyset,\mathcal{J}'}$ can be written as the disjoint union
\begin{equation*}
\Omega^{\emptyset,\mathcal{J}'}
= \bigcup_{\substack{\mathcal{L}\subseteq\mathcal{J}' \\ \mathcal{K}\subseteq\mathcal{J}' \\ \mathcal{L}\cap\mathcal{K}=\emptyset}}
\Lambda^{\mathcal{L},\mathcal{K}}
= \Lambda^{\emptyset,\mathcal{J}'} \cup
\bigcup_{\substack{\mathcal{L}\subseteq\mathcal{J}' \\ \mathcal{K}\subsetneq\mathcal{J}' \\ \mathcal{L}\cap\mathcal{K}=\emptyset}}
\Lambda^{\mathcal{L},\mathcal{K}}.
\end{equation*}
Since $\#\mathcal{L}+(m+1)\#\mathcal{K}\leq k-1$ if $\mathcal{K}\subsetneq\mathcal{J}'$, by \eqref{deg=1}
and by the inductive hypothesis,
we obtain
\begin{equation*}
\begin{aligned}
D_{L}(L-N_{\lambda,\mu},\Lambda^{\emptyset,\mathcal{J}'})
& = D_{L}(L-N_{\lambda,\mu},\Omega^{\emptyset,\mathcal{J}'})
        - \sum_{\substack{\mathcal{L}\subseteq\mathcal{J}' \\ \mathcal{K}\subsetneq\mathcal{J}' \\ \mathcal{L}\cap\mathcal{K}=\emptyset}}
        D_{L}(L-N_{\lambda,\mu},\Lambda^{\mathcal{L},\mathcal{K}})
\\  & = 1 - \sum_{\substack{\mathcal{L}\subseteq\mathcal{J}' \\ \mathcal{K}\subsetneq\mathcal{J}'
\\ \mathcal{L}\cap\mathcal{K}=\emptyset}} (-1)^{\#\mathcal{L}}.
\end{aligned}
\end{equation*}
Now we observe that
\begin{equation*}
\sum_{\substack{\mathcal{L}\subseteq\mathcal{J}' \\ \mathcal{K}\subsetneq\mathcal{J}' \\ \mathcal{L}\cap\mathcal{K}=\emptyset}} (-1)^{\#\mathcal{L}} =
\sum_{\mathcal{K}\subsetneq\mathcal{J}'} \sum_{\mathcal{L}\subseteq\mathcal{J}'\setminus\mathcal{K}} (-1)^{\#\mathcal{L}} = 0,
\end{equation*}
due to the fact that in a finite set there are so many subsets of even cardinality how many subsets of odd cardinality.
Thus we conclude that
\begin{equation*}
D_{L}(L-N_{\lambda,\mu},\Lambda^{\emptyset,\mathcal{J}'}) = 1 = (-1)^{\#\mathcal{I}'}.
\end{equation*}
As a second instance, let $\mathcal{I}'\neq\emptyset$. Using \eqref{eq-union}, we can write $\Omega^{\mathcal{I}',\mathcal{J}'}$ as the disjoint union
\begin{equation*}
\Omega^{\mathcal{I}',\mathcal{J}'}
=\bigcup_{\substack{\mathcal{L}\subseteq\mathcal{I}'\cup\mathcal{J}' \\ \mathcal{K}\subseteq\mathcal{J}' \\ \mathcal{L}\cap\mathcal{K}=\emptyset}}
\Lambda^{\mathcal{L},\mathcal{K}}
= \Lambda^{\mathcal{I}',\mathcal{J}'} \cup \bigcup_{\substack{\mathcal{L}\subseteq\mathcal{I}'\cup\mathcal{J}' \\ \mathcal{K}\subseteq\mathcal{J}'
        \\ \mathcal{L}\cap\mathcal{K}=\emptyset  \\ (\mathcal{L},\mathcal{K})\neq(\mathcal{I}',\mathcal{J}')}} \Lambda^{\mathcal{L},\mathcal{K}}.
\end{equation*}
Since $\#\mathcal{L}+(m+1)\#\mathcal{K}\leq k-1$, if $\mathcal{K}\subsetneq\mathcal{J}'$ or
if $\mathcal{K}=\mathcal{J}'$ and $\mathcal{L}\subsetneq\mathcal{I}'$, by \eqref{deg=0} and by the inductive hypothesis, we obtain
\begin{equation*}
\begin{aligned}
    & D_{L}(L-N_{\lambda,\mu},\Lambda^{\mathcal{I}',\mathcal{J}'}) =
\\  & = D_{L}(L-N_{\lambda,\mu},\Omega^{\mathcal{I}',\mathcal{J}'}) -
       \sum_{\substack{\mathcal{L}\subseteq\mathcal{I}'\cup\mathcal{J}' \\ \mathcal{K}\subseteq\mathcal{J}'
        \\ \mathcal{L}\cap\mathcal{K}=\emptyset  \\ (\mathcal{L},\mathcal{K})\neq(\mathcal{I}',\mathcal{J}')}}
       D_{L}(L-N_{\lambda,\mu},\Lambda^{\mathcal{L},\mathcal{K}})
\\  & = 0 -\sum_{\substack{\mathcal{L}\subseteq\mathcal{I}'\cup\mathcal{J}' \\ \mathcal{K}\subseteq\mathcal{J}'
           \\ \mathcal{L}\cap\mathcal{K}=\emptyset  \\ (\mathcal{L},\mathcal{K})\neq(\mathcal{I}',\mathcal{J}')}}
            (-1)^{\#\mathcal{L}}
      = (-1)^{\#\mathcal{I}'} - \sum_{\substack{\mathcal{L}\subseteq\mathcal{I}'\cup\mathcal{J}' \\ \mathcal{K}\subseteq\mathcal{J}' \\ \mathcal{L}\cap\mathcal{K}=\emptyset}}
            (-1)^{\#\mathcal{L}}
      = (-1)^{\#\mathcal{I}'},
\end{aligned}
\end{equation*}
observing, as above, that
\begin{equation*}
\sum_{\substack{\mathcal{L}\subseteq\mathcal{I}'\cup\mathcal{J}' \\ \mathcal{K}\subseteq\mathcal{J}' \\ \mathcal{L}\cap\mathcal{K}=\emptyset}} (-1)^{\#\mathcal{L}} =
\sum_{\mathcal{K}\subseteq\mathcal{J}'} \sum_{\mathcal{L}\subseteq\mathcal{I}'\cup (\mathcal{J}'\setminus\mathcal{K})} (-1)^{\#\mathcal{L}} = 0.
\end{equation*}
Then $\mathscr{P}(k)$ is proved and the lemma follows.
\end{proof}

Now, since \eqref{deg=-1} is exactly formula \eqref{eq-3.2}, in order
to complete the proof of Theorem~\ref{deg-Lambda} we have only to check that
the degrees are well defined and assumptions
\eqref{deg=1} and \eqref{deg=0} in the above combinatorial lemma are satisfied.
All these requests are obviously guaranteed by the discussion in Section~\ref{section-3.1} and by formula \eqref{eq-3.1}.
Then Lemma~\ref{lemma-1} applies and this completes the proof of Theorem~\ref{deg-Lambda}.

\section*{Acknowledgments}

Work performed under the auspicies of the Grup\-po Na\-zio\-na\-le per l'Anali\-si Ma\-te\-ma\-ti\-ca, la Pro\-ba\-bi\-li\-t\`{a} e le lo\-ro
Appli\-ca\-zio\-ni (GNAMPA) of the Isti\-tu\-to Na\-zio\-na\-le di Al\-ta Ma\-te\-ma\-ti\-ca (INdAM).
Guglielmo Feltrin and Fabio Zanolin are partially supported by the GNAMPA Project 2015
``Problemi al contorno associati ad alcune classi di equazioni differenziali non lineari''.
Alberto Boscaggin is partially supported by the GNAMPA Project 2015 ``Equazioni Differenziali Ordinarie sulla retta reale'' and by the
project ERC Advanced Grant 2013 n.~339958 ``Complex Patterns for Strongly Interacting Dynamical Systems - COMPAT''.

\bibliographystyle{elsart-num-sort}
\bibliography{BFZ_biblio}

\bigskip
\begin{flushleft}

{\small{\it Preprint}}

{\small{\it December 2015}}

\end{flushleft}

\end{document}